\documentclass[1 [leqno,11pt]{amsart}
\usepackage{amssymb, amsmath,latexsym,amsfonts,amsbsy, amsthm,mathtools,graphicx,CJKutf8,CJKnumb,CJKulem,color}
\usepackage{float}
\usepackage{hyperref}

\numberwithin{equation}{section}

\allowdisplaybreaks

\let\e=\varepsilon

\let\la=\lambda

\let\f=\frac

\let\om=\omega

\let\Om=\Omega

\let\na=\nabla

\let\pa=\partial

\def\R{\mathbb R}
\def\T{\mathbb T}
\def\Z{\mathbb Z}

\newcommand{\beq}{\begin{equation}}
\newcommand{\eeq}{\end{equation}}
\newcommand{\ben}{\begin{eqnarray}}
\newcommand{\een}{\end{eqnarray}}
\newcommand{\beno}{\begin{eqnarray*}}
\newcommand{\eeno}{\end{eqnarray*}}

\newtheorem{theorem}{Theorem}[section]

\newtheorem{lemma}[theorem]{Lemma}
\newtheorem{proposition}[theorem]{Proposition}
\newtheorem{corollary}[theorem]{Corollary}

\begin{document}
\begin{CJK*}{UTF8}{gkai}
\title[Transition threshold for the 2-D Couette flow]{Transition threshold for the 2-D Couette flow in a finite channel}

\author{Qi Chen}
\address{School of Mathematical Science, Peking University, 100871, Beijing, P. R. China}
\email{chenqi940224@gmail.com}

\author{Te Li}
\address{School of Mathematical Science, Peking University, 100871, Beijing, P. R. China}
\email{little329@163.com}

\author{Dongyi Wei}
\address{School of Mathematical Science, Peking University, 100871, Beijing, P. R. China}
\email{jnwdyi@163.com}

\author{Zhifei Zhang}
\address{School of Mathematical Science, Peking University, 100871, Beijing, P. R. China}
\email{zfzhang@math.pku.edu.cn}

\date{\today}

\maketitle

\begin{abstract}
In this paper, we study the transition threshold problem for the 2-D Navier-Stokes equations around the Couette flow $(y,0)$ at large Reynolds number $Re$ in a finite channel. We develop a systematic method to establish the resolvent estimates of the linearized operator and space-time estimates of the linearized Navier-Stokes equations. In particular, three kinds of important effects: enhanced dissipation, inviscid damping and boundary layer, are integrated into the
space-time estimates in a sharp form. As an application, we prove that if the initial velocity $v_0$ satisfies $\|v_0-(y, 0)\|_{H^2}\le cRe^{-\f12}$ for some small $c$ independent of $Re$, then the solution of the 2-D Navier-Stokes equations remains within $O(Re^{-\f12})$ of the Couette flow for any time.
\end{abstract}

\section{Introduction}

Since Reynolds's famous experiment \cite{Rey}, the hydrodynamics stability at high Reynolds number has been an important field in fluid mechanics \cite{Sch, Yag}, which is mainly concerned with how the laminar flows become unstable and transition to turbulence. Theoretical analysis shows that some laminar flows such as plane Couette flow and pipe Poiseuille flow are linearly stable for any Reynolds number \cite{Rom, DR}. However, the experiments show that they could be unstable and transition to turbulence for small but finite perturbations at high Reynolds number \cite{TA, DH}. In addition, some laminar flows such as plane Poiseuille flow become turbulent at much lower Reynolds number than the critical Reynolds number of linear instability. The resolution of these paradoxes is a long-standing problem in fluid mechanics.

There are many attempts to understand these paradoxes(see \cite{Cha} and references therein). One resolution going back to Kelvin \cite{Kel} is that the basin of attraction of the laminar flow shrinks as $Re\to \infty$ so that the flow could become nonlinearly unstable for small but finite perturbations. Then an important question firstly proposed by Trefethen et al. \cite{TTR}(see also \cite{BGM}) is that \smallskip

{\it
Given a norm $\|\cdot\|_X$, determine a $\beta=\beta(X)$ so that
\beno
&&\|u_0\|_X\le Re^{-\beta}\Longrightarrow  {stability},\\
&&\|u_0\|_X\gg Re^{-\beta}\Longrightarrow  {instability}.
\eeno
}
The exponent $\beta$ is referred to as the transition threshold in the applied literature. \smallskip

The goal of this paper is to study the transition threshold for the 2-D incompressible Navier-Stokes equations in a finite channel $\Omega=\big\{(x,y): x\in \T, y\in I=(-1,1)\big\}$:
\begin{align}
\left\{
\begin{aligned}
&\partial_t v-\nu\Delta v+v\cdot\nabla v+\na P=0,\\
&\nabla\cdot v=0,\\
&v(0,x,y)=v_0(x,y),
\end{aligned}
\right.\label{eq:NS}
\end{align}
where $\nu\sim Re^{-1}$ is the viscosity coefficient, $v(t,x,y)=(v^1,v^2)$ is the velocity, $P(t,x,y)$ is the pressure.

For the 2-D fluid, nonlinear effect is weak. More importantly, the vorticity $\om=\pa_yv^1-\pa_xv^2$ has the following beautiful structure:
\begin{align}\label{eq:NS-vorticity}
\partial_t\om-\nu\Delta \om+v\cdot\nabla \om=0.
\end{align}
Compared with strong nonlinear effect in 3-D(especially, lift-up effect), it seems to mean that the 2-D Navier-Stokes equations are not a suitable model describing the transition to turbulence. However, the boundary layer effect is still very strong in 2-D for large Reynolds number if we impose non-slip boundary condition for the velocity. This effect could lead to the instability of the flows. So, the study of the 2-D fluid is very interesting and should be an important step toward understanding
the stability of the 3-D fluid in the presence of the physical boundary.

We will first study the stability of the Couette flow $U=(y,0)$, which is a solution of \eqref{eq:NS} and linearly stable for any Reynolds number. Let $u=v-U$ be the perturbation of the velocity, which satisfies
\begin{align}\label{eq:NS-p}
\left\{
\begin{aligned}
&\partial_t u-\nu\Delta u+y\partial_x u+(u_2,0)+u\cdot\nabla u+\na P=0,\\
&\nabla\cdot u=0,\\
&u(t,x,\pm 1)=0,\\
&u(0,x,y)=u_0(x,y).
\end{aligned}
\right.
\end{align}
Here we impose non-slip boundary condition for the perturbation $u$.


There are a lot of works \cite{DBL, LHR, RSB, Yag}  in applied mathematics and physics devoted to estimating transition threshold for various flows such as Couette flow and Poiseuille flow. Recently, Bedrossian, Germain, Masmoudi  et  al.  made an important progress on the stability threshold problem for the Couette flow in a series of works \cite{BGM1, BGM2, BGM3, BMV, BWV}. Roughly speaking, their results could be summarized as follows.\smallskip

When $\Om=\T\times \R\times \T$,

\begin{itemize}

\item if the perturbation is in Gevrey class, then $\beta=1$ \cite{BGM1,BGM2};

\item if the perturbation is in Sobolev space, then $\beta\le \f32$ \cite{BGM3}.

\end{itemize}

When $\Om=\mathbb{T}\times \R$,

\begin{itemize}

\item if the perturbation is in Gevrey class, then $\beta=0$ \cite{BMV};

\item if the perturbation is in Sobolev space, then $\beta\le \f12$\cite{BWV}.

\end{itemize}

The results in \cite{BGM1, BGM3} seem to mean that the regularity of the initial data has an important effect on the transition threshold. In a recent work \cite{WZ},  Wei and Zhang  proved that the transition threshold $\beta\le 1$ still holds in Sobolev regularity for the Couette flow in $\Om=\T\times \R\times \T$. This result confirms  the transition threshold conjecture proposed in \cite{TTR}(see also \cite{Cha}).  On the other hand, the threshold is much smaller in 2-D due to the absence of lift-up effect.

\smallskip


Previous results show that three kinds of linear effects(including enhanced dissipation, inviscid damping, 3-D lift-up) and nonlinear structure play a key role in determining the transition threshold in the absence of the boundary. In this paper, we would like to understand how various effects, especially the boundary layer effect, influence the transition threshold in the presence of the boundary. There are some mathematical papers \cite{Rom, KLH, LK,Sil} devoting to nonlinear stability of the Couette flow in a channel, where they gave a rough bound of $\beta$, for example, $\beta\le 3$ in 2-D and $\beta\le 4$ in 3-D.

To study nonlinear stability, the key step is to establish the space-times estimates for the linearized Navier-Stokes equations around the Couette flow, which takes as follows
\begin{align*}
\left\{
\begin{aligned}
&\partial_t u-\nu\Delta u+y\partial_x u+(u_2,0)+\na P=0,\\
&\nabla\cdot u=0.
\end{aligned}
\right.
\end{align*}
In terms of the vorticity $w=\pa_yu^1-\pa_xu^2,$ it takes
\beno
\pa_t w-\nu\Delta w+y\pa_xw=0.
\eeno
When $\Om=\T\times \R$, the space-time estimates could be established
by using the Fourier transform in $(x,y)$. When $\Om=\T\times I$, we need to use the resolvent estimates of the linearized operator. Let us emphasize that the spectrum of the operator $A$ is not enough to determine the behaviour of semigroup $e^{tA}$, when $A$ is a non normal operator \cite{Tre}.  In \cite{KLH, TTR, LK-SIAM,AK,Sil-JDE, Sil-SIAM}, the authors established some resolvent estimates in some regimes of parameters $Re$ and wave number by using the rigorous analysis combined with numerical computations. However, based on these estimates, one can only establish a rough bound of transition threshold.

In this paper, we first develop a systematic method  to establish some sharp resolvent estimates for the linearized operator under the Navier-slip boundary condition and non-slip boundary condition respectively. In particular, for non-slip boundary condition, we use many deep properties of the Airy function to give precise $L^p$ bounds on the solutions of homogeneous Orr-Sommerfeld equation. Moreover, our resolvent estimates show that the linearized operator
has a much wider spectrum gap $O({\nu^\f13 k^\f23})$(usually $O(\nu)$)  when the wave number $k\neq 0$, which is related to the enhanced dissipation induced by mixing due to the Couette flow. See \cite{BW, LWZ-K,IMM, WZZ3, LX} for the enhanced dissipation induced by the Kolmogorov flow and \cite{CKR} for more general situation.

With the resolvent estimates at hand, a standard method for semigroup estimate is to use the Dunford integral and choose a suitable contour including the spectrum of the linearized operator. However, the semigroup estimate obtained in this way is not enough to obtain a sharp threshold. In section 5, we
develop a complete new method to establish the space-time estimates of the solution of the linearized Navier-Stokes equations. Our space-time estimates include $\|u\|_{L^2_tL^2_y}$ due to inviscid damping, some estimates due to enhanced dissipation, and $L^\infty$ estimate of the velocity in sprit of
maximum principle. Recently, the invisicd damping as an analogue of Landau damping has been well understood at least at the linear level \cite{BM, Zil, WZZ1, WZZ2, WZZ3, BCV}.

We believe that the method we develop to establish the resolvent estimates and space-time estimates could be used to other related  problems such as
the transition threshold for general flows and the stability analysis of boundary layer.

As an application of the space-time estimates, we prove the following nonlinear stability of the Couette flow.  To state our result, we define
\beno
P_0f=\overline{f}(t,y)=\int_{\T}f(t,x,y)dx, \quad f_{\neq}=f-\overline{f}=\sum_{k\neq 0}f_k(t,y)e^{ikx}.
\eeno

Our result is stated as follows.

\begin{theorem}\label{thm:stability}
Assume that $u_0\in H^1_0(\Om)\cap H^2(\Om)$ with $\text{div}\,u_0=0$. There exist constants $\nu_0$ and $c, C>0$ independent of $\nu$ so that if $\|u_0\|_{H^2}\le c\nu^\f12$, ${0<\nu\leq \nu_0}$,
then the solution $u$ of the system \eqref{eq:NS-p} is global in time and satisfies the following stability estimate:
\beno
\sum_{k\in \Z}E_k\le Cc\nu^\f12,
\eeno
where $E_0=\|\overline{w}\|_{L^{\infty}L^2}$ and for $k\neq 0$,
\begin{align*}
E_k=&\|(1-|y|)^{\f12}w_k\|_{L^{\infty}L^2}+|k|\|u_k\|_{L^2L^2}+|k|^\f12\|u_k\|_{L^{\infty}L^{\infty}}+(\nu k^2)^{\f14}\|w_k\|_{L^2L^2}.
\end{align*}
\end{theorem}

To our surprise, the threshold is the same as one obtained by \cite{BWV} for the 2-D Navier-Stokes equation in $\Om=\T\times \R$. This means that the boundary layer effect does not give rise to strong instability for the Couette flow due to weak nonlinear effect in 2-D fluid. However, in 3-D case, nonlinear structure of the system is more complex so that it is very hard to analyze how stabilizing mechanism (enhanced dissipation and inviscid damping) and destabilizing mechanism(boundary layer and lift-up) influence different modes of the solution and different components of the velocity, and complex interactions among them.
We will leave the transition threshold problem in 3-D to our future work.

\section{The linearized equation and key ideas of the proof}

\subsection{The linearized equation}
The linearized Navier-Stokes equations around the Couette flow, which takes as follows
\begin{align}\label{eq:LNS-u}
\left\{
\begin{aligned}
&\partial_t u-\nu\Delta u+y\partial_x u+(u_2,0)+\na P=0,\\
&\nabla\cdot u=0.
\end{aligned}
\right.
\end{align}
Due to the unknown pressure, it is not easy to handle the velocity equation directly.
Thus, we introduce two formulations in terms of the vorticity $w$ and the stream function $\Phi$ respectively,
which are defined by
\beno
w=\pa_yu^1-\pa_xu^2,\quad u=\nabla^{\perp}\Phi=\big(\pa_y\Phi,-\pa_x\Phi\big).
\eeno
Then the vorticity formulation of \eqref{eq:LNS-u} takes
\ben\label{eq:LNS-w}
\partial_tw-\nu\Delta w+y\partial_xw=0.
\een
Thanks to $\Delta \Phi=w$, we have
\ben\label{eq:LNS-stream}
\partial_t\Delta \Phi-\nu\Delta^2 \Phi+y\partial_x\Delta \Phi=0.
\een
Due to the beautiful structure \eqref{eq:LNS-w}, the 2-D Navier-Stokes equations have no lift-up effect.

Taking the Fourier transform in $x\in \mathbb{T}$, we get
\beno
w(t,x,y)=\sum\limits_{k\in\mathbb{Z}}\widehat{w}_k(t,y)e^{ikx}=\sum\limits_{k\in\mathbb{Z}}e^{ikx}(\partial_y^2-k^2)\widehat{\Phi}_k(t,y).
\eeno
Then we have
\begin{align}\label{eq:LNS-w-F}
\partial_t\widehat{w}_k(t,y)+L_{k}\widehat{w}_k(t,y)=0,
\end{align}
where $L_k=\nu(k^2-\partial_y^2)+iky$.

For the linearized equations, we will study two kinds of boundary conditions.
The first one is the non-slip boundary condition:
\ben\label{eq:bc-nonslip}
u(t,x,\pm 1)=0.
\een
In this case, the stream function $\widehat{\Phi}_k(t,y), k\neq 0$ takes on the boundary:
\ben\label{eq:bc-nonslip-stream}
\widehat{\Phi}_k(t,\pm 1)=\widehat{\Phi}'_k(t,\pm 1)=0.
\een
The second one is the Navier-slip boundary condition:
\ben\label{eq:bc-navier}
u^2(t,x,\pm 1)=0,\quad \pa_yu^1(t,x,\pm 1)=0.
\een
So, $w(t,x,\pm 1)=0$ and for $k\neq 0$
\ben\label{eq:bc-navier-stream}
\widehat{\Phi}_k(t,\pm 1)=\widehat{\Phi}''_k(t,\pm 1)=0.
\een

Standard method for the stability is to make the eigenvalue analysis for the linearized equation.
That is, we seek the solution of the form
\beno
\widehat{w}_k(t,y)=\om_k(y)e^{-ikt\lambda},\quad \widehat{\Phi}_k(t,y)=\varphi_k(y)e^{-ik\lambda t}.
\eeno
Then $w_k(y)$ and $\varphi_k(y)$ satisfy the following Orr-Sommerfeld(OS) equation
\ben
&&-\nu(\partial_y^2-k^2)\om_k+ik(y-\lambda)\om_k=0,\label{eq:eigen-w}\\
&&-\nu(\partial_y^2-k^2)^2\varphi_k+ik(y-\lambda)(\partial_y^2-k^2)\varphi_k=0.\label{eq:eigen-stream}
\een
If there exists a nontrivial solution of \eqref{eq:eigen-stream} with the boundary condition \eqref{eq:bc-nonslip-stream} or
\eqref{eq:bc-navier-stream} for $\lambda, k>0$ with $\rm{Im}\lambda>0$, we say that the flow is linearly unstable.

For Navier-slip boundary condition, it is easy to see that the Couette flow is linearly stable for any $\nu>0$. Indeed, it follows from \eqref{eq:eigen-w} and $\om_k(\pm1)=0$
that
\beno
\int_{-1}^1\nu\big(|\pa_y\om_k|^2+k^2|\om_k|^2\big)+k{\rm Im}\lambda|\om_k|^2 dy=0,
\eeno
which implies that $\om_k=0$ if $k\rm{Im}\lambda>0$.

For non-slip boundary condition, Romanov \cite{Rom} proved that the Couette flow is also linearly stable for any $\nu>0$. In this case, the proof is highly nontrivial. In fact, Romanov studied the Navier-Stokes equations in an infinite channel $\R\times I$, and proved that
the eigenvalues must lie in a region with ${\rm Im}\lambda<-c\nu$ for some $c>0$.
In a finite channel $\T\times I$, we proved that the eigenvalues must lie in a region with ${\rm Im}\lambda<-ck^{\f 23}\nu^\f13$
for some $c>0$ if the wave number $|k|\ge 1$, which is related to the enhanced dissipation induced by the Couette flow. Thus, it is very interesting
to investigate the long wave effect on nonlinear stability. Indeed, the instability of many plane shear flows is due to the long wave. For example,
the unstable wave numbers $k$ lie in a band $O(Re^{-\f17})\le k\le O(Re^{-\f1 {11}})$ for the plane Poiseuille flow \cite{DR}. So,
it is also linearly stable for any $\nu>0$ in a finite channel.
\smallskip

\subsection{Key ideas and structure of the paper}

In section 3, we study the resolvent estimates of the linearized operator under the Navier-slip boundary condition:
\begin{align*}
-\nu(\partial_y^2-k^2)w+ik(y-\lambda)w=F,\quad w(\pm1)=0,
\end{align*}
The proof used the idea of multiplier introduced in \cite{LWZ-O, LWZ-K}.

In section 4, we study the resolvent estimates of the linearized operator under the non-slip boundary condition:
\begin{align*}
\left\{
\begin{aligned}
&-\nu(\partial_y^2-k^2)^2\varphi+ik(y-\lambda)(\partial_y^2-k^2)\varphi=F,\\
&\varphi(\pm1)=0,\quad\varphi'(\pm1)=0,
\end{aligned}
\right.
\end{align*}
and $w=(\pa_y^2-k^2)\varphi$. First of all, we decompose $w$ into the solution $w_{Na}$ of the inhomogeneous OS equation with the Navier-slip boundary condition and the solutions $w_1,w_2$ of the homogeneous OS equation, i.e.,
\beno
w=w_{Na}+c_1w_1+c_2w_2.
\eeno
One key point is that we derive the explicit formula of the coefficients $c_1,c_2$ and give very precise estimates based on the resolvent estimates of
the linearized operator under the Navier-slip boundary condition, especially, a weak type resolvent estimate. Another
key point is that we derive the sharp $L^p$ bounds and weighted $L^2$ bounds on $w_1,w_2$, which will be proved in section 5.
To this end, we need to use many deep estimates of the Airy function derived in section 8.
In fact, some estimates have been implied in Romanov's beautiful paper \cite{Rom}.

In section 6, we derive the space-time estimates of the linearized Navier-Stokes equations. To apply them to nonlinear problem,
we consider the following inhomogeneous problem:
\begin{align*}
\partial_t\omega+L_k\omega=-ikf_1-\partial_yf_2,\quad\omega|_{t=0}=\omega_0(y),
\end{align*}
where $\omega=(\partial_y^2-k^2)\varphi$ with $\varphi(\pm1)=\varphi'(\pm1)=0$.
Usually, the space-time estimates can be obtained by using the Dunford integral and resolvent estimates.
Indeed, for Navier-slip boundary condition, we can obtain the sharp bound of semigroup by using Gearhart-Pr\"uss theorem, since the linearized
operator $L_k$ is accretive in this case. However, for non-slip boundary condition, one can only obtain a rough bound of semigroup in this way.
Using this rough bound, we can improve previous results on the transition threshold. However, it is far from the bound obtained in Theorem \ref{thm:stability}.

Our key idea is that we do not use the resolvent estimates directly, and instead use the ideas of establishing the resolvent estimates
in order to know how the boundary layer influences the space-time estimates more precisely. To this end, we first decompose the problem into inhomogeneous problem and homogeneous problem. After taking the Laplace transform for the inhomogeneous problem, we use the ideas of establishing the resolvent estimates to prove $L^2_tL^2$ estimate, and then use energy method combined with the precise estimates for $c_1, c_2$ and $w_1,w_2$ to obtain $L^\infty_tL^2$ estimate. For the solution $\om_H$ of the homogeneous problem, we split it into three parts:
\beno
\omega_{H}=\omega_{H}^{(1)}+\omega_{H}^{(2)}+\omega_{H}^{(3)},
\eeno
where $\omega_{H}^{(1)}(t,k,y)=e^{-(\nu k^2)^{1/3}t-itky}\omega_0(k,y)$, and $\omega_{H}^{(2)}$ solves
\begin{align*}
&(\partial_t-\nu(\partial_y^2-k^2)+iky)\omega_{H}^{(2)}=-(\nu k^2-(\nu k^2)^{\f13})\omega_{H}^{(1)}+\nu\partial_y^2\omega_{H}^{(1)},\\
&\omega_{H}^{(2)}|_{t=0}=0,\quad \langle\omega_{H}^{(2)},e^{\pm ky}\rangle=0,
\end{align*}
and $\omega_{H}^{(3)}$ solves
\begin{align*}
(\partial_t-\nu(\partial_y^2-k^2)+iky)\omega_{H}^{(3)}=0, \ \omega_{H}^{(3)}|_{t=0}=0,\ \langle\omega_{H}^{(3)}(t)+\omega_{H}^{(1)}(t),e^{\pm ky}\rangle=0.
\end{align*}
Now the estimates for $\omega_{H}^{(1)}$ are direct. For $\omega_{H}^{(2)}$, we can use the space-time estimates obtained for the inhomogeneous problem.
For $\omega_{H}^{(3)}$, we again need to use the $L^p$ estimates of $w_1,w_2$ and new weighted $L^2$ estimates for $c_1, c_2$.

In summary, our space-time estimates take as follows
\begin{align*}
&|k|\|u\|_{L^\infty L^\infty}^2+k^2\|u\|_{L^2L^2}^2+(\nu k^2)^{\f12}\|\omega\|_{L^2L^2}^2+\|(1-|y|)^{\f12}\omega\|_{L^{\infty}L^2}^2\\
&\leq C\big(\|\omega(0)\|_{L^2}^2+k^{-2}\|\partial_y\omega(0)\|_{L^2}^2\big)+C\big(\nu^{-\f12}|k|\|f_1\|_{L^2L^2}^2+\nu^{-1}\|f_2\|_{L^2L^2}^2\big).
\end{align*}
The estimate  $\|u\|_{L^2L^2}$ is due to the inviscid damping and plays an important role in this work(also in \cite{WZ}). The proof is relatively easier than the polynomial decay  established in \cite{WZZ2}. The $L^\infty$ estimate of the velocity is very surprising and takes in the homogeneous case($f_1=f_2=0$):
\beno
|k|\|u\|_{L^\infty L^\infty}^2\le C\big(\|\omega(0)\|_{L^2}^2+k^{-2}\|\partial_y\omega(0)\|_{L^2}^2\big)
\eeno
In some sense, this result means that maximum principle still holds for the linearized Navier-Stokes equation \eqref{eq:LNS-u}.
This is similar to Abe and Giga's breakthrough work on the analyticity of the Stokes semigroup in spaces of bounded functions \cite{AG}.

In section 7, we prove nonlinear stability by using the vorticity formulation and the space-time estimates. \smallskip

Throughout this paper, we always assume $\nu\in (0,1]$ and $|k|\ge 1$, and denote by $C$ a constant independent
of $\nu, k, \lambda$, which may be different from line to line.

\section{Resolvent estimates with Navier-slip boundary condition}
In this section, we study the resolvent estimates of the linearized operator under the Navier-slip boundary condition.
More precisely, we consider the vorticity equation
\begin{align}\label{eq:R-w-navier}
-\nu(\partial_y^2-k^2)w+ik(y-\lambda)w=F,\quad w(\pm1)=0,
\end{align}
and the stream function equation
\begin{align}
\label{eq:R-phi-navier}
\left\{
\begin{aligned}
&-\nu(\partial_y^2-k^2)^2\varphi+ik(y-\lambda)(\partial_y^2-k^2)\varphi=F,\\
&\varphi(\pm1)=0,\quad\varphi''(\pm1)=0.
\end{aligned}
\right.
\end{align}

\subsection{Resolvent estimate from $L^2$ to $H^2$}

First of all, we consider the case of $\lambda\in \R$.
\begin{proposition}\label{prop:R-navier-v1}
Let $w\in H^2(I)$ be a solution of \eqref{eq:R-w-navier} with $\lambda\in \R$ and $F\in L^2(I)$. Then it holds that
\begin{align*}
&\nu^{\f16}|k|^{\f43}\|u\|_{L^2}+{\nu^{\f16}|k|^{\f56}\|w\|_{L^1}}+\nu^{\f23}k^{\f13}\|w'\|_{L^2} +(\nu k^2)^{\f13}\|w\|_{L^2} +|k|\|(y-\lambda)w\|_{L^2}\leq C\|F\|_{L^2},
\end{align*}
where $u=(\pa_y\varphi, {-ik\varphi})$ and $(\pa_y^2-k^2)\varphi=0$ with $\varphi(\pm 1)=0$.
\end{proposition}
\begin{proof}
By integration by parts, we get
\beno
\langle F,w\rangle=\nu\|w'\|_{L^2}^2+\nu k^2\|w\|_{L^2}+ik\int_{-1}^1(y-\lambda)|w|^2dy.
\eeno
Taking the real part, we get
\beno
\nu\|w'\|_{L^2}^2\leq\|F\|_{L^2}\|w\|_{L^2}.
\eeno

We also get by integration by parts that
\begin{align*}
&\langle F,(y-\lambda)w\rangle=-\nu\int_{-1}^1w''(y-\lambda)\bar{w}dy+\nu k^2\int_{-1}^1(y-\lambda)|w|^2dy+ik\|(y-\lambda)w\|_{L^2}^2\\&=\nu\int_{-1}^1w'\bar{w}dy+\nu\int_{-1}^1(y-\lambda)|w'|^2dy+\nu k^2\int_{-1}^1(y-\lambda)|w|^2dy+ik\|(y-\lambda)w\|_{L^2}^2.
\end{align*}
Taking the imaginary part, we get
\begin{align*}
&|k|\|(y-\lambda)w\|_{L^2}^2\leq \|F\|_{L^2}\|(y-\lambda)w\|_{L^2}+\nu \|w'\|_{L^2}\|w\|_{L^2},\\
&\|(y-\lambda)w\|_{L^2}^2\leq |k|^{-2}\|F\|_{L^2}^2+2|k|^{-1}\nu\|w'\|_{L^2}\|w\|_{L^2}.
\end{align*}

Let $\delta=\nu^{\f13}|k|^{-\f13}$, $E=(-1,1)\cap(\lambda-\delta,\lambda+\delta)$, $E^c=(-1,1)\setminus(\lambda-\delta,\lambda+\delta)$. Then we have
\begin{align*}
\|w\|_{L^2}^2
=&\|w\|_{L^2(E^c)}^2+\|w\|_{L^2(E)}^2\leq \delta^{-2}\|(y-\lambda)w\|_{L^2}^2+2\delta\|w\|_{L^{\infty}}^2\\ \leq&\delta^{-2}|k|^{-2}\|F\|_{L^2}^2+2\delta^{-2}|k|^{-1}\nu\|w'\|_{L^2}\|w\|_{L^2}+2\delta\|w'\|_{L^2}\|w\|_{L^2}\\
=&(\nu k^2)^{-\f23}\|F\|_{L^2}^2+4\delta\|w'\|_{L^2}\|w\|_{L^2}\leq (\nu k^2)^{-\f23}\|F\|_{L^2}^2+4\delta\nu^{-\f12}\|F\|_{L^2}^{\f12}\|w\|_{L^2}^{\f32},
\end{align*}
which implies
\begin{align*}
&\|w\|_{L^2}^2
\leq C\big((\nu k^2)^{-\f23}+\delta^4\nu^{-2}\big)\|F\|_{L^2}^2\leq C(\nu k^2)^{-\f23}\|F\|_{L^2}^2,
\end{align*}
and thus,
\begin{align*}
\nu\|w'\|_{L^2}^2\leq\|F\|_{L^2}\|w\|_{L^2}\lesssim(\nu k^2)^{-\f13}\|F\|_{L^2}^2.
\end{align*}

For $\|(y-\lambda)w\|_{L^2}$, we have
\begin{align*}
\|(y-\lambda)w\|_{L^2}^2\leq& |k|^{-1}\|F\|_{L^2}^2+2|k|^{-1}\nu\|w'\|_{L^2}\|w\|_{L^2}\\ \leq& |k|^{-2}\|F\|_{L^2}^2+C|k|^{-1}\nu(\nu^{-\f23}k^{-\f13}\|F\|_{L^2})(\nu k^2)^{-\f13}\|F\|_{L^2}\\
\leq& C|k|^{-2}\|F\|_{L^2}^2.
\end{align*}
Notice that
\begin{align*}
\|w\|_{L^1}&=\int_E|w|dy+\int_{(-1,1)\setminus E}|w|dy\\&\leq\delta^{\f12}\|w\|_{L^2}+\Big(\int _{(-1,1)\setminus E}\f{1}{(y-\lambda)^2}dy\Big)^{\f12}\|(y-\lambda)w\|_{L^2},
\end{align*}
which gives
\begin{align*}
\|w\|_{L^1}\lesssim\delta^{\f12} (\nu k^2)^{-\f13}\|F\|_{L^2}+\delta^{-\f12}|k|^{-1}\|F\|_{L^2}\lesssim\nu^{-\f16}|k|^{-\f56}\|F\|_{L^2},
\end{align*}
from which and Lemma \ref{lem:phi-w}, we infer that
\beno
 \|u\|_{L^2}^2\leq \|\varphi'\|_{L^2}^2+k^2\|\varphi\|_{L^2}^2\lesssim|k|^{-1}\|w\|_{L^1}^2.
\eeno

 Summing up, we conclude the proof.
\end{proof}

Proposition \ref{prop:R-navier-v1} implies that the resolvent set of the linearized operator $L_{k}$ contains the region
\begin{align*}
\Omega=\big\{\lambda=\lambda_r+i\lambda_i: \lambda_r\leq \epsilon\nu^{\f13}|k|^{\f23},\lambda_i\in\R\big\}
 \end{align*}
for some small $\epsilon\le C^{-1}$ with $C$ given by Proposition \ref{prop:R-navier-v1}. As a corollary of Proposition \ref{prop:R-navier-v1}, we can deduce the following resolvent estimate for
$\lambda\in \Gamma_{\epsilon}=\big\{\lambda=-i\epsilon \nu^{\f13}|k|^{\f23}+\lambda_r: \lambda_r\in \R\big\}$.
\begin{corollary}\label{cor:R-Navier-v1}
Let $w\in H^2(I)$ be a solution of \eqref{eq:R-w-navier} with $\lambda\in \Gamma_{\epsilon}$ and $F\in L^2(I)$. Then it holds that
\begin{align*}
&\nu^{\f16}|k|^{\f43}\|u\|_{L^2}+{\nu^{\f16}|k|^{\f56}\|w\|_{L^1}}+\nu^{\f23}k^{\f13}\|w'\|_{L^2}+(\nu k^2)^{\f13}\|w\|_{L^2}+|k|\|(y-\lambda)w\|_{L^2}\leq C\|{F}\|_{L^2}.
\end{align*}
\end{corollary}
\begin{proof} Let $\widetilde{F}={F}+\epsilon\nu^{\f13}|k|^{\f23}w$. Then by Proposition \ref{prop:R-navier-v1}, we get
\begin{align*}
&\|{F}\|_{L^2}\geq\|\widetilde{F}\|_{L^2}-\epsilon\nu^{\f13}|k|^{\f23}\|w\|_{L^2} \geq \|\widetilde{F}\|_{L^2}-C\epsilon\|\widetilde{F}\|_{L^2}=(1-C\epsilon) \|\widetilde{F}\|_{L^2}\\
&\quad\gtrsim\big(\nu^{\f16}|k|^{\f43}\|u\|_{L^2} +{\nu^{\f16}|k|^{\f56}\|w\|_{L^1}}+\nu^{\f23}k^{\f13}\|w'\|_{L^2}+(\nu k^2)^{\f13}\|w\|_{L^2}+|k|\|(y-\lambda)w\|_{L^2}\big),
\end{align*}
if we take $\epsilon$ so that $C\epsilon\le \f12$.
\end{proof}

\subsection{Resolvent estimate from $H^{-1}$ to $H^1$}
For this, we need to use the stream function formulation \eqref{eq:R-phi-navier} for $\lambda\in \Gamma_{\epsilon}=\big\{-i\epsilon \nu^{\f13}|k|^{\f23}+\lambda: \lambda \in \R\big\}$ with small $\epsilon$ as above. That is,
\begin{align}
\label{eq:R-navier-phi2}
\left\{
\begin{aligned}
&-\nu(\partial_y^2-k^2)^2\varphi+ik(y-\lambda)(\partial_y^2-k^2)\varphi- \epsilon\nu^{\f13}|k|^{\f23}(\partial_y^2-k^2)\varphi=F,\\
&\varphi(\pm1)=0,\quad \varphi''(\pm1)=0,
\end{aligned}
\right.
\end{align}
with $\lambda\in \R$.

\begin{proposition}\label{prop:R-navier-v2}
Let $\varphi\in H^3(I)$ be a solution of \eqref{eq:R-navier-phi2} with  $F\in H^{-1}(I)$. Then it holds that
\begin{align*}
 (\nu |k|^2)^{\f12}\|u\|_{L^2}+\nu\|w'\|_{L^2}+\nu^{\f23}|k|^{\f13}\|w\|_{L^2} \leq C\|{F}\|_{H^{-1}} .
\end{align*}
Here $u=({\partial_y\varphi,-ik\varphi})$ and $w=(\partial_y^2-k^2)\varphi$.
\end{proposition}

The proposition is a direct consequence of Lemma \ref{lem:R-navier-2} and Lemma \ref{lem:R-navier-4}.

\begin{lemma}\label{lem:R-navier-2}
Let $w$ {be} as in Proposition \ref{prop:R-navier-v2}.
It holds that
\begin{align*}
\nu\|w'\|_{L^2}+\nu^{\f23}|k|^{\f13}\|w\|_{L^2}\leq C\|{F} \|_{H^{-1}}.
\end{align*}
\end{lemma}

\begin{proof}
Let $\delta=\nu^{\f13}|k|^{-\f13}$. By integration by parts, we have
\begin{align}
\nonumber\|{F}\|_{H^{-1}}\|w\|_{H^1}&\geq|\langle {F},w\rangle|=|\langle -\nu(\partial_y^2-k^2)w+ik(y-\lambda)w- \epsilon\nu^{\f13}|k|^{\f23}w,w\rangle|\\
&\geq\nu\|w'\|_{L^2}^2+\nu k^2\|w\|_{L^2}^2- \epsilon\nu^{\f13}|k|^{\f23}\|w\|^2_{L^2}\\
&\geq\nu\|w\|_{H^1}^2-\epsilon\nu^{\f13}|k|^{\f23}\|w\|^2_{L^2},\nonumber
\end{align}
which gives
\begin{align*}
\nu\|w\|^2_{H^1}\leq\|{F}\|_{H^{-1}}\|w\|_{H^1} +\epsilon\nu^{\f13}|k|^{\f23}\|w\|^2_{L^2}\leq\f{\nu^{-1}}{2}\|{F}\|_{H^{-1}}^2 +\f{\nu}{2}\|w\|^2_{H^1}+\epsilon\nu^{\f13}|k|^{\f23}\|w\|^2_{L^2}.
\end{align*}
This shows that
\begin{align}\label{eq:w-H1-t1}
  \|w\|_{H^1}\leq\nu^{-1}\|{F}\|_{H^{-1}}+\sqrt{2\epsilon}\delta^{-1}\|w\|_{L^2}.
\end{align}


Let us introduce a cutoff function $\rho(y)$ as follows
\begin{align}\label{def:cutoff-rho}
\rho(y)=\left\{
  \begin{aligned}
&-1\quad\qquad\qquad\qquad y\in(-1,1)\cap(-1,\lambda-\delta),\\
&\sin\big(\f{\pi(y-\lambda)}{2\delta}\big)\qquad\quad y\in(-1,1)\cap(\lambda-\delta,\lambda+\delta),\\
&1\qquad\qquad\qquad\qquad y\in(-1,1)\cap(\lambda+\delta,1).
\end{aligned}
\right.
\end{align}
We get by integration by parts that
\begin{align*}
&\big|{\rm Im}\langle ik(y-\lambda)w-\nu(\partial_y^2-k^2)w - \epsilon\nu^{\f13}|k|^{\f23}w,\rho w\rangle \big|\\
&\geq \delta|k|\int_{(-1,1)\setminus(\lambda-\delta,\lambda+\delta)}|w|^2dy-\nu \Big|{\rm Im}\int_{-1}^{1}\partial^2_yw\overline{w}\rho dy \Big|\\
&\geq  \delta|k|\int_{(-1,1)\setminus(\lambda-\delta,\lambda+\delta)}|w|^2dy-\nu \Big|\int_{-1}^{1}w'\overline{w}\rho'dy\Big|\\
&\geq  \delta|k|\int_{(-1,1)\setminus(\lambda-\delta,\lambda+\delta)}|w|^2dy-\nu\|w\|_{L^\infty}\|w'\|_{L^2}\|\rho'\|_{L^2},
\end{align*}
which implies
\begin{align*}
\delta|k|\int_{(-1,1)\setminus(\lambda-\delta,\lambda+\delta)}|w|^2dy \lesssim& \|{F}\|_{H^{-1}}\|\rho w\|_{H^1}+\nu\delta^{-\f12}\|w'\|_{L^2}\|w\|_{L^{\infty}}.
\end{align*}
This along with $\|\rho w\|_{H^1}\lesssim\|w\|_{H^1}+\delta^{-1}\|w\|_{L^2}$ shows
\begin{align*}
\|w\|_{L^2((-1,1)\setminus(\lambda-\delta,\lambda+\delta))}^{{2}}
\lesssim\delta^{-1}|k|^{-1}\big(\|{F}\|_{H^{-1}}\|w\|_{H^1}+\delta^{-1}\|{F}\|_{H^{-1}}\|w\|_{L^2}+ \delta^{-\f12}\nu\|w'\|_{L^2}\|w\|_{L^{\infty}}\big),
\end{align*}
which implies
\begin{align}
\|w\|_{L^2}^2\lesssim\delta\|w\|_{L^{\infty}}^2+ \f{\|{F}\|_{H^{-1}}\|w\|_{H^1}}{|k|\delta}+\f{\|{F}\|_{H^{-1}}\|w\|_{L^2}}{|k|\delta^2}+ \f{\nu\|w'\|_{L^2}\|w\|_{L^{\infty}}}{|k|\delta^{\f32}}.
\end{align}
Thanks to $\|w\|_{L^{\infty}}^2\leq\|w\|_{L^2}\|w'\|_{L^2}$ and \eqref{eq:w-H1-t1}, we get
\begin{align*}
\|w\|_{L^2}^2\lesssim&\delta\|w\|_{H^1}\|w\|_{L^2}+\f{\|{F}\|_{H^{-1}}\|w\|_{H^1}}{|k|\delta}+\f{\|{F}\|_{H^{-1}}\|w\|_{L^2}}{|k|\delta^2}+ \f{\nu\|w\|^{\f32}_{H^1}\|w\|^{\f12}_{L^2}}{|k|\delta^{\f32}}\\
\lesssim&\f{\delta}{\nu}\|{F}\|_{H^{-1}}\|w\|_{L^2}+\f{\|{F}\|^2_{H^{-1}}}{|k|\delta\nu}+\f{ (1+\sqrt{\epsilon})\|{F}\|_{H^{-1}}\|w\|_{L^2}}{|k|\delta^2}+(\epsilon^{\f34}+\sqrt{\epsilon})\|w\|^2_{L^2}.
\end{align*}
Due to $\delta=\nu^{\f13}|k|^{-\f13}$ and $\epsilon$ small, we get by Young's inequality that
\begin{align*}
  \|w\|_{L^2}^2\leq C\nu^{-\f43}|k|^{-\f23}\|{F}\|^2_{H^{-1}}.
\end{align*}
This along with \eqref{eq:w-H1-t1} shows that $\nu\|w\|_{H^1}\lesssim\|{F}\|_{H^{-1}}$.
\end{proof}

\begin{lemma}\label{lem:R-navier-4}
Let $u$ {be} as in Proposition \ref{prop:R-navier-v2}.
It holds that
\begin{align*}
 (\nu |k|^2)^{\f12}\|u\|_{L^2}\leq C\|{F} \|_{H^{-1}}.
\end{align*}
\end{lemma}

\begin{proof}
Let $\delta=\nu^{\f13}|k|^{-\f13}$ and introduce a cut-off function $\chi(y)$ as follows
\begin{align}\label{def:cutoff-chi}
\chi(y)=\left\{
  \begin{aligned}
&\f{1}{y-\lambda}\quad\quad\quad\quad\quad\quad\quad\quad y\in(-1,1)\cap(-1,\lambda-\delta),\\
&2\f{y-\lambda}{\delta^2}-\f{(y-\lambda)^3}{\delta^4}\quad\quad y\in(-1,1)\cap(\lambda-\delta,\lambda+\delta),\\
&\f{1}{y-\lambda}\quad\quad\quad\quad\quad\quad\quad\quad y\in(-1,1)\cap(\lambda+\delta,1),
\end{aligned}
\right.
\end{align}
which satisfies
\beno
\|\chi\|_{L^2}\lesssim\delta^{-\f12},\quad \|\chi\|_{L^{\infty}}\lesssim\delta^{-1},\quad \|\chi'\|_{L^2}\lesssim\delta^{-\f32}.
\eeno

We get by integration by parts that
\begin{align*}
  \big\langle {F},\chi\varphi\big\rangle =&\nu\int_{-1}^{1}w'(\chi\overline{\varphi})'dy +\nu k^2\int_{-1}^{1}w\chi\overline{\varphi} dy +ik\int_{(-1,1)\setminus(\lambda-\delta,\lambda+\delta)}w\overline{\varphi}dy \\
  &+ik\int_{(-1,1)\cap(\lambda-\delta,\lambda+\delta)}\big((y-\lambda)\chi\big) w\overline{\varphi}dy-\epsilon\nu^{\f13}|k|^{\f23}\int_{-1}^{1}w\chi\overline{\varphi}dy.
\end{align*}
Using the facts that $\|w\|_{L^{\infty}}^2\leq\|w'\|_{L^2}\|w\|_{L^2}$ and $\|\varphi\|_{L^{\infty}}^2\leq\|\varphi'\|_{L^2}\|\varphi\|_{L^2}$, we get
\begin{align*}
   &\|u\|^{2}_{L^2}=\big\langle w,\varphi\big\rangle = \int_{(-1,1)\cap(\lambda-\delta,\lambda+\delta)}w\overline{\varphi} dy +\int_{(-1,1)\setminus(\lambda-\delta,\lambda+\delta)}w\overline{\varphi}dy\\
  &\lesssim \delta\|w\|_{L^\infty}\|\varphi\|_{L^\infty}+|k|^{-1}|\langle{F},\chi\varphi\big\rangle|+\nu|k|^{-1}\|w'\|_{L^2}\|(\chi\varphi)'\|_{L^2}+\nu |k|\|w\|_{L^2}\|\chi\|_{L^2}\|\varphi\|_{L^\infty} \\
 &\quad+\|w\|_{L^2}\|\varphi\|_{L^\infty}\|(y-\lambda)\chi\|_{L^{2}((-1,1)\cap(\lambda-\delta,\lambda+\delta))} +\epsilon\nu^{\f13}|k|^{-\f13}\|w\|_{L^2}\|\chi\|_{L^2}\|\varphi\|_{L^\infty}\\
&\lesssim \delta\|w'\|^{\f12}_{L^2}\|w\|^{\f12}_{L^2}\|\varphi'\|_{L^2}^{\f12}\|\varphi\|_{L^2}^{\f12}+|k|^{-1}\|{F}\|_{H^{-1}}\|\chi\varphi\|_{H^1}+ \nu|k|^{-1}\|w'\|_{L^2}\|(\chi\varphi)'\|_{L^2}\\
  &\quad+\nu|k|\delta^{-\f12}\|w\|_{L^2}\|\varphi'\|_{L^2}^{\f12}\|\varphi\|_{L^2}^{\f12}+ \delta^{\f12}\|w\|_{L^2}\|\varphi'\|_{L^2}^{\f12}\|\varphi\|_{L^2}^{\f12} +\epsilon\nu^{\f13}|k|^{-\f13}\delta^{-\f12}\|w\|_{L^2}\|\varphi'\|_{L^2}^{\f12}\|\varphi\|_{L^2}^{\f12}.
  \end{align*}
Thanks to $\|(\chi\varphi)'\|_{L^2}\lesssim\delta^{-1}\|\varphi'\|_{L^2}+\delta^{-\f32}\|\varphi\|_{L^{\infty}}$,
we get by Lemma {\ref{lem:R-navier-2}} that
\begin{align*}
   \|u\|^{2}_{L^2} \lesssim&\delta\nu^{-\f56}|k|^{-\f16}\|{F}\|_{H^{-1}}\|\varphi'\|_{L^2}^{\f12}\|\varphi\|_{L^2}^{\f12} +|k|^{-1}\|{F}\|_{H^{-1}}(\delta^{-1}\|\varphi'\|_{L^2}+\delta^{-\f32}\|\varphi\|_{L^{\infty}})\\
  &+\delta^{-\f12}\nu^{\f13}|k|^{\f23}\|{F}\|_{H^{-1}}\|\varphi'\|_{L^2}^{\f12}\|\varphi\|_{L^2}^{\f12} +\delta^{\f12}\nu^{-\f23}|k|^{-\f13}\|{F}\|_{H^{-1}}\|\varphi'\|_{L^2}^{\f12}\|\varphi\|_{L^2}^{\f12}\\ &+\epsilon\nu^{-\f13}|k|^{-\f23}\delta^{-\f12}\|{F}\|_{H^{-1}}\|\varphi'\|_{L^2}^{\f12}\|\varphi\|_{L^2}^{\f12}.
\end{align*}
As $\|u\|_{L^2}\thicksim\|\varphi'\|_{L^2}+|k|\|\varphi\|_{L^2}, \|\varphi'\|_{L^2}^{\f12}|k|^{\f12}\|\varphi\|_{L^2}^{\f12}\leq\|u\|_{L^2}$, we have
\begin{align*}
  \|u\|^{2}_{L^2}  \lesssim&(\delta\nu^{-\f56}|k|^{-\f23}+\delta^{-1}|k|^{-1}+\delta^{-\f32}|k|^{-\f32}+\delta^{-\f12}\nu^{\f13}|k|^{\f16} +\delta^{\f12}\nu^{-\f23}|k|^{-\f56}\\
  &+\epsilon\nu^{-\f13}|k|^{-\f76}\delta^{-\f12})\|{F}\|_{H^{-1}}\|u\|_{L^2}\\
  \lesssim&(\nu^{-\f12}|k|^{-1}+\nu^{-\f13}|k|^{-\f23}+\nu^{\f16}|k|^{\f13})\|{F}\|_{H^{-1}}\|u\|_{L^2}.
\end{align*}
When $\nu k^2\leq1$, we deduce that $\nu^{\f12}|k|\|u\|_{L^2}\lesssim\|{F}\|_{H^{-1}}$. When $\nu k^2\geq1$, we get by Lemma {\ref{lem:R-navier-2}} that
\begin{align*}
  \|u\|_{L^2} & \lesssim |k|^{-1}\|w\|_{L^2}\lesssim \nu^{-\f{1}{2}}|k|^{-1}(\nu^{-\f{1}{6}}|k|^{-\f{1}{3}})\|{F}\|_{H^{-1}} \leq\nu^{-\f{1}{2}}|k|^{-1}\|{F}\|_{H^{-1}}.
\end{align*}
\end{proof}

\subsection{Weak type resolvent estimate}

\begin{lemma}\label{lem:R-navier-weak}
Let $(\varphi, w)$ be as in Proposition \ref{prop:R-navier-v2}. Assume that $\nu k^2\le 1$ and $f\in H^1(-1,1),$  {$j\in\{\pm1\}$ and $f(-j)=0$}.
Then it holds that
\begin{align*}
|\langle w,f\rangle|\leq& C|k|^{-1}\|{F}\|_{H^{-1}}\big(\delta^{-\f32}\|f\|_{L^{\infty}((-1,1)\cap(\lambda-\delta,\lambda+\delta))}\\
&\quad+|f(j)|(|j-\lambda|+\delta)^{-\f34}\delta^{-\f34}
+\|f\chi\|_{H^1}+\delta^{-1}\|f\chi\|_{L^2}\big).
\end{align*}
Here $\delta=\nu^{\f13}|k|^{-\f13}$ and $ \chi$ is a cut-off function defined in \eqref{def:cutoff-chi}.
\end{lemma}

\begin{proof}
For $\phi\in H_0^1(-1,1)$, we get by integration by parts that
\begin{align*}
\|{F}\|_{H^{-1}}\|\phi\|_{H^1}&\geq|\langle{F},\phi\rangle|=\big|\langle -\nu(\partial_y^2-k^2)w+ik(y-\lambda)w- \epsilon\nu^{\f13}|k|^{\f23}w,\phi\rangle\big|\\
&\geq-\nu\|w'\|_{L^2}\|\phi\|_{L^2}-\big|\nu k^2-\epsilon\nu^{\f13}|k|^{\f23}\big|\|w\|_{L^2}\|\phi\|_{L^2}+|\langle k(y-\lambda)w,\phi\rangle|.
\end{align*}
As $\nu k^2\leq1$, we have $\left|\nu k^2-\epsilon\nu^{\f13}|k|^{\f23}\right|\leq \nu^{\f13}|k|^{\f23}$. Then by Lemma \ref{lem:R-navier-2}, we get
\begin{align}
|k||\langle w,(y-\lambda)\phi\rangle|\leq&\|{F}\|_{H^{-1}}\|\phi\|_{H^1}+\nu\|w'\|_{L^2}\|\phi\|_{L^2}+
\nu^{\f13}|k|^{\f23}\|w\|_{L^2}\|\phi\|_{L^2}\nonumber\\
\leq& C\|{F}\|_{H^{-1}}\|\phi\|_{H^1}+C\nu^{-\f13}|k|^{\f13}\|{F}\|_{H^{-1}}\|\phi\|_{L^2}\nonumber\\
=& C\|{F}\|_{H^{-1}}\big(\|\phi\|_{H^1}+\delta^{-1}\|\phi\|_{L^2}\big).\label{eq:w-weak}
\end{align}

Next we deal with the case when $\phi\in H^1(-1,1),\ \phi(-1)=0.$
In this case, for every $\delta_*\in(0,\delta]\subset(0,1]$, let $\chi_1(y)=\max(1-(1-y)/\delta_*,0)$ and $\phi_1(y)=\phi(y)-\phi(1)\chi_1(y)$ for $y\in[-1,1]$.
Then we have $\chi_1\in H^1(-1,1),\ \phi_1\in H_0^1(-1,1),\, {\rm supp}\chi_1=[1-\delta_*,1]$ and
\beno
&&\|\chi_1\|_{L^{\infty}}=1,\quad \|\chi_1\|_{L^{2}}\leq\delta_*^{\f12},\quad \|\chi_1'\|_{L^{2}}\leq\delta_*^{-\f12},\\ &&\|(y-\lambda)\chi_1\|_{L^{\infty}}\leq\|y-\lambda\|_{L^{\infty}([1-\delta_*,1])}\|\chi_1\|_{L^{\infty}}\leq |1-\lambda|+\delta_*.
\eeno
As $ w(1)=0$, we have $|w(y)|=|\int_y^1w'(z)dz|\leq |1-y|^{\f12}\|w'\|_{L^2}\leq \delta_*^{\f12}\|w'\|_{L^2}$ for $y\in[1-\delta_*,1]$ and $\|w\|_{L^{1}([1-\delta_*,1])}\leq \delta_*^{\f32}\|w'\|_{L^2}.$ By Lemma \ref{lem:R-navier-2} and \eqref{eq:w-weak}, we get
\begin{align*}
&|\langle w,(y-\lambda)\phi\rangle|\leq |\phi(1)\langle w,(y-\lambda)\chi_1\rangle|+|\langle w,(y-\lambda)\phi_1\rangle|\\
&\leq |\phi(1)|\|w\|_{L^{1}([1-\delta_*,1])}\|(y-\lambda)\chi_1\|_{L^{\infty}} +C|k|^{-1}\|{F}\|_{H^{-1}}(\|\phi_1\|_{H^1}+\delta^{-1}\|\phi_1\|_{L^2})\\ &\leq |\phi(1)|\delta_*^{\f32}\|w'\|_{L^2}(|1-\lambda|+\delta_*)
+C|k|^{-1}|\phi(1)|\|{F}\|_{H^{-1}}(\|\chi_1\|_{H^1}+\delta^{-1}\|\chi_1\|_{L^2}) \\&\quad+C|k|^{-1}\|{F}\|_{H^{-1}}(\|\phi\|_{H^1}+\delta^{-1}\|\phi\|_{L^2})\\
&\leq C|\phi(1)|\delta_*^{\f32}\nu^{-1}\|{F}\|_{H^{-1}}(|1-\lambda|+\delta)
+C|k|^{-1}|\phi(1)|\|{F}\|_{H^{-1}}(\delta_*^{-\f12}+\delta_*^{\f12}+\delta^{-1}\delta_*^{\f12}) \\&\quad+C|k|^{-1}\|{F}\|_{H^{-1}}(\|\phi\|_{H^1}+\delta^{-1}\|\phi\|_{L^2})\\ &\leq C|k|^{-1}|\phi(1)|\|{F}\|_{H^{-1}}(\delta_*^{\f32}(|1-\lambda|+\delta)\delta^{-3}+\delta_*^{-\f12}) +C|k|^{-1}\|{F}\|_{H^{-1}}(\|\phi\|_{H^1}+\delta^{-1}\|\phi\|_{L^2}).
\end{align*}
Here we used the fact that $ \nu^{-1}|k|=\delta^{-3}.$ Taking $\delta_*=(|1-\lambda|+\delta)^{-\f12}\delta^{\f32}$, we get
\begin{align*}
|\langle w,(y-\lambda)\phi\rangle|\leq& C|k|^{-1}|\phi(1)|\|{F}\|_{H^{-1}}(|1-\lambda|+\delta)^{\f14}\delta^{-\f34}\\& +C|k|^{-1}\|{F}\|_{H^{-1}}\big(\|\phi\|_{H^1}+\delta^{-1}\|\phi\|_{L^2}\big)
\end{align*}
for $\phi\in H^1(-1,1)$ with $\phi(-1)=0.$

Now for $f\in H^1(-1,1),\ f(-1)=0,$ let $\phi=f\chi$. Then we have $\phi\in H^1(-1,1),\ \phi(-1)=0$. Thus, we have
\begin{align*}
|\langle w,f\rangle|\leq& \|w\|_{L^2}\|f-(y-\lambda)\phi\|_{L^2}+|\langle w,(y-\lambda)\phi\rangle|\\ \leq& C\nu^{-\f23}|k|^{-\f13}\|{F}\|_{H^{-1}}\|f-(y-\lambda)\phi\|_{L^2} + C|k|^{-1}|\phi(1)|\|{F}\|_{H^{-1}}(|1-\lambda|+\delta)^{\f14}\delta^{-\f34}\\& +C|k|^{-1}\|{F}\|_{H^{-1}}(\|\phi\|_{H^1}+\delta^{-1}\|\phi\|_{L^2})\\ =&C|k|^{-1}\|{F}\|_{H^{-1}}\big(\delta^{-2}\|f-(y-\lambda)\phi\|_{L^2}+|\phi(1)|(|1-\lambda|+\delta)^{\f14}\delta^{-\f34}
+\|\phi\|_{H^1}+\delta^{-1}\|\phi\|_{L^2}\big).
\end{align*}
As $0\leq 1-(y-\lambda)\chi\leq 1$ for $y\in[-1,1]$ and $1-(y-\lambda)\chi=0$ for $y\not\in(\lambda-\delta,\lambda+\delta)$, we have
\begin{align*}
&\|f-(y-\lambda)\phi\|_{L^2}=\|f(1-(y-\lambda)\chi)\|_{L^2}\leq (2\delta)^{\f12}\|f\|_{L^{\infty}((-1,1)\cap(\lambda-\delta,\lambda+\delta))}.
\end{align*}
Thanks to $|\chi(y)|\leq C(|y-\lambda|+\delta)^{-1} $ for $y\in[-1,1]$, we get
\begin{align*}
|\phi(1)|=|f(1)||\chi(1)|\leq C|f(1)|(|1-\lambda|+\delta)^{-1}.
\end{align*}
Thus, we conclude that
\begin{align*}
|\langle w,f\rangle|\leq&  C|k|^{-1}\|{F}\|_{H^{-1}}\big(\delta^{-\f32}\|f\|_{L^{\infty}((-1,1)\cap(\lambda-\delta,\lambda+\delta))}\\&+
|f(1)|(|1-\lambda|+\delta)^{-\f34}\delta^{-\f34}
+\|f\chi\|_{H^1}+\delta^{-1}\|f\chi\|_{L^2}\big).
\end{align*}

The case of $f(1)=0$ can be proved similarly.
\end{proof}

\section{Resolvent estimates with nonslip boundary condition}
In this section, we study the resolvent estimates of the linearized operator under the non-slip boundary condition.
For this, we will use the stream function formulation
\begin{align}
\label{eq:R-phi-non}
\left\{
\begin{aligned}
&-\nu(\partial_y^2-k^2)^2\varphi+ik(y-\lambda)(\partial_y^2-k^2)\varphi-\epsilon\nu^{\f13}|k|^{\f23}(\partial_y^2-k^2)\varphi=F,\\
&\varphi(\pm1)=0,\quad\varphi'(\pm1)=0,
\end{aligned}
\right.
\end{align}
where $\lambda\in \R$ and ${\epsilon\geq0}$ small enough independent of $\nu, k,\lambda$.
We introduce
\beno
w=(\pa_y^2-k^2)\varphi,\quad u=(-\pa_y\varphi, ik\varphi).
\eeno

\subsection{Reformulation of the problem}
We introduce the following decomposition
\begin{align}\label{eq:varphi-decom}
\varphi=\varphi_{Na}+c_1\varphi_1+c_2\varphi_2,
\end{align}
where $\varphi_{Na}$ solves the OS equation with the Navier-slip boundary condition
\begin{align}
\label{eq:varphi-navier}
\left\{
\begin{aligned}
&-\nu (\partial_y^2-k^2)^2\varphi_{Na}+ik(y-\lambda) (\partial_y^2-k^2)\varphi_{Na}-\epsilon\nu^{\f13}|k|^{\f23}(\partial^2_y-k^2)\varphi_{Na}={F},\\
&\varphi_{Na}(\pm 1)=0, \quad \varphi_{Na}''(\pm 1)=0,
\end{aligned}
\right.
\end{align}
and $\varphi_1$, $\varphi_2$ solve the following homogeneous OS equations
\begin{align}
\label{eq:phi1-hom}
\left\{
\begin{aligned}
&-\nu (\partial_y^2-k^2)^2\varphi_1+ik(y-\lambda) (\partial_y^2-k^2)\varphi_1-\epsilon\nu^{\f13}|k|^{\f23}(\partial^2_y-k^2)\varphi_1=0, \\
& \varphi_1(\pm 1)=0, \quad\varphi_1'(1)=1,\quad\varphi_1'(-1)=0,
\end{aligned}
\right.
\end{align}
and
\begin{align}
\label{eq:phi2-hom}
\left\{
\begin{aligned}
&-\nu (\partial_y^2-k^2)^2\varphi_2+ik(y-\lambda) (\partial_y^2-k^2)\varphi_2-\epsilon\nu^{\f13}|k|^{\f23}(\partial^2_y-k^2)\varphi_2=0, \\
& \varphi_2(\pm 1)=0, \quad\varphi_2'(-1)=1,\quad\varphi_2'(1)=0.
\end{aligned}
\right.
\end{align}
Let $w_{Na}=(\partial_y^2-k^2)\varphi_{Na}$ and  $w_i=(\partial_y^2-k^2)\varphi_i$. Then we have
\begin{align}\label{eq:w-decomp}
w=w_{Na}+c_1w_1+c_2w_2= (\partial_y^2-k^2)\varphi.
\end{align}

Next we determine the coefficients $c_1, c_2$. The boundary condition $\varphi(\pm 1)=\varphi'(\pm 1)=0$ implies that
\begin{align*}
&\int_{-1}^1e^{\pm ky}w(y)dy=\int_{-1}^1e^{\pm ky}(\partial_y^2-k^2)\varphi(y)dy=0, \\
&\int_{-1}^1e^{\pm ky}w_1(y)dy=\int_{-1}^1e^{\pm ky}(\partial_y^2-k^2)\varphi_1(y)dy=e^{\pm k}, \\
&\int_{-1}^1e^{ky}w_2(y)dy=\int_{-1}^1e^{ky}(\partial_y^2-k^2)\varphi_2(y)dy=-e^{\mp k}.
\end{align*}
Then we infer that
\begin{align*}
0&=\int_{-1}^1e^{ky}w(y)dy=\int_{-1}^1e^{ky}w_{Na}(y)dy+c_1\int_{-1}^1e^{ky}w_1(y)dy+c_2\int_{-1}^1e^{ky}w_2(y)dy \\
&=\int_{-1}^1e^{ky}w_{Na}(y)dy+ e^kc_1-e^{-k}c_2,
\end{align*}
and
\begin{align*}
0&=\int_{-1}^1e^{-ky}w(y)dy=\int_{-1}^1e^{-ky}w_{Na}(y)dy+c_1\int_{-1}^1e^{-ky}w_1(y)dy+c_2\int_{-1}^1e^{-ky}w_2(y)dy \\
&=\int_{-1}^1e^{-ky}w_{Na}(y)dy+ e^{-k}c_1-e^{k}c_2.
\end{align*}
That is,
\begin{align*}
\left\{
\begin{aligned}
&e^kc_1-e^{-k}c_2=-\int_{-1}^1e^{ky}w_{Na}(y)dy,\\
&e^{-k}c_1-e^{k}c_2=-\int_{-1}^1e^{-ky}w_{Na}(y)dy.
\end{aligned}
\right.
\end{align*}
Hence, we obtain
\begin{align}
c_1(\lambda)&=-\f{1}{e^{2k}-e^{-2k}}\Big(\int_{-1}^1e^{k(y+1)}w_{Na}(y)dy-\int_{-1}^1e^{-k(y+1)}w_{Na}(y)dy\Big)\nonumber\\
&=-\int_{-1}^1\f{\sinh k(y+1)}{\sinh 2k}w_{Na}(y)dy,\label{def:c1}\\
c_2(\lambda)&=-\f{1}{e^{2k}-e^{-2k}}\Big(\int_{-1}^1e^{k(y-1)}w_{Na}(y)dy-\int_{-1}^1e^{-k(y-1)}w_{Na}(y)dy\Big)\nonumber\\
&=\int_{-1}^1\f{\sinh k(1-y)}{\sinh 2k}w_{Na}(y)dy.\label{def:c2}
\end{align}

\subsection{Bounds on $c_1$ and $c_2$}
We assume that $\nu k^2\le 1$.

\begin{lemma}\label{lem:c12-bounds-L2}
If $F\in L^2(I)$, then we have
\begin{align}\label{eq:c12-L2}
 &(1+|k(\lambda-1)|)|c_1|+(1+|k(\lambda+1)|)|c_2|\leq C\nu^{-\f{1}{6}}|k|^{-\f56}\|{F}\|_{L^2}.
\end{align}
\end{lemma}

\begin{proof}
 For the case of $|\lambda-1|\leq|k|^{-1}$, we get by \eqref{def:c1} and  Corollary \ref{cor:R-Navier-v1} that
  \begin{align*}
    |c_1|\leq C\|w_{Na}\|_{L^1}\leq C\nu^{-\f{1}{6}}|k|^{-\f56}\|{F}\|_{L^2}.
  \end{align*}
 For the case of  $\lambda-1\geq|k|^{-1}$, we have $ \lambda-y\geq \lambda-1>0$ for $y\in(-1,1)$.
 Then we get by Corollary \ref{cor:R-Navier-v1} and Lemma \ref{lem:hyperfct-1} that
 \begin{align*}
 |c_1|=&\Big|\int_{-1}^{1}\f{\sinh k(1+y)}{\sinh 2k}w_{Na}dy\Big|\\
 \leq& \Big(\int_{-1}^1\Big(\f{\sinh k(1+y)}{(y-\lambda)\sinh 2k }\Big)^2dy\Big)^{\f12}\|(y-\lambda)w_{Na}\|_{L^2}\\
    \lesssim&|k|^{-\f12}(\lambda-1)^{-1}\|(y-\lambda)w_{Na}\|_{L^2}\\
    \lesssim& |k|^{-\f32}(\lambda-1)^{-1}\|{F}\|_{L^2}\leq \nu^{-\f{1}{6}}|k|^{-\f{5}6}|k(\lambda-1)|^{-1}\|{F}\|_{L^2}.
  \end{align*}

  For the case of $1-\lambda\geq|k|^{-1}$, let $E_1=(-1,1)\cap(-\infty,(\lambda+1)/2)$. Then $|\lambda-y|\geq |\lambda-1|/2>0$ for $y\in(-1,1)\setminus E_1$.
  By Corollary \ref{cor:R-Navier-v1}, Lemma \ref{lem:hyperfct-1} and Lemma \ref{lem:hperfct-2}, we get
  \begin{align*}
    |c_1|=&\Big|\int_{-1}^{1}\f{\sinh k(1+y)}{\sinh 2k}w_{Na}dy\Big|\\ \leq&\Big(\int_{(-1,1)\setminus E_1}\Big(\f{\sinh k(1+y)}{(y-\lambda)\sinh 2k }\Big)^2dy\Big)^{\f12}\|(y-\lambda)w_{Na}\|_{L^2}+\Big\|\f{\sinh k(1+y)}{\sinh 2k}\Big\|_{L^{\infty}(E_1)}\|w_{Na}\|_{L^1}\\
    \lesssim&|k|^{-\f12}|\lambda-1|^{-1}\|(y-\lambda)w_{Na}\|_{L^2}+e^{-|k|(1-\lambda)/2}\|w_{Na}\|_{L^1}\\
    \lesssim&|k|^{-\f32}|\lambda-1|^{-1}\|{F}\|_{L^2}+e^{-|k|(1-\lambda)/2}\nu^{-\f{1}{6}}|k|^{-\f56}\|{F}\|_{L^2}\\ \lesssim& |k(\lambda-1)|^{-1}\nu^{-\f{1}{6}}|k|^{-\f56}\|{F}\|_{L^2}.
  \end{align*}
Summing up, we conclude that
\beno
(1+|k(\lambda-1)|)|c_1|\lesssim\nu^{-\f{1}{6}}|k|^{-\f56}\|{F}\|_{L^2}.
\eeno
The estimate of $c_2$ is similar.
\end{proof}

\begin{lemma}\label{lem:c12-bounds-H-1}
If $F\in H^{-1}(I)$, then we have
\begin{align}\label{eq:c12-H-1}
 &(1+|k(\lambda-1)|)^{\f34}|c_1|+(1+|k(\lambda+1)|)^{\f34}|c_2|\leq C\nu^{-\f{1}{2}}|k|^{-\f12}\|{F}\|_{H^{-1}}.
\end{align}
\end{lemma}

\begin{proof}
Let $\delta=\nu^{\f13}|k|^{-\f13}\le |k|^{-1}$. We split the proof into three cases. \smallskip

{\bf Case 1.} $\lambda\geq 1+2|k|^{-1}$.\smallskip

 As $(-1,1)\cap(\lambda-\delta,\lambda+\delta)=\varnothing$, $\chi(y)=\f{1}{y-\lambda}$, where $\chi(y)$ is defined in \eqref{def:cutoff-chi}. Applying Lemma \ref{lem:R-navier-weak} with $f(y)=\f{\sinh(k(1+y))}{\sinh(2k)}$ and $j=1$, we deduce that
    \begin{align*}
      |c_1|&=\big|\langle w_{Na},f\rangle\big|\leq C|k|^{-1}\|{F}\|_{H^{-1}}\Big((|1-\lambda|+\delta)^{-\f34}\delta^{-\f34} +\big\|\f{f(y)}{y-\lambda}\big\|_{H^{1}}+\delta^{-1}\big\|\f{f(y)}{y-\lambda}\big\|_{L^2}\Big).
    \end{align*}
Thanks to Lemma \ref{lem:hyperfct-1}, we have
    \begin{align*}
      \Big\|\f{f(y)}{y-\lambda}\Big\|_{H^1} &\leq\|f'\|_{L^2}\big\|\f{1}{y-\lambda}\big\|_{L^\infty}+2\|f\|_{L^2}\big\|\f{1}{(y-\lambda)^2}\big\|_{L^\infty} +\|f\|_{L^2}\big\|\f{1}{y-\lambda}\big\|_{L^\infty}.\\
      &\lesssim |k|^{\f12}|1-\lambda|^{-1}+|k|^{-\f12}|1-\lambda|^{-2}+|k|^{-\f12}|1-\lambda|^{-1}\\
      &\lesssim|k|^{\f32}(1+|k(\lambda-1)|)^{-\f34},\\
      \Big\|\f{f(y)}{y-\lambda}\Big\|_{L^2} &\leq\|f\|_{L^2}\big\|\f{1}{y-\lambda}\big\|_{L^\infty}\lesssim|k|^{-\f12}|1-\lambda|^{-1}\lesssim |k|^{\f12}(1+|k(\lambda-1)|)^{-\f34}.
    \end{align*}
Due to $\delta^{-1}\ge |k|$, $(|1-\lambda|+\delta)^{-\f34}\delta^{-\f34}\leq \delta^{-\f32}(1+|k(\lambda-1)|)^{-\f34}$. Thus, we obtain
\begin{align*}
|c_1|\lesssim |k|^{-1}\|{F}\|_{H^{-1}}\delta^{-\f32} (1+|k(\lambda-1)|)^{-\f34} =\nu^{-\f{1}{2}}|k|^{-\f12}(1+|k(\lambda-1)|)^{-\f34}\|{F}\|_{H^{-1}}.
\end{align*}

{\bf Case 2.}  $|\lambda-1|\leq2|k|^{-1}$.\smallskip

Applying Lemma \ref{lem:R-navier-weak} with $f(y)=\f{\sinh(k(1+y))}{\sinh(2k)}$ and $j=1$, we get
  \begin{align}
      |c_1|=\big|\langle w_{Na},f\rangle\big|\leq& C|k|^{-1}\|{F}\|_{H^{-1}}\big(\delta^{-\f32}\|f\|_{L^\infty(E)}+(|1-\lambda|+\delta)^{-\f34}\delta^{-\f34} +\|f\chi\|_{H^{1}}\nonumber\\
      &\quad+\delta^{-1}\|f\chi\|_{L^2}\big),\label{eq:c1-est-2}
    \end{align}
  where $E=(-1,1)\cap(\lambda-\delta,\lambda+\delta)$. Using the facts that
    \begin{align*}
      \|f\chi\|_{H^1} &\leq\|f'\|_{L^2}\|\chi\|_{L^\infty}+\|f\|_{L^\infty}\|\chi'\|_{L^2} +\|f\|_{L^\infty}\|\chi\|_{L^2},\\
      &\lesssim \delta^{-1}|k|^{\f12}+\delta^{-\f32}+\delta^{-\f12}\lesssim\delta^{-\f32}(1+|k(\lambda-1)|)^{-\f34},\\
      \|f\chi\|_{L^2} &\leq\|f\|_{L^\infty}\|\chi\|_{L^2}\lesssim\delta^{-\f12}\lesssim \delta^{-\f12}(1+|k(\lambda-1)|)^{-\f34},
     \end{align*}
  we deduce that
    \begin{align*}
      |c_1| &\lesssim |k|^{-1}\|{F}\|_{H^{-1}}\delta^{-\f32} (1+|k(\lambda-1)|)^{-\f34}
      =\nu^{-\f{1}{2}}|k|^{-\f12}(1+|k(\lambda-1)|)^{-\f34}\|{F}\|_{H^{-1}}.
    \end{align*}

{\bf Case 3. } $\lambda\leq1-2|k|^{-1}$.\smallskip

 Let $E_1=(-1,1)\cap(-\infty, (\lambda+1)/2)$ and $E_1^c=(-1,1)\setminus(-\infty, (\lambda+1)/2)$. Due to $\f{\lambda+1}{2}\geq\lambda+\delta$, we have $\chi\big|_{E^c_1}=\f{1}{y-\lambda}$ and $E\subset E_1$. It is easy to see that
  \begin{align*}
    \|f\chi\|_{L^2} & \leq\|f\|_{L^\infty(E_1)}\|\chi\|_{L^2(E_1)} +\|f\|_{L^2(E_1^c)}\|\chi\|_{L^\infty(E_1^c)}\\
    &\leq\|f\|_{L^\infty(E_1)}\|\chi\|_{L^2(-1,1)} +\|f\|_{L^2(-1,1)}\|1/(y-\lambda)\|_{L^\infty(E_1^c)},\\
    \|f'\chi\|_{L^2}&\leq\|f'\|_{L^\infty(E_1)}\|\chi\|_{L^2(E_1)} +\|f'\|_{L^2(E_1^c)}\|\chi\|_{L^\infty(E_1^c)}\\
    &\leq\|f'\|_{L^\infty(E_1)}\|\chi\|_{L^2(-1,1)} +\|f'\|_{L^2(-1,1)}\|1/(y-\lambda)\|_{L^\infty(E_1^c)},\\
    \|f\chi'\|_{L^2}&\leq\|f\|_{L^\infty(E_1)}\|\chi'\|_{L^2(E_1)} +\|f\|_{L^2(E_1^c)}\|\chi'\|_{L^\infty(E_1^c)}\\
    &\leq\|f\|_{L^\infty(E_1)}\|\chi'\|_{L^2(-1,1)} +\|f\|_{L^2(-1,1)}\|2/(y-\lambda)^{2}\|_{L^\infty(E_1^c)}.
  \end{align*}
By  Lemma \ref{lem:hyperfct-1}, Lemma \ref{lem:hperfct-2} and $\|1/(y-\lambda)\|_{L^\infty(E_1^c)}\leq 2/(1-\lambda)$, we infer that
  \begin{align*}
    \|f\chi\|_{L^2} & \lesssim e^{-|k|(1-\lambda)/2}\delta^{-\f12} +|k|^{-\f12}(1-\lambda)^{-1}\lesssim\delta^{-\f12}(1+|k(1-\lambda)|)^{-\f34},\\
    \|f'\chi\|_{L^2} & \lesssim |k|e^{-|k|(1-\lambda)/2}\delta^{-\f12} +|k|^{\f12}(1-\lambda)^{-1}\lesssim\delta^{-\f32}(1+|k(1-\lambda)|)^{-\f34},\\
    \|f\chi'\|_{L^2} & \lesssim e^{-|k|(1-\lambda)/2}\delta^{-\f32} +|k|^{-\f12}(1-\lambda)^{-2}\lesssim\delta^{-\f32}(1+|k(1-\lambda)|)^{-\f34}.
  \end{align*}
This shows that
\beno
\|f\chi\|_{H^1}\leq\|f'\chi\|_{L^2}+\|f\chi'\|_{L^2} +{\|f\chi\|_{L^2}}\lesssim\delta^{-\f32}(1+|k(\lambda-1)|)^{-\f34}.
\eeno

On the other hand, we have
  \begin{align*}
    &\delta^{-\f32}\|f\|_{L^\infty(E)} \leq\delta^{-\f32}\|f\|_{L^\infty(E_1)}\leq\delta^{-\f32}e^{-|k|(1-\lambda)/2} \lesssim\delta^{-\f32}(1+|k(1-\lambda)|)^{-\f34},\\
    &(|1-\lambda|+\delta)^{-\f34}\delta^{-\f34}\leq \delta^{-\f32}(1+|k(\lambda-1)|)^{-\f34}.
  \end{align*}
Plugging these inequalities above into \eqref{eq:c1-est-2}, we get
    \begin{align*}
      |c_1|\lesssim\nu^{-\f{1}{2}}|k|^{-\f12}(1+|k(\lambda-1)|)^{-\f34}\|{F}\|_{H^{-1}}.
    \end{align*}

Combining three cases, we get
$$(1+|k(\lambda-1)|)^{\f34}|c_1|\lesssim\nu^{-\f{1}{2}}|k|^{-\f12}\|{F}\|_{H^{-1}}.$$
In a similar way, we can deduce the estimate of $c_2$.
\end{proof}

\subsection{Bounds on $w_1$ and $w_2$}

For the solutions $w_1,w_2$ of the homogeneous equation, we have the following uniform bounds.

\begin{proposition}\label{prop:w12-bounds}
Let $L=(\f{k}{\nu})^{\f13}$. There exists $k_0$ and $\delta_0$ independent of $\nu$ such that
if $L\ge 6k$ or $L\ge k\ge k_0$ and $\epsilon\in {[0},\delta_0)$, there holds
\begin{align*}
&\|w_1\|_{L^{\infty}}\leq C\nu^{-\f12}(1+|k||\lambda-1|)^{\f12},\\
&\|w_2\|_{L^{\infty}}\leq C\nu^{-\f12}(1+|k||\lambda+1|)^{\f12},\\
&\|w_1\|_{L^1}+\|w_2\|_{L^1}\leq C.
\end{align*}
\end{proposition}

Let $\rho_k$ be a weight function defined as
\begin{align}
\rho_k(y)=
\left\{
\begin{aligned}
&L(y+1)\qquad y\in[-1,-1+L^{-1}],\\
&1\qquad\quad\qquad y\in[-1+L^{-1},1-L^{-1}],\\
&L(1-y)\qquad y\in[1-L^{-1},1].
\end{aligned}
\right.\label{def:rhok}
\end{align}
We also need the following weighted version.

\begin{proposition}\label{prop:w12-bounds-weight}
There exists $k_0$ and $\delta_0$ independent of $\nu$ such that
if $L\ge 6k$ or $L\ge k\ge k_0$ and $\epsilon\in [0,\delta_0)$, there holds
\begin{align}
&\|\rho_k^{-\f14}w_1\|_{L^2}\leq C\nu^{-\f{7}{24}}|k|^{-\f{1}{12}}(1+|k(\lambda-1)|)^{\f38},\label{eq:w1-L2w}\\
&\|\rho_k^{-\f14}w_2\|_{L^2}\leq C\nu^{-\f{7}{24}}|k|^{-\f{1}{12}}(1+|k(\lambda+1)|)^{\f38},\label{eq:w2-L2w}\\
&\|\rho_k^{\f12}w_1\|_{L^2}+\|\rho_k^{\f12}w_2\|_{L^2}\leq CL^{\f12}.\label{eq:w12-L2w}
\end{align}
\end{proposition}

Let us remark that if $\nu k^2\le 1$ and $\nu\le \epsilon_0=6^{-3}k_0^{-2}$, then we have $L\ge 6k$ or $L\ge k\ge k_0$.\smallskip

The proof of Proposition \ref{prop:w12-bounds} and Proposition \ref{prop:w12-bounds-weight} is very technical.
So, the proof is left to next section.

\subsection{Resolvent estimates for $\nu k^2\ge 1$}

This case can be proved directly by using integration by parts.

\begin{proposition}\label{prop:R-non-vb}
Let $\varphi$ be a solution of \eqref{eq:R-phi-non}. If $F\in L^2(I)$, then we have
\begin{align*}
\nu ^{\f{5}{12}}|k|^{\f{5}{6}}\|w\|_{L^2}\leq\nu k^2\|w\|_{L^2}\leq C\|F\|_{L^2}.
\end{align*}
If $F\in H^{-1}(I)$, then we have
\beno
&&(\nu k^2)^{\f12}\|u\|_{L^2}\lesssim\nu k^2\|u\|_{L^2}\leq C\|{F} \|_{H^{-1}},\\
&&\nu^{\f34}|k|^{\f12}\|w\|_{L^2}\lesssim\nu |k|\|w\|_{L^2}\leq C\|{F}\|_{H^{-1}}.
\eeno
\end{proposition}

\begin{proof}By integration by parts, we get
\begin{align*}
\langle -{F},\varphi\rangle=&\big\langle\nu (\partial_y^2-k^2)^2\varphi-ik(y-\lambda)(\partial_y^2-k^2)\varphi +\epsilon\nu^{\f13}|k|^{\f23}(\partial^2_y-k^2)\varphi,\varphi\big\rangle\\
=&\nu\big(\|\varphi^{''}\|_{L^2}^2+2k^2\|\varphi'\|_{L^2}^2+k^4\|\varphi\|_{L^2}^2\big) -\epsilon\nu^{\f13}|k|^{\f23}\big(\|\varphi'\|^2_{L^2}+k^2\|\varphi\|^2_{L^2}\big)+ik\int_{-1}^1\overline{\varphi} \varphi'dy\\
&+ik^3\int_{-1}^1(y-\lambda)|\varphi|^2dy +ik \int_{-1}^1(y-\lambda)|\varphi'|^2dy,
\end{align*}
which implies
\begin{align*}
|\langle {F},\varphi\rangle|\geq& \nu\big(\|\varphi^{''}\|_{L^2}^2+2k^2\|\varphi'\|_{L^2}^2+k^4\|\varphi\|_{L^2}^2\big) -\epsilon\nu^{\f13}|k|^{\f23}\big(\|\varphi'\|^2_{L^2}+k^2\|\varphi\|^2_{L^2}\big)-|k|\int_{-1}^1|\varphi\varphi'|dy\\
\geq& \nu\big(\|\varphi^{''}\|_{L^2}^2+2k^2\|\varphi'\|_{L^2}^2+k^4\|\varphi\|_{L^2}^2\big) -\epsilon\nu^{\f13}|k|^{\f23}\big(\|\varphi'\|^2_{L^2}+k^2\|\varphi\|^2_{L^2}\big)\\
&-\f12\big(\f{1}{\nu k^2}\|\varphi'\|_{L^2}^2+\nu k^4\|\varphi\|_{L^2}^2\big)\\
\geq& \f14\nu\big(\|\varphi^{''}\|_{L^2}^2+2k^2\|\varphi'\|_{L^2}^2+k^4\|\varphi\|_{L^2}^2\big)=\f14\nu \|w\|_{L^2}^2\geq\f14\nu k^2\|w\|_{L^2}\|\varphi\|_{L^2},
\end{align*}
by using the facts that $\nu k^2\geq1$, $\|w\|_{L^2}\geq k^2\|\varphi\|_{L^2}$, and taking $\epsilon\leq\f14$. This implies the first inequality of the lemma.

Notice that
\beno
\nu\big(\|\varphi^{''}\|_{L^2}^2+2k^2\|\varphi'\|_{L^2}^2+k^4\|\varphi\|_{L^2}^2\big)\ge \nu k^2\|u\|_{L^2}^2.
\eeno
Then we have
\begin{align*}
\f14\nu k^2\|u\|_{L^2}^2\leq |\langle{F},\varphi\rangle|\leq\|{F}\|_{H^{-1}}\|\varphi\|_{H^1}\leq \|{F}\|_{H^{-1}}\|u\|_{L^2},
\end{align*}
which gives the second inequality. On the other hand,
\begin{align*}
\f14\nu\|w\|_{L^2}^2\leq |\langle{F},\varphi\rangle|\leq \|{F}\|_{H^{-1}}\|u\|_{L^2}\le C(\nu k^2)^{-1}\|F\|^2_{H^{-1}},
\end{align*}
which gives the third inequality.
\end{proof}

\subsection{Resolvent estimates for $\nu k^2\le 1$}

\begin{proposition}\label{prop:R-non-vs}
Let $\varphi$ be a solution of \eqref{eq:R-phi-non}. If $F\in L^2(I)$, then we have
\begin{align*}
 \nu^{\f16}|k|^{\f56}\|w\|_{L^1}+\nu ^{\f{5}{12}}|k|^{\f{5}{6}}\|w\|_{L^2}\leq C\|F\|_{L^2}.
\end{align*}
If $F\in H^{-1}(I)$, then we have
\begin{align*}
 \nu^{\f34}|k|^{\f12}\|w\|_{L^2}+\nu^{\f23}|k|^{\f13}\|\rho_k^{\f12}w\|_{L^2}+(\nu k^2)^{\f12}\|u\|_{L^2}\leq C\|{F}\|_{H^{-1}}.
\end{align*}
\end{proposition}

\begin{proof}
First of all, we consider the case of $F\in L^2(I)$. By Corollary \ref{cor:R-Navier-v1}, Proposition \ref{prop:w12-bounds} and Lemma \ref{lem:c12-bounds-L2},
we get
\begin{align*}
\|w\|_{L^1}\leq&\|w_{Na}\|_{L^1}+|c_1|\|w_1\|_{L^1}+|c_2|\|w_2\|_{L^1}\\
\leq&\|w_{Na}\|_{L^1}+C|c_1|+C|c_2|\\
\lesssim&\|w_{Na}\|_{L^1}\lesssim\nu^{-\f16}|k|^{-\f56}\|{F}\|_{L^2}.
\end{align*}
By Proposition \ref{prop:w12-bounds} and Lemma \ref{lem:c12-bounds-L2}, we have
\beno
|c_1|\|w_1\|_{L^2}\le |c_1|\|w_1\|_{L^1}^{\f12}\|w_1\|_{L^{\infty}}^{\f12}\leq C|c_1|\nu^{-\f14}(1+|k||\lambda-1|)^{\f14}\le C\nu^{-\f{5}{12}}|k|^{-\f56}\|{F}\|_{L^2}.
\eeno
Similarly, we have
\beno
|c_2|\|w_2\|_{L^2}\le C\nu^{-\f{5}{12}}|k|^{-\f56}\|{F}\|_{L^2}.
\eeno
Then by Corollary \ref{cor:R-Navier-v1}, we get
\begin{align*}
\|w\|_{L^2}\leq&\|w_{Na}\|_{L^2}+|c_1|\|w_1\|_{L^2}+|c_2|\|w_2\|_{L^2}\\
\lesssim&(\nu k^2)^{-\f13}\|{F}\|_{L^2}+\nu^{-\f{5}{12}}|k|^{-\f56}\|{F}\|_{L^2}\lesssim\nu^{-\f{5}{12}}|k|^{-\f56}\|{F}\|_{L^2}.
\end{align*}

Next we consider the case of $F\in H^{-1}(I)$.
By Lemma \ref{lem:c12-bounds-H-1} and Proposition \ref{prop:w12-bounds}, we have
\beno
|c_1|\|w_1\|_{L^2}+|c_2|\|w_2\|_{L^2}\le C\nu^{-\f{3}{4}}|k|^{-\f12}\|{F}\|_{H^{-1}},
\eeno
which along with Proposition \ref{prop:R-navier-v2} gives
\begin{align*}
   \|w\|_{L^2} &\leq\|w_{Na}\|_{L^2}+|c_1|\|w_1\|_{L^2}+|c_2|\|w_2\|_{L^2}\\
   &\lesssim\nu^{-\f23}|k|^{-\f13}\|{F}\|_{H^{-1}} +\nu^{-\f{3}{4}}|k|^{-\f12}\|{F}\|_{H^{-1}}\lesssim \nu^{-\f{3}{4}}|k|^{-\f12}\|{F}\|_{H^{-1}}.
 \end{align*}
By Lemma \ref{lem:c12-bounds-H-1} and Proposition \ref{prop:w12-bounds-weight}, we have
\beno
|c_1|\|{\rho_k^{\f12}w_1}\|_{L^2}+|c_2|\|{\rho_k^{\f12}w_2}\|_{L^2}\le C\nu^{-\f{2}{3}}|k|^{-\f13}\|{F}\|_{H^{-1}},
\eeno
which along with Proposition \ref{prop:R-navier-v2} gives
\begin{align*}
   \|\rho_k^{\f12}w\|_{L^2} &\leq\|\rho_k^{\f12}w_{Na}\|_{L^2}+|c_1|\|\rho_k^{\f12}w_1\|_{L^2}+|c_2|\|\rho_k^{\f12}w_2\|_{L^2}\\
   &\lesssim\|w_{Na}\|_{L^2} +\nu^{-\f{2}{3}}|k|^{-\f13}\|{F}\|_{H^{-1}}\lesssim \nu^{-\f{2}{3}}|k|^{-\f13}\|{F}\|_{H^{-1}}.
 \end{align*}

By Lemma \ref{lem:c12-bounds-H-1} and Proposition \ref{prop:w12-bounds}, we have
\beno
|c_1|\|w_1\|_{L^1}+|c_2|\|w_2\|_{L^1}\le C\nu^{-\f{1}{2}}|k|^{-\f12}\|{F}\|_{H^{-1}}.
\eeno
For the velocity, we have
\begin{align*}
u=u_{Na}+c_1u_1+c_2u_2,
\end{align*}
where $u_i=(-\partial_y\varphi_i,ik\varphi_i)$ and $(\partial^{2}_{y}-k^2)\varphi_i=w_i$ with $\varphi_i(\pm 1)=0, i=1,2$.
Then by Proposition \ref{prop:R-navier-v2} and Lemma \ref{lem:phi-w}, we get
\begin{align*}
  \|u\|_{L^2} &\leq\|u_{Na}\|_{L^2}+|c_1|\|u_{1}\|_{L^2}+|c_2|\|u_{2}\|_{L^2}
  \lesssim \nu^{-\f{1}{2}}|k|^{-1}\|{F}\|_{H^{-1}}.
\end{align*}

This completes the proof.
\end{proof}

\section{$L^p$ bounds on $w_1$ and $w_2$}

Recall that  $w_i=(\partial_y^2-k^2)\varphi_i(i=1,2)$, where $\varphi_1, \varphi_2$ solve
\begin{align*}
\left\{
\begin{aligned}
&-\nu (\partial_y^2-k^2)^2\varphi_1+ik(y-\lambda) (\partial_y^2-k^2)\varphi_1-\epsilon\nu^{\f13}|k|^{\f23}(\partial^2_y-k^2)\varphi_1=0, \\
& \varphi_1(\pm 1)=0, \quad\varphi_1'(1)=1,\quad\varphi_1'(-1)=0,
\end{aligned}
\right.
\end{align*}
and
\begin{align*}
\left\{
\begin{aligned}
&-\nu (\partial_y^2-k^2)^2\varphi_2+ik(y-\lambda) (\partial_y^2-k^2)\varphi_2-\epsilon\nu^{\f13}|k|^{\f23}(\partial^2_y-k^2)\varphi_2=0, \\
& \varphi_2(\pm 1)=0, \quad\varphi_2'(-1)=1,\quad\varphi_2'(1)=0.
\end{aligned}
\right.
\end{align*}

In this section, we use the Airy function to solve $w_i$ and give the $L^p$ bounds on $w_i$. Let $L=\big(\f k \nu\big)^\f13$.
We always assume that $L\ge 6k$ or $L\ge k\ge k_0$ for some big $k_0$.

\subsection{Airy function and the OS equation}

Let $Ai(y)$ be the Airy function, which is a nontrivial solution of $f''-yf=0$. Let
\begin{align*}
f_1(y)=Ai(e^{i\f{\pi}{6}}y),\quad f_2(y)=Ai(e^{i\f{5\pi}{6}}y).
\end{align*}
Then $f_1$ and $f_2$ are two linearly independent solutions of $f''-iyf=0$. Hence,
\begin{align*}
W_1(y)=Ai\big(e^{i\f{\pi}{6}}((L(y-\lambda-i k\nu)+i\epsilon)\big),\quad W_2(y)=Ai\big(e^{i\f{5\pi}{6}}((L(y-\lambda-i k\nu)+i\epsilon)\big)
\end{align*}
are two linearly independent solutions of the homogeneous OS equation
\beno
-\nu (w''-k^2w)+ik(y-\lambda)w-\epsilon\nu^{\f13}|k|^{\f23}w=0.
\eeno
Thus, $w_1$ and $w_2$ can be expressed as
\begin{align}
\label{two linear solutions}w_1=C_{11}W_1(y)+C_{12} W_2(y),\quad w_2=C_{21}W_1(y)+C_{22}W_2(y),
\end{align}
where $C_{ij}, i,j=1,2$ are constants. Thanks to the facts that
\begin{align*}
&\int_{-1}^1e^{ky}w_1(y)dy=e^k,\quad \int_{-1}^1e^{ky}w_1(y)dy=e^{-k},
\end{align*}
we get
\begin{align*}
\left\{
\begin{aligned}
&e^k=C_{11}\int_{-1}^1e^{ky}W_1(y)dy+C_{12}\int_{-1}^1e^{ky}W_2(y)dy,\\
&e^{-k}=C_{21}\int_{-1}^1e^{-ky}W_1(y)dy+C_{22}\int_{-1}^1e^{-ky}W_2(y)dy.
\end{aligned}
\right.
\end{align*}
Define the matrix $J$ as
\begin{equation*}
J=\left(
\begin{array}{cc}
A_1 & B_2\\
B_1 & A_2\\
 \end{array}
\right),
\end{equation*}
where
\begin{align*}
&A_1=\int_{-1}^1e^{ky}W_1(y)dy,\quad A_2=\int_{-1}^1e^{-ky}W_2(y)dy,\\
&B_1=\int_{-1}^1e^{-ky}W_1(y)dy,\quad B_2=\int_{-1}^1e^{ky}W_2(y)dy.
\end{align*}
Thus, if $A_1A_2-B_1B_2\neq0 $, then the solution exists and we have
\begin{equation}\label{eq:coeff of W12}
\left( \begin{array}{cc}  C_{11} \\  C_{12}\\
 \end{array}\right)=\f{\left( \begin{array}{cc}  A_2e^k-B_2e^{-k} \\  -B_1e^k+A_1e^{-k}\\
 \end{array}\right)}{A_1A_2-B_1B_2},
 \quad
\left( \begin{array}{cc}  C_{21} \\  C_{22}\\
 \end{array}\right)=\f{\left( \begin{array}{cc}  -A_2e^{-k}+B_2e^{k} \\  B_1e^{-k}-A_1e^{k}\\
 \end{array}\right)}{A_1A_2-B_1B_2}.
\end{equation}

\subsection{Estimates of $C_{ij}$ and $W_i$}

We introduce some notations
\begin{align*}
&A_0(z)=\int_{e^{i\pi/6}z}^{\infty}Ai(t)dt=e^{i\pi/6}\int_{z}^{\infty}Ai(e^{i\pi/6}t)dt,\\
&d=-1-\lambda-ik\nu,\quad\widetilde{d}=-1+\lambda-ik\nu.
\end{align*}

\begin{lemma}\label{lem:C-ij}
It holds that
\begin{align*}
&|C_{11}|\leq \f{CLe^{-2k}}{|A_0(Ld+i\epsilon)|},\quad|C_{12}|\leq \f{CL}{|A_0(L\widetilde{d}+i\epsilon)|},\\
&|C_{21}|\leq \f{CL}{|A_0(Ld+i\epsilon)|},\quad
|C_{22}|\leq \f{CLe^{-2k}}{|A_0(L\widetilde{d}+i\epsilon)|}.
\end{align*}
\end{lemma}

\begin{lemma}\label{lem:W12}
It holds that
\begin{align*}
&\f{L}{|A_0(Ld+i\epsilon)|}\|W_1\|_{L^{\infty}}\leq C\nu^{-\f12}\big(1+ |k(\lambda+1)|\big)^{\f12},\\
&\f{L}{|A_0(L\widetilde{d}+i\epsilon)|}\|W_2\|_{L^{\infty}}\leq C\nu^{-\f12}\big(1+ |k(\lambda-1)|\big)^{\f12},\\
&\f{L}{|A_0(Ld+i\epsilon)|}\|W_1\|_{L^1}+\f{L}{|A_0(L\widetilde{d}+i\epsilon)|}\|W_2\|_{L^1}\leq C,
\end{align*}
and
\begin{align*}
\f{L}{|A_0(Ld+i\epsilon)|}\|\rho_k^{\f12}W_1\|_{L^2}+\f{L}{|A_0(L\widetilde{d}+i\epsilon)|}\|\rho_k^{\f12}W_2\|_{L^2}\leq CL^{\f12}.
\end{align*}
\end{lemma}

In order to prove Lemma \ref{lem:C-ij} and Lemma \ref{lem:W12}, we need to use many deep estimates on the Airy function.
To lighten the reader's burden, a complete proof will be presented in section 8.

\subsection{Proof of Proposition \ref{prop:w12-bounds} and Proposition \ref{prop:w12-bounds-weight}}

Proposition \ref{prop:w12-bounds} is a direct consequence of Lemma \ref{lem:C-ij} and Lemma \ref{lem:W12}.
Next we prove Proposition \ref{prop:w12-bounds-weight}.
\begin{proof}
First of all, \eqref{eq:w12-L2w} is a direct consequence of Lemma \ref{lem:C-ij} and Lemma \ref{lem:W12}.

Let $\delta_*=\nu^{\f12}(1+|k(\lambda+1)|)^{-\f12}$. Due to \eqref{def:rhok}, we know that $\rho_k(y)=1\geq L\delta_* $ for $|y|\leq 1-L^{-1},$ $\rho_k(y)=L(1-|y|)\geq L\delta_* $ for $1-L^{-1}\leq|y|\leq 1-\delta_*,$ and $\rho_k(y)=L(1-|y|)$ for $1-\delta_*\leq|y|\leq 1.$
Then we deduce from Proposition \ref{prop:w12-bounds}  that\begin{align*}
\|w_1\|_{L^2}\leq\|w_1\|_{L^1}^{\f12}\|w_1\|_{L^{\infty}}^{\f12}\leq C(\nu^{-\f12}(1+|k||\lambda-1|)^{\f12})^{\f12}=C\nu^{-\f14}(1+|k||\lambda-1|)^{\f14},
\end{align*}
and that
\begin{align*}
\|\rho_k^{-\f14}w_1\|_{L^2}\leq&\|\rho_k^{-\f14}w_1\|_{L^2(-1+\delta_*,1-\delta_*)}+
\|\rho_k^{-\f14}w_1\|_{L^2((-1,-1+\delta_*)\cup(1-\delta_*,1))}\\ \leq& (L\delta_*)^{-\f14}\|w_1\|_{L^2}+\|(L(1-|y|))^{-\f14}\|_{L^2((-1,-1+\delta_*)\cup(1-\delta_*,1))}\|w_1\|_{L^{\infty}}\\ \leq& C(L\delta_*)^{-\f14}\nu^{-\f14}(1+|k||\lambda-1|)^{\f14}+CL^{-\f14}\delta_*^{\f14}\nu^{-\f12}(1+|k||\lambda-1|)^{\f12}\\ \leq& CL^{-\f14}\nu^{-\f38}(1+|k||\lambda-1|)^{\f38}=C\nu^{-\f{7}{24}}|k|^{-\f{1}{12}}(1+|k(\lambda-1)|)^{\f38}.
\end{align*}
The estimate of $\|\rho_k^{-\f14}w_2\|_{L^2}$ is similar. This proves \eqref{eq:w1-L2w} and \eqref{eq:w12-L2w}.
\end{proof}

\section{Space-time estimates of the linearized NS equations}

In this section, we establish the space-time estimates of the linearized 2-D Navier-Stokes equation in the vorticity formulation:
\begin{align}\label{eq:vorticity-LNS}
\partial_t\omega+L_k\omega=-ikf_1-\partial_yf_2,\quad\omega|_{t=0}=\omega_0(k,y),
\end{align}
where $\omega=(\partial_y^2-k^2)\varphi$ with $\varphi(\pm1)=\varphi'(\pm1)=0$ and
\beno
L_k\om=\nu(\pa_y^2-k^2)\om+iky\om.
\eeno

We introduce the following norms
\begin{align*}
&\|f\|_{L^pH^s}=\big\|\|f(t)\|_{H^s(I)}\big\|_{L^p(\R^+)},\quad\|f\|_{L^pL^q}=\big\|\|f(t)\|_{L^q(I)}\big\|_{L^p(\R^+)}.
\end{align*}

The main result of this section is the following space-time estimates.

\begin{proposition}\label{prop:sapce-times}
Let ${0<\nu\le \epsilon_0}$ and $\om$ be a solution of \eqref{eq:vorticity-LNS} with $\om_0\in {H^1}(I)$ and $f_1, f_2\in L^2L^2$, where $\om_0$
satisfies ${\langle\om_0,e^{\pm ky}\rangle=0}$. Then there exists a constant $C>0$ independent of $\nu, k$ so that
\begin{align*}
&|k|\|u\|_{L^\infty L^\infty}^2+k^2\|u\|_{L^2L^2}^2+(\nu k^2)^{\f12}\|\omega\|_{L^2L^2}^2+\|(1-|y|)^{\f12}\omega\|_{L^{\infty}L^2}^2\\
&\leq C\big(\|\omega_0\|_{L^2}^2+k^{-2}\|\partial_y\omega_0\|_{L^2}^2\big)+C\big(\nu^{-\f12}|k|\|f_1\|_{L^2L^2}^2+\nu^{-1}\|f_2\|_{L^2L^2}^2\big).
\end{align*}
Here $u=(\pa_y\varphi, {-ik\varphi})$.
\end{proposition}

We remark that the condition $\langle\om_0,e^{\pm ky}\rangle=0 $ is equivalent to $ \pa_y\varphi_0(\pm 1)=0$ or $u_0\in H_0^1(I).$

\subsection{Semigroup bounds}
First of all, we consider the linearized equation with the Navier-slip boundary condition:
\begin{align}\label{eq:w-Na}
\partial_t\omega_{Na}+L_k\omega_{Na}=0,\quad \omega_{Na}(t,k,\pm 1)=0,\quad \omega_{Na}|_{t=0}=\omega_{0}(k,y),
\end{align}
where $L_k=\nu(k^2-\pa_y^2)+iky$ with $D(L_k)={H^2\cap H_0^1}(-1,1)=\big\{f\in H^2(-1,1):f(\pm1)=0\big\}$.

Thanks to the fact that for $f\in D(L_k)$
\begin{align}
\label{eq:accretive-Lk}{\rm Re}\langle L_kf,f \rangle=\nu k^2\|f\|_{L^2}^2+\nu\|f'\|_{L^2}^2,
\end{align}
$L_k$ is an accretive operator for any $k\in\Z$. Let us recall that an operator $A$ in a Hilbert space $H$ is accretive if ${\rm Re} \langle Af,f\rangle\geq 0$ for all $f\in {D}(A),$ or equivalently $\|(\la+A)f\|\geq \la\|f\|$ for all $f\in {D}(A)$ and all $\la>0$. The operator $A$ is called m-accretive if in addition any $\la<0$ belongs to the resolvent set of $A$. We define
\beno
\Psi(A)=\inf\big\{\|(A-i\la)f\|: f\in {D}(A),\ \la\in \mathbb{R},\  \|f\|=1\big\}.
\eeno

We need the following Gearhart-Pr\"{u}ss type lemma with sharp bound \cite{Wei}.

\begin{lemma}\label{lem:GP}
Let $A$ be a m-accretive operator in a Hilbert space $H$. Then
$\|e^{-tA}\|\leq e^{-t\Psi+{\pi}/{2}}$ for any $t\geq 0$.
\end{lemma}

\begin{lemma}\label{lem:semi-Navier}
Let $\om_{Na}$ be a solution of \eqref{eq:w-Na} with $\om_0\in L^2(I)$. Then for any $k\in \Z$, there exist constants $C,c>0$ independent with $\nu, k$ such that
\begin{align*}
\|e^{-tL_k}\omega_{0}\|_{L^2}\leq Ce^{-c\nu^{\f13}|k|^{\f23}t-\nu t}\|\omega_{0}\|_{L^2},
\end{align*}
Moreover, for any $|k|\geq1$,
\begin{align*}
(\nu k^2)^{\f13}\|e^{-tL_k}\omega_{0}\|_{L^2L^2}^2\leq C\|\omega_{0}\|_{L^2}^2.
\end{align*}
\end{lemma}

\begin{proof}
Thanks to Proposition \ref{prop:R-navier-v1} and \eqref{eq:accretive-Lk}, there exists $c>0$ so that for any $k\in \Z$,
\begin{align*}
\Psi(L_k)\geq c(\nu k^2)^{\f13}+\nu,
\end{align*}
which along with Lemma \ref{lem:GP} gives the first inequality. The second inequality is a direct consequence of the first one.
\end{proof}

Next we consider the linearized equation \eqref{eq:vorticity-LNS} with non-slip boundary condition. In this case, $L_k$ is not an accretive operator.  By \cite{Rom} and the argument in Section 5 and Section 8, we know that the eigenvalues of $L_k$ must lie in a region with ${\rm Im}\lambda<-ck^{\f 23}\nu^\f13$ for some $c>0,\ \nu k^2\leq 1$. This implies a rough bound 
\beno
\|e^{-tL_k}\omega_0\|_{L^2}\leq C(\nu,k)e^{-c\nu^{\f13}|k|^{\f23}t}\|\omega_0\|_{L^2}.
\eeno
 This bound ensures that we can take the Fourier transform in $t$. In fact, 
we can give a more precise bound of  $e^{-tL_k}$ via the Dunford integral and the resolvent estimates:
\begin{align*}
\|e^{-tL_k}\omega_0\|_{L^2}\leq
\left\{
\begin{aligned}
&Ce^{-ct}\|\omega_0\|_{L^2},\quad\nu k^2\geq1,\\
&C\big({t^{-1}+}\nu^{-\f{5}{12}}|k|^{\f{5}{12}}e^{-c\nu^{\f13}|k|^{\f23}t}\big)\|\omega_0\|_{L^2},\quad\nu k^2\leq1.
\end{aligned}
\right.
\end{align*}

%

\subsection{Space-time estimates for $\nu k^2\ge 1$}

\begin{proposition}\label{prop:SPT-big}
Let $\nu k^2\ge 1$ and $\om$ be a solution of \eqref{eq:vorticity-LNS} with $\om_0\in L^2(I)$ and $f_1, f_2\in L^2L^2$.
Then there exists a constant $C>0$ independent of $\nu, k$ so that
\begin{align*}
&k^2\|u\|_{L^\infty L^2}^2+k^2\|u\|_{L^2L^2}^2+\nu k^2\|\omega\|_{L^2L^2}^2+\|\omega\|_{L^{\infty}L^2}^2\\ &\quad\leq C\nu^{-1}\big(\|f_1\|_{L^2L^2}^2+\|f_2\|_{L^2L^2}^2\big)+\|\omega_0\|_{L^2}^2.
\end{align*}
\end{proposition}
\begin{proof}
Taking $L^2$ inner product between \eqref{eq:vorticity-LNS} and $\varphi$, we get
\begin{align*}
  &\big\langle(\partial_t-\nu(\partial_y^2-k^2)+iky)\omega,-\varphi\big\rangle =\big\langle-ikf_1-\partial_yf_2,-\varphi\big\rangle,
\end{align*}
which gives
\begin{align*}
& \big\langle\partial_tu,u\big\rangle+\nu\|\omega\|_{L^2}^2 +ik\int_{-1}^{1}{\varphi}'\overline{\varphi} dy+ik\int_{-1}^{1}y|\varphi|^2dy+ ik^{3}\int_{-1}^{1}y|\varphi'|^2dy\\
&=\big\langle-ikf_1-\partial_yf_2,-\varphi\big\rangle=\big\langle ikf_1,\varphi\big\rangle-\big\langle f_2,\partial_y\varphi\big\rangle.
\end{align*}
Taking the real part of the above equality, we get
\begin{align*}
  &\f12\f{d}{dt}\|u\|^2_{L^2}+\nu\|\omega\|^2_{L^2}\\
&\leq |k|\int_{-1}^{1}|\varphi'{\varphi}|dy+(\nu k^2)^{-1}\|f_1\|_{L^2}^2+\f14\nu k^4\|\varphi\|_{L^2}^2+(\nu k^2)^{-1}\|f_2\|_{L^2}^2+\f14\nu k^2\|\varphi'\|_{L^2}^2\\
  &\leq \f12(\f{1}{\nu k^2}\|\varphi'\|_{L^2}^2+\nu k^4\|\varphi\|_{L^2}^2)+\f14\nu(k^2\|\varphi'\|_{L^2}^2+ k^4\|\varphi\|_{L^2}^2)+(\nu k^2)^{-1}(\|f_1\|_{L^2}^2+\|f_2\|_{L^2}^2)\\
&\leq\f{3\nu}{4}(k^2\|\varphi'\|_{L^2}^2+ k^4\|\varphi\|_{L^2}^2)+(\nu k^2)^{-1}(\|f_1\|_{L^2}^2+\|f_2\|_{L^2}^2)\\
&\leq\f{3\nu}{4}\|\omega\|_{L^2}^2+(\nu k^2)^{-1}(\|f_1\|_{L^2}^2+\|f_2\|_{L^2}^2).
\end{align*}
This shows that
\beno
\f{d}{dt}\|u\|^2_{L^2}+\f12\nu\|\omega\|^2_{L^2}\lesssim(\nu k^2)^{-1}\big(\|f_1\|_{L^2}^2+\|f_2\|_{L^2}^2\big),
\eeno
which gives
\begin{align*}
  \|u(t)\|_{L^2}^2+\nu\int_{0}^{t}\|\omega(s)\|_{L^2}^2ds &\lesssim(\nu k^2)^{-1}(\|f_1\|_{L^2L^2}^2+\|f_2\|_{L^2L^2}^2)+\|u_0\|_{L^2}^2,
\end{align*}
that is,
\begin{align*}
  k^2\|u\|_{L^\infty L^2}^2+\nu k^2\|\omega\|_{L^2L^2}^2 &\lesssim\nu^{-1}(\|f_1\|_{L^2L^2}^2+\|f_2\|_{L^2L^2}^2)+k^2\|u_0\|_{L^2}^2.
\end{align*}
Thanks to  $k^2\|u\|^2_{L^2}\leq\|\omega\|^2_{L^2}\le \nu k^2\|\omega\|_{L^2L^2}^2$, we get
\begin{align*}
k^2\|u\|_{L^\infty L^2}^2+k^2\|u\|_{L^2L^2}^2+\nu k^2\|\omega\|_{L^2L^2}^2 &\lesssim\nu^{-1}(\|f_1\|_{L^2L^2}^2+\|f_2\|_{L^2L^2}^2)+\|\omega_0\|^2_{L^2}.
\end{align*}

It remains to estimate $\|\omega\|_{L^{\infty}L^2}^2 .$ Let $F_1=\partial_t\varphi+iky\varphi$, which holds
\beno
(\partial_y^2-k^2)F_1=\partial_t\omega+iky\omega+2ik\partial_y\varphi,\quad F_1|_{y=\pm 1}=\partial_yF_1|_{y=\pm 1}=0.
\eeno
We get by integration by parts that
\begin{align*}
  & \big\langle ikf_1,F_1\big\rangle-\big\langle f_2,\partial_yF_1\big\rangle=\big\langle-ikf_1-\partial_yf_2,-F_1\big\rangle=
  \big\langle(\partial_t-\nu(\partial_y^2-k^2)+iky)\omega,-F_1\big\rangle\\
&= \big\langle(\partial_y^2-k^2)F_1-\nu(\partial_y^2-k^2)\omega-2ik\partial_y\varphi,-F_1\big\rangle\\
&=\|\partial_yF_1\|_{L^2}^2+k^2\|F_1\|_{L^2}^2+\nu\big\langle\omega,(\partial_y^2-k^2)F_1\big\rangle+
\big\langle2ik\partial_y\varphi,F_1\big\rangle\\&=\|\partial_yF_1\|_{L^2}^2+k^2\|F_1\|_{L^2}^2+\nu\big\langle\omega,
\partial_t\omega+iky\omega+2ik\partial_y\varphi\big\rangle+
\big\langle2ik\partial_y\varphi,F_1\big\rangle.
\end{align*}
Taking the real part of the above equality, we get
\begin{align*}
  &\f{\nu}2\f{d}{dt}\|\om\|^2_{L^2}+\|\partial_yF_1\|_{L^2}^2+k^2\|F_1\|_{L^2}^2\\
&\leq 2\nu|k||\big\langle\omega,
\partial_y\varphi\big\rangle|+2|k||\big\langle
\partial_y\varphi,F_1\big\rangle|
+|\big\langle ikf_1,F_1\big\rangle|+|\big\langle f_2,\partial_yF_1\big\rangle|\\
  &\leq\nu\|\om\|_{L^2}^2+\nu k^2\|\partial_y\varphi\|_{L^2}^2+2\|\partial_y\varphi\|_{L^2}^2+ k^2\|F_1\|_{L^2}^2/2+\|f_1\|_{L^2}^2/2+ k^2\|F_1\|_{L^2}^2/2\\&+ \|f_2\|_{L^2}^2/2+\|\partial_yF_1\|_{L^2}^2/2\\
&=\nu\|\om\|_{L^2}^2+(\nu k^2+2)\|\partial_y\varphi\|_{L^2}^2+ k^2\|F_1\|_{L^2}^2+(\|f_1\|_{L^2}^2+\|f_2\|_{L^2}^2)/2+\|\partial_yF_1\|_{L^2}^2/2,
\end{align*}
which gives
\beno
\nu\f{d}{dt}\|\om\|^2_{L^2}+\|\partial_yF_1\|_{L^2}^2\leq\|f_1\|_{L^2}^2+\|f_2\|_{L^2}^2+2\nu\|\om\|_{L^2}^2+2(\nu k^2+2)\|\partial_y\varphi\|_{L^2}^2.
\eeno
This along with the fact that
\begin{align*}
  & (\nu k^2+2)\|\partial_y\varphi\|_{L^2}^2\leq 3\nu k^2\|\partial_y\varphi\|_{L^2}^2\leq 3\nu\|\om\|_{L^2}^2,
\end{align*}
yields that
\begin{align*}
  \nu\|\om(t)\|_{L^2}^2 &\leq\|f_1\|_{L^2L^2}^2+\|f_2\|_{L^2L^2}^2+8\nu\|\om\|_{L^2L^2}^2 +\nu\|\om_0\|_{L^2}^2.
\end{align*}
Thus, we have
\begin{align*}
\|\om\|_{L^\infty L^2}^2\lesssim&\nu^{-1}(\|f_1\|_{L^2L^2}^2+\|f_2\|_{L^2L^2}^2)+\|\om\|_{L^2L^2}^2+\|\om_0\|_{L^2}^2\\ \lesssim&\nu^{-1}(\|f_1\|_{L^2L^2}^2+\|f_2\|_{L^2L^2}^2)+\|\om_0\|_{L^2}^2.
\end{align*}

Summing up, we conclude the proof.
\end{proof}

\subsection{Space-time estimates for $\nu k^2\le 1$}

We decompose $\omega=\omega_{I}+\omega_{H}$, where $\omega_I$ solves
\begin{align}
&\label{eq:w-inhom}(\partial_t-\nu(\partial_y^2-k^2)+iky)\omega_{I}=-ikf_1-\partial_yf_2,\quad\omega_{I}|_{t=0}=0,
\end{align}
and $\omega_H$ solves
\begin{align}
&\label{eq:w-hom}(\partial_t-\nu(\partial_y^2-k^2)+iky)\omega_{H}=0,\quad\omega_{H}|_{t=0}=\omega_0(k,y),
\end{align}
together with the boundary conditions
\beno
&&\omega_I=(\partial_y^2-k^2)\varphi_I,\quad\varphi_I(\pm1)=\varphi'_I(\pm1)=0,\\
&& \omega_H=(\partial_y^2-k^2)\varphi_H,\quad\varphi_H(\pm1)=\varphi'_H(\pm1)=0.
\eeno

\subsubsection{Space-time estimates of the inhomogeneous problem}

\begin{proposition}\label{prop:SPT-inhom}
Let $\nu k^2\leq1$ and $\om_I$ be a solution of \eqref{eq:w-inhom}. Then there exists a constant $C>0$ independent of $\nu,k$ such that
\begin{align*}
&(\nu k^2)^{\f13}\|\rho_k^{\f12}\omega_{I}\|_{L^2L^2}^2+k^2\|u_I\|_{L^2L^2}^2+(\nu k^2)^{\f12}\|\omega_{I}\|_{L^2L^2}^2+(\nu k^2)^{\f12}\|\omega_{I}\|_{L^{\infty}L^2}^2\\
\nonumber&\leq C\big(\nu^{-\f13}|k|^{\f43}\|f_1\|_{L^2L^2}^2+\nu^{-1}\|f_2\|_{L^2L^2}^2\big),
\end{align*}
\end{proposition}

\begin{proof}We take the Fourier transform in $t$:
\beno
w(\lambda,k,y)=\int_{0}^{+\infty}\omega_{I}(t,k,y)e^{-it\lambda}dt,\quad F_j(\lambda,k,y)=\int_{0}^{+\infty}f_{j}(t,k,y)e^{-it\lambda}dt,\ j=1,2.
\eeno
Thus, we have
\begin{align}\label{feq:wI=fourier}
(i\lambda -\nu(\partial_y^2-k^2)+iky)w(\lambda,k,y)=-ikF_1(\lambda,k,y)-\partial_yF_2(\lambda,k,y)
\end{align}
with $\int_{-1}^1w(\lambda,k,y)e^{\pm ky}dy=0.$ Using Plancherel's theorem, we know that
\begin{align*}
&\int_{0}^{+\infty}\|\omega_I(t)\|_{L^2}^2dt\sim\int_{\mathbb{R}}\|w(\lambda)\|_{L^2}^2d\lambda,\ \int_{0}^{+\infty}\|\rho_k^{\f12}\omega_I(t)\|_{L^2}^2dt\sim\int_{\mathbb{R}}\|\rho_k^{\f12}w(\lambda)\|_{L^2}^2d\lambda,\\&
\int_{0}^{+\infty}\|u_I(t)\|_{L^2}^2dt\sim\int_{\mathbb{R}}\|(k,\partial_y)(\partial_y^2-k^2)^{-1}w(\lambda)\|_{L^2}^2d\lambda,\\
&\int_{0}^{+\infty}\|f_j(t)\|_{L^2}^2dt\sim\int_{\mathbb{R}}\|F_j(\lambda)\|_{L^2}^2d\lambda,\quad j=1,2.
\end{align*}

We further decompose $w$ as follows
\beno
w(\lambda,k,y)=w_{Na}^{(1)}+w_{Na}^{(2)}+\big(c_1^{(1)}(\lambda)+c_1^{(2)}(\lambda)\big)w_1+\big(c_2^{(1)}(\lambda)+c_2^{(2)}(\lambda)\big)w_2,
\eeno
where $w_{Na}^{(1)}$ and $w_{Na}^{(2)}$ solve
\begin{align}
\label{eq:w-na1}&\big(i\lambda -\nu(\partial_y^2-k^2)+iky\big)w_{Na}^{(1)}(\lambda,k,y)=-ikF_1(\lambda,k,y),\ w_{Na}^{(1)}|_{y=\pm 1}=0,\\
\label{eq:w-na2}&\big(i\lambda -\nu(\partial_y^2-k^2)+iky\big)w_{Na}^{(2)}(\lambda,k,y)=-\partial_yF_2(\lambda,k,y),\ w_{Na}^{(2)}|_{y=\pm 1}=0,
\end{align}
and {$w_i=(\partial_y^2-k^2)\varphi_i$ with $\varphi_1$, $\varphi_2$ solving \eqref{eq:phi1-hom}, \eqref{eq:phi2-hom} with $ \epsilon=0,$ $ \lambda$ replaced by $ \lambda'=-\lambda/k$}, and
\begin{align*}
&c_1^{(j)}(\lambda)=-\int_{-1}^1\f{\sinh k(y+1)}{\sinh 2k}w_{Na}^{(j)}(\lambda,k,y)dy,\\
&c_2^{(j)}(\lambda)=\int_{-1}^1\f{\sinh k(1-y)}{\sinh 2k}w_{Na}^{(j)}(\lambda,k,y)dy.
\end{align*}

By Corollary \ref{cor:R-Navier-v1}, we have
\begin{align*}
&\nu^{\f16}|k|^{\f43}\|(k,\partial_y)(\partial_y^2-k^2)^{-1}w_{Na}^{(1)}(\lambda)\|_{L^2}+(\nu k^2)^{\f13}\|w_{Na}^{(1)}(\lambda)\|_{L^2}\leq {C}\|k{F}_1(\lambda)\|_{L^2},
\end{align*}
and by Proposition \ref{prop:R-navier-v2},
\begin{align*}
 (\nu |k|^2)^{\f12}\|(k,\partial_y)(\partial_y^2-k^2)^{-1}w_{Na}^{(2)}(\lambda)\|_{L^2}+\nu^{\f23}|k|^{\f13}\|w_{Na}^{(2)}(\lambda)\|_{L^2}\leq C\|{F}_2(\lambda) \|_{L^{2}} .
\end{align*}

By Lemma \ref{lem:c12-bounds-L2}, Proposition \ref{prop:w12-bounds} and Proposition \ref{prop:w12-bounds-weight}, we have
\begin{align*}
&|c_1^{(1)}(\lambda)|\|w_1\|_{L^2}+|c_2^{(1)}(\lambda)|\|w_2\|_{L^2}\leq C\nu^{-\f{5}{12}}|k|^{-\f56}\|k{F}_1(\lambda)\|_{L^2},\\& |c_1^{(1)}(\lambda)|\|\rho_k^{\f12}w_1\|_{L^2}+|c_2^{(1)}(\lambda)|\|\rho_k^{\f12}w_2\|_{L^2}\leq C\nu^{-\f{1}{3}}|k|^{-\f23}\|k{F}_1(\lambda)\|_{L^2},\\& |c_1^{(1)}(\lambda)|\|w_1\|_{L^1}+|c_2^{(1)}(\lambda)|\|w_2\|_{L^1}\leq C\nu^{-\f{1}{6}}|k|^{-\f56}\|k{F}_1(\lambda)\|_{L^2},
\end{align*}
and by Lemma \ref{lem:c12-bounds-H-1}, Proposition \ref{prop:w12-bounds} and Proposition \ref{prop:w12-bounds-weight},
\begin{align*}
&|c_1^{(2)}(\lambda)|\|w_1\|_{L^2}+|c_2^{(2)}(\lambda)|\|w_2\|_{L^2}\leq C\nu^{-\f{3}{4}}|k|^{-\f12}\|{F}_2(\lambda)\|_{L^2},\\& |c_1^{(2)}(\lambda)|\|\rho_k^{\f12}w_1\|_{L^2}+|c_2^{(2)}(\lambda)|\|\rho_k^{\f12}w_2\|_{L^2}\leq C\nu^{-\f{2}{3}}|k|^{-\f13}\|{F}_2(\lambda)\|_{L^2},\\& |c_1^{(2)}(\lambda)|\|w_1\|_{L^1}+|c_2^{(2)}(\lambda)|\|w_2\|_{L^1}\leq C\nu^{-\f{1}{2}}|k|^{-\f12}\|{F}_2(\lambda)\|_{L^2}.
\end{align*}
This shows that
\begin{align*}
\|w(\lambda)\|_{L^2}\leq&\|w_{Na}^{(1)}(\lambda)\|_{L^2}+\|w_{Na}^{(2)}(\lambda)\|_{L^2}+
|c_1^{(1)}(\lambda)|\|w_1\|_{L^2}+|c_2^{(1)}(\lambda)|\|w_2\|_{L^2}\\&
+|c_1^{(2)}(\lambda)|\|w_1\|_{L^2}+|c_2^{(2)}(\lambda)|\|w_2\|_{L^2}\\ \leq& {C}(\nu k^2)^{-\f13}\|k{F}_1(\lambda)\|_{L^2}+C\nu^{-\f23}|k|^{-\f13}\|{F}_2(\lambda) \|_{L^{2}}\\&+C\nu^{-\f{5}{12}}|k|^{-\f56}\|k{F}_1(\lambda)\|_{L^2}+C\nu^{-\f{3}{4}}|k|^{-\f12}\|{F}_2(\lambda)\|_{L^2}\\ \leq& C\nu^{-\f{5}{12}}|k|^{-\f56}\|k{F}_1(\lambda)\|_{L^2}+C\nu^{-\f{3}{4}}|k|^{-\f12}\|{F}_2(\lambda)\|_{L^2},
\end{align*}
and
\begin{align*}
&\|\rho_k^{\f12}w(\lambda)\|_{L^2}\leq {C}(\nu k^2)^{-\f13}\|k{F}_1(\lambda)\|_{L^2}+C\nu^{-\f23}|k|^{-\f13}\|{F}_2(\lambda)\|_{L^{2}},
\end{align*}
and by Lemma \ref{lem:phi-w},
\begin{align*}
&\|(k,\partial_y)(\partial_y^2-k^2)^{-1}w(\lambda)\|_{L^2}\\ &\leq\|(k,\partial_y)(\partial_y^2-k^2)^{-1}w_{Na}^{(1)}(\lambda)\|_{L^2}
+\|(k,\partial_y)(\partial_y^2-k^2)^{-1}w_{Na}^{(2)}(\lambda)\|_{L^2}\\&\quad+
|c_1^{(1)}(\lambda)|\|(k,\partial_y)(\partial_y^2-k^2)^{-1}w_1\|_{L^2}+
|c_2^{(1)}(\lambda)|\|(k,\partial_y)(\partial_y^2-k^2)^{-1}w_2\|_{L^2}\\&\quad
+|c_1^{(2)}(\lambda)|\|(k,\partial_y)(\partial_y^2-k^2)^{-1}w_1\|_{L^2}+
|c_2^{(2)}(\lambda)|\|(k,\partial_y)(\partial_y^2-k^2)^{-1}w_2\|_{L^2}\\ &\leq {C}\nu^{-\f16}|k|^{-\f43}\|k{F}_1(\lambda)\|_{L^2}+C(\nu k^2)^{-\f12}\|{F}_2(\lambda) \|_{L^{2}}+
|k|^{-\f12}\big(|c_1^{(1)}(\lambda)|\|w_1\|_{L^1}\\&\qquad+
|c_2^{(1)}(\lambda)|\|w_2\|_{L^1}
+|c_1^{(2)}(\lambda)|\|w_1\|_{L^1}+
|c_2^{(2)}(\lambda)|\|w_2\|_{L^1}\big)\\ &\leq {C}\nu^{-\f16}|k|^{-\f43}\|k{F}_1(\lambda)\|_{L^2}+C(\nu k^2)^{-\f12}\|{F}_2(\lambda) \|_{L^{2}}.
\end{align*}

In summary, we conclude that
\begin{align*}
 &(\nu k^2)^{\f13}\|\rho_k^{\f12}\omega_{I}\|_{L^2L^2}^2+k^2\|u_I\|_{L^2L^2}^2+(\nu k^2)^{\f12}\|\omega_{I}\|_{L^2L^2}^2\\
 &\quad\sim(\nu k^2)^{\f13}\big\|\|\rho_k^{\f12}w(\lambda)\|_{L^2}\big\|_{L^2(\mathbb{R})}^2+
 k^2\big\|\|(k,\partial_y)(\partial_y^2-k^2)^{-1}w(\lambda)\|_{L^2}\big\|_{L^2(\mathbb{R})}^2+(\nu k^2)^{\f12}\big\|\|w(\lambda)\|_{L^2}\big\|_{L^2(\mathbb{R})}^2\\ &\leq
 C(\nu k^2)^{\f13}(\nu k^2)^{-\f23}\big\|\|kF_1(\lambda)\|_{L^2}\big\|_{L^2(\mathbb{R})}^2+C(\nu k^2)^{\f13}\nu^{-\f{4}{3}}|k|^{-\f23}\big\|\|F_2(\lambda)\|_{L^2}\big\|_{L^2(\mathbb{R})}^2\\&\quad
 +C k^2\nu^{-\f{1}{3}}|k|^{-\f83}\big\|\|kF_1(\lambda)\|_{L^2}\big\|_{L^2(\mathbb{R})}^2+C k^2(\nu k^2)^{-1}\big\|\|F_2(\lambda)\|_{L^2}\big\|_{L^2(\mathbb{R})}^2\\&\quad
+ C(\nu k^2)^{\f12}\nu^{-\f{5}{6}}|k|^{-\f53}\big\|\|kF_1(\lambda)\|_{L^2}\big\|_{L^2(\mathbb{R})}^2+C(\nu k^2)^{\f12}\nu^{-\f{3}{2}}|k|^{-1}\big\|\|F_2(\lambda)\|_{L^2}\big\|_{L^2(\mathbb{R})}^2\\ &\leq
 C\nu^{-\f{1}{3}}|k|^{\f43}\big\|\|F_1(\lambda)\|_{L^2}\big\|_{L^2(\mathbb{R})}^2+C\nu^{-1}\big\|\|F_2(\lambda)\|_{L^2}\big\|_{L^2(\mathbb{R})}^2\\ &\quad\sim \nu^{-\f13}|k|^{\f43}\|f_1\|_{L^2L^2}^2+\nu^{-1}\|f_2\|_{L^2L^2}^2.
\end{align*}

It remains to consider $L^{\infty}L^2 $ estimate of $\omega_I$. Let $\omega_{Na}$ be a solution of\begin{align*}
&(\partial_t-\nu(\partial_y^2-k^2)+iky)\omega_{Na}=-ikf_1-\partial_yf_2,\quad\omega_{Na}|_{t=0}=0,\quad\omega_{Na}|_{y=\pm 1}=0.
\end{align*}
Set $w_{Na}(\lambda,k,y)=\int_{0}^{+\infty}\omega_{Na}(t,k,y)e^{-it\lambda}dt$, which satisfies
\begin{align*}
(i\lambda -\nu(\partial_y^2-k^2)+iky)w_{Na}(\lambda,k,y)=-ikF_1(\lambda,k,y)-\partial_yF_2(\lambda,k,y),\quad\omega_{Na}|_{y=\pm 1}=0.
\end{align*}
Thus, $w_{Na}=w_{Na}^{(1)}+w_{Na}^{(2)}. $ By Plancherel's theorem, we have
\begin{align*}
 &(\nu k^2)^{\f13}\|\omega_{Na}\|_{L^2L^2}^2 \sim(\nu k^2)^{\f13}\big\|\|w_{Na}(\lambda)\|_{L^2}\big\|_{L^2(\mathbb{R})}^2\\
 &\leq 2(\nu k^2)^{\f13}\big(\big\|\|w_{Na}^{(1)}(\lambda)\|_{L^2}\big\|_{L^2(\mathbb{R})}^2+\big\|\|w_{Na}^{(2)}(\lambda)\|_{L^2}\big\|_{L^2(\mathbb{R})}^2)\\ &\leq C(\nu k^2)^{\f13}\big(\big\|(\nu k^2)^{-\f13}\|kF_1(\lambda)\|_{L^2}\big\|_{L^2(\mathbb{R})}^2+\big\|\nu^{-\f23}|k|^{-\f13}\|F_2(\lambda)\|_{L^2}\big\|_{L^2(\mathbb{R})}^2\big)\\&=
 C\nu^{-\f{1}{3}}|k|^{\f43}\big\|\|F_1(\lambda)\|_{L^2}\big\|_{L^2(\mathbb{R})}^2+C\nu^{-1}\big\|\|F_2(\lambda)\|_{L^2}\big\|_{L^2(\mathbb{R})}^2\\ &\quad\sim \nu^{-\f13}|k|^{\f43}\|f_1\|_{L^2L^2}^2+\nu^{-1}\|f_2\|_{L^2L^2}^2.
\end{align*}
Moreover, we have
\begin{align*}
&\int_{0}^{+\infty}(\omega_{I}(t,k,y)-\omega_{Na}(t,k,y))e^{-it\lambda}dt\\
&=w(\lambda,k,y)-w_{Na}(\lambda,k,y)=\big(c_1^{(1)}(\lambda)+c_1^{(2)}(\lambda)\big)w_1+\big(c_2^{(1)}(\lambda)+c_2^{(2)}(\lambda)\big)w_2.
\end{align*}
Thus, we can write
\beno
\omega_{I}=\omega_{Na}+\omega_1^{(1)}+\omega_1^{(2)}+\omega_2^{(1)}+\omega_2^{(2)},
\eeno
where for $j=1,2$,
\beno
\omega_j^{(l)}(t,k,y)=\frac{1}{2\pi}\int_{\mathbb{R}}c_j^{(l)}(\lambda)w_j(\lambda,k,y)e^{it\lambda}d\lambda,\quad t>0.
\eeno

Now we estimate $\|\omega_{Na}\|_{L^{\infty}L^2} $ and $\|\omega_j^{(k)}\|_{L^{\infty}L^2} $ separately. Notice that
\begin{align*}
&\partial_t\|\omega_{Na}\|_{L^2}^2/2+\nu\|\partial_y\omega_{Na}\|_{L^2}^2+\nu k^2\|\omega_{Na}\|_{L^2}^2\\&={\rm Re}\big\langle(\partial_t-\nu(\partial_y^2-k^2)+iky)\omega_{Na},\omega_{Na}\big\rangle\\
&={\rm Re}\big\langle-ikf_1-\partial_yf_2,\omega_{Na}\big\rangle={\rm Re}\big(-ik \langle f_1,\omega_{Na}\rangle+\langle f_2,\partial_y\omega_{Na}\rangle\big)\\
&\leq |k|\|f_1\|_{L^2}\|\omega_{Na}\|_{L^2}+\|f_2\|_{L^2}\|\partial_y\omega_{Na}\|_{L^2},
\end{align*}
which gives
\begin{align*}
&\partial_t\|\omega_{Na}\|_{L^2}^2+\nu\|\partial_y\omega_{Na}\|_{L^2}^2+2\nu k^2\|\omega_{Na}\|_{L^2}^2\leq \nu^{-\f13}|k|^{\f43}\|f_1\|_{L^2}^2+(\nu k^2)^{\f13}\|\omega_{Na}\|_{L^2}^2+\nu^{-1}\|f_2\|_{L^2}^2.
\end{align*}
As $\omega_{Na}|_{t=0}=0$, this shows that
\begin{align*}
\|\omega_{Na}(t)\|_{L^2}^2\leq& \int_0^t\big(\nu^{-\f13}|k|^{\f43}\|f_1(s)\|_{L^2}^2+(\nu k^2)^{\f13}\|\omega_{Na}(s)\|_{L^2}^2+\nu^{-1}\|f_2(s)\|_{L^2}^2\big)ds\\
\leq& \nu^{-\f13}|k|^{\f43}\|f_1\|_{L^2L^2}^2+(\nu k^2)^{\f13}\|\omega_{Na}\|_{L^2L^2}^2+\nu^{-1}\|f_2\|_{L^2L^2}^2\\
\lesssim& \nu^{-\f13}|k|^{\f43}\|f_1\|_{L^2L^2}^2+\nu^{-1}\|f_2\|_{L^2L^2}^2.
\end{align*}
Thus, we get
\begin{align*}
&\|\omega_{Na}\|_{L^{\infty}L^2}^2\leq C\big(\nu^{-\f13}|k|^{\f43}\|f_1\|_{L^2L^2}^2+\nu^{-1}\|f_2\|_{L^2L^2}^2\big).
\end{align*}

Notice that \eqref{eq:w-na1} is equivalent to
\begin{align*}
&(-\nu(\partial_y^2-k^2)+ik(y-\lambda'))w_{Na}^{(1)}(\lambda,k,y)=-ikF_1(\lambda,k,y),\quad w_{Na}^{(1)}|_{y=\pm 1}=0,
\end{align*}
with $ \lambda'=-\lambda/k$. By Lemma \ref{lem:c12-bounds-L2} and Proposition \ref{prop:w12-bounds}, we have
\begin{align*}
&(1+|k(-\lambda/k-1)|)^{\f34}|c_1^{(1)}(\lambda)|\|w_1\|_{L^2}\leq C\nu^{-\f{5}{12}}|k|^{-\f56}\|k{F}_1(\lambda)\|_{L^2}.
\end{align*}
Thanks to $|k(-\lambda/k-1)|=|\lambda+k|$ and $\|(1+|\lambda+k|)^{-\f34}\|_{ L^2(\mathbb{R})}\leq C$, we have
\begin{align*}
\|\omega_1^{(1)}(t)\|_{L^2}\leq& \frac{1}{2\pi}\int_{\mathbb{R}}|c_1^{(1)}(\lambda)|\|w_1(\lambda)\|_{L^2}d\lambda\\
\leq& C\|(1+|\lambda+k|)^{-\f34}\nu^{-\f{5}{12}}|k|^{-\f56}\big\|k\|{F}_1(\lambda)\|_{L^2}\big\|_{L^1(\mathbb{R})}\\ \leq &C\nu^{-\f{5}{12}}|k|^{\f16}\|(1+|\lambda+k|)^{-\f34}\|_{L^2(\mathbb{R})}\big\|\|{F}_1(\lambda)\|_{L^2}\big\|_{L^2(\mathbb{R})}\sim
\nu^{-\f{5}{12}}|k|^{\f16}\|f_1\|_{L^2L^2},
\end{align*}
which shows that
\begin{align*}
&(\nu k^2)^{\f12}\|\omega_1^{(1)}\|_{L^{\infty}L^2}^2\leq C(\nu k^2)^{\f12}(\nu^{-\f{5}{12}}|k|^{\f16}\|f_1\|_{L^2L^2})=C\nu^{-\f13}|k|^{\f43}\|f_1\|_{L^2L^2}^2.
\end{align*}
Similarly, we have
\begin{align*}
&(\nu k^2)^{\f12}\|\omega_2^{(1)}\|_{L^{\infty}L^2}^2\leq C\nu^{-\f13}|k|^{\f43}\|f_1\|_{L^2L^2}^2.
\end{align*}

By Lemma \ref{lem:c12-bounds-H-1} and Proposition \ref{prop:w12-bounds-weight}, we have
\begin{align*}
 &(1+|k(-\lambda/k-1)|)^{\f34}|c_1^{(2)}(\lambda)|\|\rho_k^{\f12}w_1\|_{L^2}\leq C\nu^{-\f{2}{3}}|k|^{-\f13}\|{F}_2(\lambda)\|_{L^2},\\&(1+|k(-\lambda/k-1)|)^{\f38}|c_1^{(2)}(\lambda)|\|\rho_k^{-\f14}w_1\|_{L^2}\leq C\nu^{-\f{19}{24}}|k|^{-\f{7}{12}}\|{F}_2(\lambda)\|_{L^2},
\end{align*}
which give
\begin{align*}
 &|\lambda+k|^{\f32}|c_1^{(2)}(\lambda)|^2\|\rho_k^{\f12}w_1\|_{L^2}^2+
 (\nu k^2)^{\f14}|\lambda+k|^{\f34}|c_1^{(2)}(\lambda)|^2\|\rho_k^{-\f14}w_1\|_{L^2}^2\leq C\nu^{-\f{4}{3}}|k|^{-\f23}\|{F}_2(\lambda)\|_{L^2}^2.
\end{align*}
Thus, we obtain
\begin{align*}
&|\omega_1^{(2)}(t,k,y)|^2\leq \left|\int_{\mathbb{R}}|c_1^{(2)}(\lambda)w_1(\lambda,k,y)|d\lambda\right|^2\\ &\leq \int_{\mathbb{R}}\big(|\lambda+k|^{\f32}|c_1^{(2)}(\lambda)|^2\rho_k(y)|w_1(\lambda,k,y)|^2+ (\nu k^2)^{\f14}|\lambda+k|^{\f34}|c_1^{(2)}(\lambda)|^2\rho_k^{-\f12}(y)|w_1(\lambda,k,y)|^2]\big)d\lambda\\&\quad
\times\int_{\mathbb{R}}\big(|\lambda+k|^{\f32}\rho_k(y)+ (\nu k^2)^{\f14}|\lambda+k|^{\f34}\rho_k^{-\f12}(y)\big)^{-1}d\lambda,
\end{align*}
for $t>0,\ y\in(-1,1)$. Notice that
\begin{align*}
&\int_{\mathbb{R}}(|\lambda+k|^{\f32}\rho_k(y)+ (\nu k^2)^{\f14}|\lambda+k|^{\f34}\rho_k^{-\f12}(y))^{-1}d\lambda\\ &\overset{z=(\lambda+k)\rho_k^{2}(y)(\nu k^2)^{-\f13}}{=}\int_{\mathbb{R}}\big(|z|^{\f32}\rho_k^{-2}(y)(\nu k^2)^{\f12}+ (\nu k^2)^{\f12}|z|^{\f34}\rho_k^{-2}(y)\big)^{-1}\rho_k^{-2}(y)(\nu k^2)^{\f13}dz\\&=(\nu k^2)^{-\f16} \int_{\mathbb{R}}(|z|^{\f32}+ |z|^{\f34})^{-1}dz=C(\nu k^2)^{-\f16}<+\infty.
\end{align*}
Thus, we have
\begin{align*}
&\|\omega_1^{(2)}(t)\|_{L^2}^2=\int_{-1}^1|\omega_1^{(2)}(t,k,y)|^2dy\leq C(\nu k^2)^{-\f16}\int_{\mathbb{R}}\int_{-1}^1\big(|\lambda+k|^{\f32}|c_1^{(2)}(\lambda)|^2\rho_k(y)|w_1(\lambda,k,y)|^2\\&\qquad+ (\nu k^2)^{\f14}|\lambda+k|^{\f34}|c_1^{(2)}(\lambda)|^2\rho_k^{-\f12}(y)|w_1(\lambda,k,y)|^2\big)dyd\lambda\\&=C(\nu k^2)^{-\f16}\int_{\mathbb{R}}\big(|\lambda+k|^{\f32}|c_1^{(2)}(\lambda)|^2\|\rho_k^{\f12}w_1\|_{L^2}^2+
 (\nu k^2)^{-\f14}|\lambda+k|^{\f34}|c_1^{(2)}(\lambda)|^2\|\rho_k^{\f14}w_1\|_{L^2}^2\big)d\lambda\\&\leq C(\nu k^2)^{-\f16}\int_{\mathbb{R}}\nu^{-\f{4}{3}}|k|^{-\f23}\|{F}_2(\lambda)\|_{L^2}^2d\lambda\sim
 \nu^{-\f{3}{2}}|k|^{-1}\|f_2\|_{L^2L^2}^2,
\end{align*}
which shows that
\begin{align*}
&(\nu k^2)^{\f12}\|\omega_1^{(2)}\|_{L^{\infty}L^2}^2\leq C(\nu k^2)^{\f12} \nu^{-\f{3}{2}}|k|^{-1}\|f_2\|_{L^2L^2}^2=C\nu^{-1}\|f_2\|_{L^2L^2}^2.
\end{align*}
Similarly, we have
\begin{align*}
&(\nu k^2)^{\f12}\|\omega_2^{(2)}\|_{L^{\infty}L^2}^2\leq C\nu^{-1}\|f_2\|_{L^2L^2}^2.
\end{align*}

In summary, we conclude that
\begin{align*}
(\nu k^2)^{\f12}\|\omega\|_{L^{\infty}L^2}^2&\leq C\|\omega_{Na}\|_{L^{\infty}L^2}^2+C\sum_{j,l=1}^2(\nu k^2)^{\f12}\|\omega_j^{(l)}\|_{L^{\infty}L^2}^2\\&\leq C\big( \nu^{-\f13}|k|^{\f43}\|f_1\|_{L^2L^2}^2+\nu^{-1}\|f_2\|_{L^2L^2}^2\big).
\end{align*}
\end{proof}

\subsubsection{Space-time estimates of the homogeneous problem}

\begin{proposition}\label{prop:SPT-hom}
Let $\nu k^2\leq1$ and $\om_H$ be a solution of  \eqref{eq:w-hom} with ${\langle\om_0,e^{\pm ky}\rangle=0}$. Then there exists a constant $C>0$ independent of $\nu,k$ such that
\begin{align*}
&(\nu k^2)^{\f12}\|\omega_{H}\|_{L^2L^2}^2+(\nu k^2)^{\f12}\|\omega_{H}\|_{L^{\infty}L^2}^2+(\nu k^2)^{\f13}\|\rho_k^{\f12}\omega_{H}\|_{L^2L^2}^2+k^2\|u_H\|_{L^2L^2}^2\\
\nonumber&\leq C\|\omega_0\|_{L^2}^2+C\nu^{\f13}|k|^{-\f43}\|\partial_y\omega_0\|_{L^2}^2.
\end{align*}
\end{proposition}

\begin{proof}
We introduce
\beno
&&\omega_{H}^{(0)}(t,k,y)=e^{-itky}\omega(0,k,y),\quad t\in \mathbb{R},\\
&&\omega_{H}^{(1)}(t,k,y)=e^{-(\nu k^2)^{1/3}t}\omega_{H}^{(0)}(t,k,y),\quad t>0.
\eeno
Then we have
\beno
&&(\partial_t+iky)\omega_{H}^{(0)}=0,\quad (\partial_t+iky+(\nu k^2)^{1/3})\omega_{H}^{(1)}=0,\\
&&\omega_{H}^{(0)}(0)=\omega_{H}^{(1)}(0)=\omega(0),\quad \|\omega_{H}^{(0)}(t)\|_{L^2}=\|\omega_{H}(0)\|_{L^2},\\
&&(\partial_t-\nu(\partial_y^2-k^2)+iky)\omega_{H}^{(1)}=(\nu k^2-(\nu k^2)^{\f13})\omega_{H}^{(1)}-\nu\partial_y^2\omega_{H}^{(1)}.
\eeno
Thus, we can decompose $\omega_{H}$ as follows
\beno
\omega_{H}=\omega_{H}^{(1)}+\omega_{H}^{(2)}+\omega_{H}^{(3)},
\eeno
where $\omega_{H}^{(2)}$ solves
\begin{align*}
&(\partial_t-\nu(\partial_y^2-k^2)+iky)\omega_{H}^{(2)}=-(\nu k^2-(\nu k^2)^{\f13})\omega_{H}^{(1)}+\nu\partial_y^2\omega_{H}^{(1)},\\
&\omega_{H}^{(2)}|_{t=0}=0,\quad \langle\omega_{H}^{(2)},e^{\pm ky}\rangle=0.
\end{align*}
and $\omega_{H}^{(3)}$ solves
\begin{align*}
(\partial_t-\nu(\partial_y^2-k^2)+iky)\omega_{H}^{(3)}=0, \ \omega_{H}^{(3)}|_{t=0}=0,\ \langle\omega_{H}^{(3)}(t)+\omega_{H}^{(1)}(t),e^{\pm ky}\rangle=0.
\end{align*}

We denote
\beno
u_H^{(j)}=({\partial_y,-ik})\varphi_H^{(j)},\quad \varphi_H^{(j)}=(\partial_y^2-k^2)^{-1}\omega_{H}^{(j)},\quad  j=0,1,2,3.
\eeno

{\bf Step 1.} Estimates of $\omega_{H}^{(j)}, j=1,2$.\smallskip

By Proposition \ref{prop:SPT-inhom}, we have
\begin{align*}
&(\nu k^2)^{\f13}\|\rho_k^{\f12}\omega_{H}^{(2)}\|_{L^2L^2}^2+k^2\|u_H^{(2)}\|_{L^2L^2}^2+(\nu k^2)^{\f12}\|\omega_{H}^{(2)}\|_{L^2L^2}^2+(\nu k^2)^{\f12}\|\omega_{H}^{(2)}\|_{L^{\infty}L^2}^2\\
\nonumber&\leq C\big(\nu^{-\f13}|k|^{-\f23}\|(\nu k^2-(\nu k^2)^{\f13})\omega_{H}^{(1)}\|_{L^2L^2}^2+\nu\|\partial_y\omega_{H}^{(1)}\|_{L^2L^2}^2\big)\\
&\leq C\big((\nu k^2)^{\f13}\|\omega_{H}^{(1)}\|_{L^2L^2}^2+\nu\|\partial_y\omega_{H}^{(1)}\|_{L^2L^2}^2\big).
\end{align*}
It is easy to see that
\begin{align*}
&\|\omega_{H}^{(1)}(t)\|_{L^2}=\|e^{-(\nu k^2)^{1/3}t}\omega_{H}^{(0)}(t)\|_{L^2}=e^{-(\nu k^2)^{1/3}t}\|\omega_0\|_{L^2},\\&(\nu k^2)^{\f13}\|\omega_{H}^{(1)}\|_{L^2L^2}^2=(\nu k^2)^{\f13}\|e^{-(\nu k^2)^{1/3}t}\|_{L^2(0,+\infty)}^2\|\omega_0\|_{L^2}^2=\|\omega_0\|_{L^2}^2/2,
\end{align*}and
\begin{align*}
\|\partial_y\omega_{H}^{(1)}(t)\|_{L^2}=&\|\partial_y(e^{-(\nu k^2)^{1/3}t-itky}\omega_0(k,y))\|_{L^2}=e^{-(\nu k^2)^{1/3}t}\|e^{-itky}(\partial_y-itk)\omega_0(k,y)\|_{L^2}\\
\leq& e^{-(\nu k^2)^{1/3}t}\big(\|\partial_y\omega_0\|_{L^2}+|tk|\|\omega_0\|_{L^2}\big),
\end{align*}
and
\begin{align*}
\nu\|\partial_y\omega_{H}^{(1)}\|_{L^2L^2}^2&\leq2\nu\|e^{-(\nu k^2)^{1/3}t}\|_{L^2(0,+\infty)}^2\|\partial_y\omega_0\|_{L^2}^2+2\nu\|tke^{-(\nu k^2)^{1/3}t}\|_{L^2(0,+\infty)}^2\|\omega_0\|_{L^2}^2\\&=\nu^{\f13}|k|^{-\f43}\|\partial_y\omega_0\|_{L^2}^2+2\|\omega_0\|_{L^2}^2.
\end{align*}
This shows that
\begin{align}\label{eq:omH2}
&(\nu k^2)^{\f13}\|\rho_k^{\f12}\omega_{H}^{(2)}\|_{L^2L^2}^2+k^2\|u_H^{(2)}\|_{L^2L^2}^2+(\nu k^2)^{\f12}\|\omega_{H}^{(2)}\|_{L^2L^2}^2+(\nu k^2)^{\f12}\|\omega_{H}^{(2)}\|_{L^{\infty}L^2}^2\\
\nonumber&\leq C\|\omega(0)\|_{L^2}^2+C\nu^{\f13}|k|^{-\f43}\|\partial_y\omega(0)\|_{L^2}^2,
\end{align}
and
\begin{align}\label{eq:omH1}
&(\nu k^2)^{\f13}\|\rho_k^{\f12}\omega_{H}^{(1)}\|_{L^2L^2}^2+(\nu k^2)^{\f12}\|\omega_{H}^{(1)}\|_{L^2L^2}^2+(\nu k^2)^{\f12}\|\omega_{H}^{(1)}\|_{L^{\infty}L^2}^2\\
\nonumber&\leq 2(\nu k^2)^{\f13}\|\omega_{H}^{(1)}\|_{L^2L^2}^2+\|\omega_{H}^{(1)}\|_{L^{\infty}L^2}^2\leq C\|\omega(0)\|_{L^2}^2.
\end{align}

We use the basis in $L^2([-1,1]): \varphi_j(y)=\sin(\pi j(y+1)/2),\ j\in\mathbb{Z}_+.$ Then we have\begin{align*}
&\om_{H}^{(0)}=\sum_{j=1}^{+\infty}\langle \omega_{H}^{(0)},\varphi_j\rangle\varphi_j,\ \|\om_{H}^{(0)}\|_{L^2}^2=\sum_{j=1}^{+\infty}|\langle \omega_{H}^{(0)},\varphi_j\rangle|^2,\ \varphi_{H}^{(0)}=-\sum_{j=1}^{+\infty}\frac{\langle \omega_{H}^{(0)},\varphi_j\rangle}{(\pi j/2)^2+k^2}\varphi_j,\\ &\|u_{H}^{(0)}(t)\|_{L^2}^2=\|\partial_y\varphi_{H}^{(0)}(t)\|_{L^2}^2+k^2\|\varphi_{H}^{(0)}(t)\|_{L^2}^2=-\langle \varphi_{H}^{(0)}(t),\om_{H}^{(0)}(t)\rangle=\sum_{j=1}^{+\infty}\frac{|\langle \omega_{H}^{(0)}(t),\varphi_j\rangle|^2}{(\pi j/2)^2+k^2}.
\end{align*}
Thanks to $\omega_{H}^{(0)}(t,k,y)=e^{-itky}\omega_0(k,y)$, we have
\begin{align*}
&\langle \omega_{H}^{(0)}(t),\varphi_j\rangle=\int_{-1}^1e^{-i tky}\omega_0(k,y)\varphi_j(y)dy,
\end{align*}
from which and Plancherel's formula, we infer that
\begin{align*}
\int_{\mathbb{R}}|\langle \omega_{H}^{(0)}(t),\varphi_j\rangle|^2dt&=\frac{2\pi}{|k|}\int_{-1}^{1}\left| \omega_0(k,y)\varphi_j(y)\right|^2dz\leq \frac{2\pi}{|k|}\|\om_0\|_{L^2}^2.
\end{align*}
Therefore, we have
\begin{align*}
\int_{\mathbb{R}}\|u_{H}^{(0)}(t)\|_{L^2}^2dt=&\sum_{j=1}^{+\infty}\int_{\mathbb{R}}\frac{|\langle \omega_{H}^{(0)}(t),\varphi_j\rangle|^2}{(\pi j/2)^2+k^2}dt\leq\sum_{j=1}^{+\infty}\frac{2\pi}{|k|}\frac{\|\om_0\|_{L^2}^2}{(\pi j/2)^2+k^2}\\ \leq&\int_{0}^{+\infty}\frac{2\pi}{|k|}\frac{\|\om_0\|_{L^2}^2}{(\pi z/2)^2+|k|^2}dz=\frac{2\pi}{k^2 }\|\om_0\|_{L^2}^2,
\end{align*}
from which and $u_{H}^{(1)}=e^{-(\nu k^2)^{1/3}t}u_{H}^{(0)}$, we infer that
\begin{align}\label{eq:uH1}
&k^2\|u_H^{(1)}\|_{L^2L^2}^2\leq k^2\|u_H^{(0)}\|_{L^2L^2}^2\leq C\|\omega_0\|_{L^2}^2.
\end{align}

{\bf Step 2.} Estimates of $\omega_H^{(3)}$.\smallskip

We introduce
\beno
&&a_1(t)=\int_{-1}^1\f{\sinh k(y+1)}{\sinh 2k}\omega_{H}^{(1)}(t,k,y)dy=-\int_{-1}^1\f{\sinh k(y+1)}{\sinh 2k}\omega_{H}^{(3)}(t,k,y)dy,\\
&&a_2(t)=\int_{-1}^1\f{\sinh k(1-y)}{\sinh 2k}\omega_{H}^{(1)}(t,k,y)dy=-\int_{-1}^1\f{\sinh k(1-y)}{\sinh 2k}\omega_{H}^{(3)}(t,k,y)dy,\\
&&a_1^{(0)}(t)=\int_{-1}^1\f{\sinh k(y+1)}{\sinh 2k}\omega_{H}^{(0)}(t,k,y)dy,\quad a_2^{(0)}(t)=\int_{-1}^1\f{\sinh k(1-y)}{\sinh 2k}\omega_{H}^{(1)}(t,k,y)dy.
\eeno
Then we have
\beno
a_1(t)=e^{-(\nu k^2)^{1/3}t}a_1^{(0)}(t),\quad a_2(t)=e^{-(\nu k^2)^{1/3}t}a_2^{(0)}(t).
\eeno
We take the Fourier transform in $t$:
\beno
w(\lambda,k,y):=\int_{0}^{+\infty}\omega_{H}^{(3)}(t,k,y)e^{-it\lambda}dt,\quad c_j(\lambda):=\int_{0}^{+\infty}a_{j}(t)e^{-it\lambda}dt,\ j=1,2.
\eeno
Then we have
\begin{align*}
&(i\lambda -\nu(\partial_y^2-k^2)+iky)w(\lambda,k,y)=0,\\
&c_1(\lambda)=-\int_{-1}^1\f{\sinh k(1+y)}{\sinh 2k}w(\lambda,k,y)dy,\quad c_2(\lambda)=-\int_{-1}^1\f{\sinh k(1-y)}{\sinh 2k}w(\lambda,k,y)dy.
\end{align*}
Thus, we have
\beno
w=-c_1(\lambda)w_1-c_2(\lambda)w_2,
\eeno
where $w_1,w_2$ are defined as in section 4.1. Let us first claim that
\begin{align}\label{eq:c1-est}
&\|(1+|\lambda+k|)c_1\|_{L^2(\mathbb{R})}^2\leq C|k|^{-1}\|\om_0\|_{L^2}^2,\\
&\|(1+|\lambda-k|)c_2\|_{L^2(\mathbb{R})}^2\leq C|k|^{-1}\|\om_0\|_{L^2}^2.\label{eq:c2-est}
\end{align}

By Proposition \ref{prop:w12-bounds}, we know that
\beno
&&\|w_1\|_{L^2}\leq C\nu^{-\f14}(1+|k||-\lambda/k-1|)^{\f14}=C\nu^{-\f14}(1+|\lambda+k|)^{\f14},\\
&&\|w_2\|_{L^2}\leq C\nu^{-\f14}(1+|k||-\lambda/k+1|)^{\f14}=C\nu^{-\f14}(1+|\lambda-k|)^{\f14},
\eeno
from which, we infer that
\begin{align*}
\|w(\lambda)\|_{L^2}&=\|c_1(\lambda)w_1+c_2(\lambda)w_2\|_{L^2}\leq|c_1(\lambda)\|w_1\|_{L^2}+|c_2(\lambda)|\|w_2\|_{L^2}\\&\leq C\nu^{-\f14}\big(|c_1(\lambda)|(1+|\lambda+k|)^{\f14}+|c_2(\lambda)|(1+|\lambda-k|)^{\f14}\big).
\end{align*}
By Proposition \ref{prop:w12-bounds}, we have $\|w_1\|_{L^1}+\|w_2\|_{L^1}\leq C$, which along with Lemma \ref{lem:phi-w} implies
\begin{align*}
&\|(k,\partial_y)(\partial_y^2-k^2)^{-1}w(\lambda)\|_{L^2}\\ &\leq
|c_1(\lambda)|\|(k,\partial_y)(\partial_y^2-k^2)^{-1}w_1\|_{L^2}+
|c_2(\lambda)|\|(k,\partial_y)(\partial_y^2-k^2)^{-1}w_2\|_{L^2}\\ &\leq
|k|^{-\f12}\big(|c_1(\lambda)|\|w_1\|_{L^1}+
|c_2(\lambda)|\|w_2\|_{L^1}\big)\leq {C}|k|^{-\f12}\big(|c_1(\lambda)|+|c_2(\lambda)|\big).
\end{align*}
By Proposition \ref{prop:w12-bounds-weight}, we have $\|\rho_k^{\f12}w_1\|_{L^2}+\|\rho_k^{\f12}w_2\|_{L^2}\leq CL^{\f12}$ with $L=(k/\nu)^{\f13}$, which gives
\begin{align*}
\|\rho_k^{\f12}w(\lambda)\|_{L^2}&=\|c_1(\lambda)\rho_k^{\f12}w_1+c_2(\lambda)\rho_k^{\f12}w_2\|_{L^2}
\leq|c_1(\lambda)\|\rho_k^{\f12}w_1\|_{L^2}+|c_2(\lambda)|\|\rho_k^{\f12}w_2\|_{L^2}\\&\leq CL^{\f12}\big(|c_1(\lambda)|+|c_2(\lambda)|\big)= C(k/\nu)^{\f16}\big(|c_1(\lambda)|+|c_2(\lambda)|\big).
\end{align*}

Summing up, we conclude that
\begin{align*}
 &(\nu k^2)^{\f13}\|\rho_k^{\f12}\omega_{H}^{(3)}\|_{L^2L^2}^2+k^2\|u_{H}^{(3)}\|_{L^2L^2}^2+(\nu k^2)^{\f12}\|\omega_{H}^{(3)}\|_{L^2L^2}^2\\
  &\quad\sim(\nu k^2)^{\f13}\big\|\|\rho_k^{\f12}w(\lambda)\|_{L^2}\big\|_{L^2(\mathbb{R})}^2+
 k^2\big\|\|(k,\partial_y)(\partial_y^2-k^2)^{-1}w(\lambda)\|_{L^2}\big\|_{L^2(\mathbb{R})}^2+(\nu k^2)^{\f12}\big\|\|w(\lambda)\|_{L^2}\big\|_{L^2(\mathbb{R})}^2\\ &\leq C(\nu k^2)^{\f13}(k/\nu)^{\f13}\big(\|c_1\|_{L^2(\mathbb{R})}^2+\|c_2\|_{L^2(\mathbb{R})}^2\big)
 +C k^2|k|^{-1}(\|c_1\|_{L^2(\mathbb{R})}^2+\|c_2\|_{L^2(\mathbb{R})}^2)\\&\quad
+ C(\nu k^2)^{\f12}\nu^{-\f{1}{2}}\big(\|(1+|\lambda+k|)^{\f14}c_1\|_{L^2(\mathbb{R})}^2+\|(1+|\lambda-k|)^{\f14}c_2\|_{L^2(\mathbb{R})}^2\big)\\ &\leq
 C|k|(\|(1+|\lambda+k|)^{\f14}c_1\|_{L^2(\mathbb{R})}^2+\|(1+|\lambda-k|)^{\f14}c_2\|_{L^2(\mathbb{R})}^2).
\end{align*}

Thanks to $\omega_{H}^{(3)}(t)=\frac{1}{2\pi}\int_{\mathbb{R}}w(\lambda)e^{it\lambda}d\lambda$, we also have
\begin{align*}
 &(\nu k^2)^{\f12}\|\omega_{H}^{(3)}\|_{L^{\infty}L^2}^2 \leq(\nu k^2)^{\f12}\|\|w(\lambda)\|_{L^2}\|_{L^1(\mathbb{R})}^2\\ &\leq
 C(\nu k^2)^{\f12}\nu^{-\f{1}{2}}\big(\|(1+|\lambda+k|)^{\f14}c_1\|_{L^1(\mathbb{R})}^2+\|(1+|\lambda-k|)^{\f14}c_2\|_{L^1(\mathbb{R})}^2\big).
\end{align*}
Then it follows from \eqref{eq:c1-est} and \eqref{eq:c2-est} that
\begin{align}
 &(\nu k^2)^{\f13}\|\rho_k^{\f12}\omega_{H}^{(3)}\|_{L^2L^2}^2+k^2\|u_{H}^{(3)}\|_{L^2L^2}^2+(\nu k^2)^{\f12}\|\omega_{H}^{(3)}\|_{L^2L^2}^2+(\nu k^2)^{\f12}\|\omega_{H}^{(3)}\|_{L^{\infty}L^2}^2\nonumber\\ \nonumber&\leq
 C|k|\big(\|(1+|\lambda+k|)^{\f14}c_1\|_{L^2(\mathbb{R})}^2+\|(1+|\lambda-k|)^{\f14}c_2\|_{L^2(\mathbb{R})}^2\big)\\ \nonumber&\quad+C(\nu k^2)^{\f12}\nu^{-\f{1}{2}}\big(\|(1+|\lambda+k|)^{\f14}c_1\|_{L^1(\mathbb{R})}^2+\|(1+|\lambda-k|)^{\f14}c_2\|_{L^1(\mathbb{R})}^2\big)\\
 &\leq
 C|k|\big(\|(1+|\lambda+k|)c_1\|_{L^2(\mathbb{R})}^2+\|(1+|\lambda-k|)c_2\|_{L^2(\mathbb{R})}^2\big)\le C\|\om_0\|_{L^2}^2.\label{eq:omH3}
\end{align}

{\bf Step 3.} Proof of \eqref{eq:c1-est} and \eqref{eq:c2-est}.\smallskip

Let us assume that $\langle\om_0,e^{\pm ky}\rangle=0$. Then we have $a_1(0)=a_2(0)=0. $ Thanks to $\omega_{H}^{(0)}(t,k,y)=e^{-itky}\omega_0(k,y)$, we get
\begin{align*}
&a_1^{(0)}(t)=\int_{-1}^1\f{\sinh k(y+1)}{\sinh 2k}\omega_{H}^{(0)}(t,k,y)dy=\int_{-1}^1\f{\sinh k(y+1)}{\sinh 2k}e^{-itky}\omega_0(k,y)dy,
\end{align*}from which and Plancherel's formula, we infer that
\begin{align*}
\int_{\mathbb{R}}|a_1^{(0)}(t)|^2dt&=\frac{2\pi}{|k|}\int_{-1}^{1}\left| \f{\sinh k(y+1)}{\sinh 2k}\omega_0(k,y)\right|^2dz\leq \frac{2\pi}{|k|}\|\om_0\|_{L^2}^2.
\end{align*}
Using the formula
\begin{align*}
&\partial_ta_1^{(0)}(t)+ika_1^{(0)}(t)=\int_{-1}^1ik(1-y)\f{\sinh k(y+1)}{\sinh 2k}e^{-itky}\omega(0,k,y)dy,
\end{align*}
we get
\begin{align*}
\int_{\mathbb{R}}|\partial_ta_1^{(0)}(t)+ika_1^{(0)}(t)|^2dt&=\frac{2\pi}{|k|}\int_{-1}^{1}\left| ik(1-y)\f{\sinh k(y+1)}{\sinh 2k}\omega_0(k,y)\right|^2dz\leq \frac{2\pi}{|k|}\|\om_0\|_{L^2}^2,
\end{align*}
here we used $\left|ik(1-y)\f{\sinh k(y+1)}{\sinh 2k}\right| \leq k(1-y)(2e^{-k(1-y)})\leq1$ for $y\in[-1,1].$
As $a_1(t)=e^{-(\nu k^2)^{1/3}t}a_1^{(0)}(t),$ we have
\begin{align*}
\|e^{ikt}a_1(t)\|_{L^2(0,+\infty)}^2=\|e^{ikt-(\nu k^2)^{1/3}t}a_1^{(0)}(t)\|_{L^2(0,+\infty)}^2\leq C|k|^{-1}\|\om_0\|_{L^2}^2,
\end{align*}
and (using $\nu k^2\leq 1$)\begin{align*}
\|\partial_t(e^{ikt}a_1(t))\|_{L^2(0,+\infty)}^2&=\|e^{ikt-(\nu k^2)^{1/3}t}(\partial_ta_1^{(0)}(t)+ika_1^{(0)}(t)-(\nu k^2)^{1/3}a_1^{(0)}(t))\|_{L^2(0,+\infty)}^2\\ &\leq 2\|\partial_ta_1^{(0)}(t)+ika_1^{(0)}(t)\|_{L^2}^2+2\|(\nu k^2)^{1/3}a_1^{(0)}(t)\|_{L^2}^2\\
&\leq C|k|^{-1}\|\om_0\|_{L^2}^2.
\end{align*}

We define $\widetilde{a}_1(t)=e^{ikt}a_1(t)$ for $t\geq0$ and $\widetilde{a}_1(t)=0$ for $t\leq0$. Due to $a_1(0)=0,$ we have $\widetilde{a}_1(t)\in H^1(\mathbb{R})$ and $\|\widetilde{a}_1(t)\|_ {H^1(\mathbb{R})}^2\leq C|k|^{-1}\|\om_0\|_{L^2}^2$. Moreover,
\beno
\int_{\mathbb{R}}\widetilde{a}_1(t)e^{-it\lambda}dt=\int_{0}^{+\infty}a_{1}(t)e^{itk-it\lambda}dt=c_1(\lambda-k),
\eeno
which gives
\begin{align*}
\|\widetilde{a}_1(t)\|_ {H^1(\mathbb{R})}^2\sim\int_{\mathbb{R}}(1+|\lambda|)^2|c_1(\lambda-k)|^2d\lambda=\|(1+|\lambda+k|)c_1\|_{L^2(\mathbb{R})}^2 .
\end{align*}
Thus, we obtain
\begin{align*}
\|(1+|\lambda+k|)c_1\|_{L^2}^2\leq C|k|^{-1}\|\om_0\|_{L^2}^2.
\end{align*}
Similarly, we can prove \eqref{eq:c2-est}.\smallskip

 Now the result follows from \eqref{eq:omH2}, \eqref{eq:omH1}, \eqref{eq:uH1} and \eqref{eq:omH3}.
\end{proof}

\subsubsection{Space-time estimates of the full problem}

\begin{proposition}\label{prop:SPT-small}
Let $\nu k^2\leq1$ and $\om$ be a solution of \eqref{eq:vorticity-LNS}.
Then there exists a constant $C>0$ independent of $\nu,k$ such that
\begin{align*}
&\|\rho_k^{\f32}\omega\|_{L^{\infty}L^2}^2+(\nu k^2)^{\f12}\|\omega\|_{L^2L^2}^2+(\nu k^2)^{\f12}\|\omega\|_{L^{\infty}L^2}^2+(\nu k^2)^{\f13}\|\rho_k^{\f12}\omega\|_{L^2L^2}^2+k^2\|u\|_{L^2L^2}^2\\
\nonumber&\leq C\|\omega_0\|_{L^2}^2+C\nu^{\f13}|k|^{-\f43}\|\partial_y\omega_0\|_{L^2}^2+
C\big(\nu^{-\f13}|k|^{\f43}\|f_1\|_{L^2L^2}^2+\nu^{-1}\|f_2\|_{L^2L^2}^2\big).
\end{align*}
\end{proposition}

\begin{proof}
It follows from Proposition \ref{prop:SPT-inhom} and Proposition \ref{prop:SPT-hom} that
\begin{align}\label{eq:om1}
&(\nu k^2)^{\f12}\|\omega\|_{L^2L^2}^2+(\nu k^2)^{\f12}\|\omega\|_{L^{\infty}L^2}^2+(\nu k^2)^{\f13}\|\rho_k^{\f12}\omega\|_{L^2L^2}^2+k^2\|u\|_{L^2L^2}^2\\
\nonumber&\leq C\|\omega_0\|_{L^2}^2+C\nu^{\f13}|k|^{-\f43}\|\partial_y\omega_0\|_{L^2}^2+
C\big(\nu^{-\f13}|k|^{\f43}\|f_1\|_{L^2L^2}^2+\nu^{-1}\|f_2\|_{L^2L^2}^2\big).
\end{align}

It remains to estimate $\|\rho_k^{\f32}\omega\|_{L^{\infty}L^2}^2.$ For this, we introduce a new weight function $\tilde{\rho}_k$ defined as follows
  \begin{align}
\tilde{\rho}_k(y)=
\left\{
\begin{aligned}
&(Ly+L-1)^3+1\qquad y\in[-1,-1+L^{-1}],\\
&1\qquad\quad\qquad\qquad\qquad y\in[-1+L^{-1},1-L^{-1}],\\
&(L-Ly-1)^3+1\qquad y\in[1-L^{-1},1],
\end{aligned}
\right.\label{tilde rho_k}
\end{align}
where $L=(\f{|k|}{\nu})^{\f13}$. It is easy to see that
\beno
\tilde{\rho}_k\in C^2(-1,1),\quad  |\tilde{\rho_k}'|\lesssim L, \quad |\tilde{\rho_k}''|\lesssim L^2.
\eeno
 And there exits a constant number $C$, independent of $\nu,k$, so that $C^{-1}\rho_k\leq\tilde{\rho_k}\leq C\rho_k$. We also have $\nu L^2=(\nu k^2)^{\f13}. $
Recall that $\om$ satisfies \eqref{eq:vorticity-LNS}. By integration by parts, we get
\begin{align*}
&\f12\f{d}{dt}\|\tilde{\rho}_k^{\f32}\omega\|_{L^2}^2+ \nu\big(\|\tilde{\rho}_k^{\f32}\omega'\|^2_{L^2}+k^2\|\tilde{\rho}_k^{\f32}\omega\|^2_{L^2}\big)={\rm Re}\big(-\nu\int_{-1}^{1}\omega'\bar{\omega}(\tilde{\rho}_k^{3})'dy\\
  &\qquad-ik\big\langle f_1,\tilde{\rho}_k^3\omega\rangle+\big\langle f_2,\partial_y(\tilde{\rho}_k^3\omega)\rangle\big)\\ &\leq \nu\big|\int_{-1}^{1}|\omega|^2(\tilde{\rho}_k^{3})''dy\big|
  +|k|\|f_1\|_{L^2}\|\tilde{\rho}_k^3\omega\|_{L^2}+\|f_2\|_{L^2}(3\|\tilde{\rho}_k'\|_{L^\infty}\|\tilde{\rho}_k^2\omega\|_{L^2}
  +\|\tilde{\rho}_k^3\omega'\|_{L^2})\\ &\leq\nu(6\|\tilde{\rho}_k'\|^2_{L^\infty}+3\|\tilde{\rho}_k''\|_{L^\infty})\int_{-1}^{1}|\omega|^2\tilde{\rho}_k dy+\f12\nu^{-\f13}|k|^{\f43}\|f_1\|_{L^2}^2+\f12(\nu k^2)^{\f13}\|\tilde{\rho}_k^{3}\omega\|^2_{L^2}\\&\quad+\nu^{-1}\|f_2\|_{L^2}^2
  +3\nu\|\tilde{\rho}_k'\|_{L^\infty}^2\|\tilde{\rho}_k^{2}\omega\|_{L^2}^2+\nu\|\tilde{\rho}_k^{3}\omega'\|_{L^2}^2,
\end{align*}
which shows that
\begin{align*}
  \f{d}{dt}\|\tilde{\rho}_k^{\f32}\omega\|_{L^2}^2 \lesssim&\nu L^2\|\tilde{\rho}_k^{\f12}\omega\|^2_{L^2}+\nu^{-\f13}|k|^{\f43}\|f_1\|_{L^2}^2+(\nu k^2)^{\f13}\|\tilde{\rho}_k^{3}\omega\|^2_{L^2}+\nu^{-1}\|f_2\|_{L^2}^2\\ \lesssim&\nu^{-\f13}|k|^{\f43}\|f_1\|_{L^2}^2+(\nu k^2)^{\f13}\|\rho_k^{\f12}\omega\|^2_{L^2}+\nu^{-1}\|f_2\|_{L^2}^2,
\end{align*}
from which and \eqref{eq:om1}, we infer that
\begin{align*}
\|\rho_k^{\f32}\omega\|_{L^{\infty}L^2}^2\leq C\|\omega_0\|_{L^2}^2+C\nu^{\f13}|k|^{-\f43}\|\partial_y\omega_0\|_{L^2}^2+
C\big(\nu^{-\f13}|k|^{\f43}\|f_1\|_{L^2L^2}^2+\nu^{-1}\|f_2\|_{L^2L^2}^2\big).
\end{align*}

This completes the proof.
\end{proof}

\subsection{Proof of Proposition \ref{prop:sapce-times}}

\begin{proof}For the case of $\nu k^2\leq 1,$ we have
\beno
\nu^{\f13}|k|^{-\f43}=k^{-2}(\nu k^2)^{\f13},\quad \nu^{-\f13}|k|^{\f43}=\nu^{-\f12}|k|(\nu k^2)^{\f16}\leq\nu^{-\f12}|k|,
\eeno
which along with Proposition \ref{prop:SPT-small} imply that
\begin{align}\label{omL2a}
&k^2\|u\|_{L^2L^2}^2+(\nu k^2)^{\f12}\|\omega\|_{L^2L^2}^2+(\nu k^2)^{\f12}\|\omega\|_{L^{\infty}L^2}^2+\|\rho_k^{\f32}\omega\|_{L^{\infty}L^2}^2\\ \nonumber&\leq C(\|\omega_0\|_{L^2}^2+k^{-2}\|\partial_y\omega_0\|_{L^2}^2)+C\big(\nu^{-\f12}|k|\|f_1\|_{L^2L^2}^2+\nu^{-1}\|f_2\|_{L^2L^2}^2\big).
\end{align}

Due to the definition of $\rho_k$, we know that $\rho_k^{\f32}(y)=1\geq (1-|y|)^{\f12} $ for $|y|\leq 1-L^{-1},$ $\rho_k^{\f32}(y)=L^{\f32}(1-|y|)^{\f32}\geq L^{\f32}\nu^{\f12}(1-|y|)^{\f12}=|k|^{\f12}(1-|y|)^{\f12}\geq (1-|y|)^{\f12} $ for $1-L^{-1}\leq|y|\leq 1-\nu^{\f12},$ where $L=(|k|/\nu)^{\f13}\leq \nu^{-\f12}$ and $(\nu k^2)^{\f12}\geq \nu^{\f12}\geq 1-|y| $ for $1-\nu^{\f12}\leq|y|\leq 1$. Thus,
\begin{align}\label{omL2}
&\|(1-|y|)^{\f12}\omega\|_{L^{\infty}L^2}^2\leq(\nu k^2)^{\f12}\|\omega\|_{L^{\infty}L^2}^2+\|\rho_k^{\f32}\omega\|_{L^{\infty}L^2}^2.
\end{align}

Since $0\leq1-\rho_k^{\f32}\leq1$ for $ |y|\leq 1$ and $1-\rho_k^{\f32}=0 $ for $|y|\leq 1-L^{-1},$ we have
\beno
\|(1-\rho_k^{\f32})\om\|_{L^1}\leq \|\om\|_{L^1(1-L^{-1},1)}+\|\om\|_{L^1(-1,-1+L^{-1})}.
\eeno
Notice that
\begin{align*}
\|\rho_k^{-\f32}\|_{L^2(1-L^{-1},1-\nu^{\f12})}^2&=\int_{1-L^{-1}}^{1-\nu^{\f12}}L^{-3}(1-y)^{-3}dy\\
&=\f12L^{-3}(1-y)^{-2}\big|_{1-L^{-1}}^{1-\nu^{\f12}}\leq \f12L^{-3}\nu^{-1}=\f12|k|^{-1}.
\end{align*}
Thus, we have
\begin{align*}
\|\om\|_{L^1(1-L^{-1},1)}=&\|\om\|_{L^1(1-L^{-1},1-\nu^{\f12})}+\|\om\|_{L^1(1-\nu^{\f12},1)}\\ \leq& \|\rho_k^{\f32}\omega\|_{L^2}\|\rho_k^{-\f32}\|_{L^2(1-L^{-1},1-\nu^{\f12})}+ \nu^{\f14}\|\omega\|_{L^2}\\
 \leq& 2|k|^{-\f12}\|\rho_k^{\f32}\omega\|_{L^2}+ \nu^{\f14}\|\omega\|_{L^2}.
\end{align*}
Similarly, we have
\begin{align*}
\|\om\|_{L^1(-1,-1+L^{-1})} \leq2|k|^{-\f12}\|\rho_k^{\f32}\omega\|_{L^2}+ \nu^{\f14}\|\omega\|_{L^2}.
\end{align*}
This shows that
\begin{align*}
&\|(1-\rho_k^{\f32})\om\|_{L^1}\leq \|\om\|_{L^1(1-L^{-1},1)}+\|\om\|_{L^1(-1,-1+L^{-1})} \leq C\big(|k|^{-\f12}\|\rho_k^{\f32}\omega\|_{L^2}+ \nu^{\f14}\|\omega\|_{L^2}\big).
\end{align*}
By Lemma \ref{lem:phi-w}, we have
\begin{align*}
\|u\|_{L^\infty}\leq&\|(\partial_y,k)(\partial_y^2-k^2)^{-1}\rho_k^{\f32}\om\|_{L^\infty }+\|(\partial_y,k)(\partial_y^2-k^2)^{-1}(1-\rho_k^{\f32})\om\|_{L^\infty}\\ \leq&C\big(|k|^{-\f12}\|\rho_k^{\f32}\omega\|_{L^2}+\|(1-\rho_k^{\f32})\om\|_{L^1}\big)\\
 \leq& C\big(|k|^{-\f12}\|\rho_k^{\f32}\omega\|_{L^2}+ \nu^{\f14}\|\omega\|_{L^2}\big),
\end{align*}
which gives
\begin{align}\label{omu}
&|k|\|u\|_{L^\infty L^\infty}^2\leq C\|\rho_k^{\f32}\omega\|_{L^\infty L^2}^2+ C|k|\nu^{\f12}\|\omega\|_{L^\infty L^2}^2.
\end{align}

Then the desired result follows from \eqref{omL2a}, \eqref{omL2} and \eqref{omu}.\smallskip

For the case of $\nu k^2\geq 1,$  we have
$|k|\|u\|_{L^\infty L^\infty}^2\leq C\|\omega\|_{L^\infty L^2}^2$, and then the result follows from Proposition \ref{prop:SPT-big} and the facts that $(\nu k^2)^{\f12}\leq \nu k^2,\ \nu^{-1}=\nu^{-\f12}|k|{(\nu k^2)^{-\f12}}\leq\nu^{-\f12}|k|.$
\end{proof}

\section{Nonlinear stability}

In this section, we prove Theorem \ref{thm:stability}.
For the 2-D Navier-Stokes equation, the global existence of smooth solution is well-known for the data $u_0\in H^2(\Om)$.
Main interest of Theorem \ref{thm:stability} is the stability estimate: $\sum_{k\in\Z}E_k\le Cc\nu^\f12$. Let us recall that $E_0=\|\overline{w}\|_{L^{\infty}L^2}$ and for $k\neq 0$,
\begin{align*}
E_k=&\|(1-|y|)^{\f12}w_k\|_{L^{\infty}L^2}+|k|\|u_k\|_{L^2L^2}+|k|^\f12\|u_k\|_{L^{\infty}L^{\infty}}+(\nu k^2)^{\f14}\|w_k\|_{L^2L^2}.
\end{align*}

First of all,  we derive the evolution equations of $\overline{u}(t,y)$ and $w_k(t,y)={\frac{1}{2\pi}}\int_{\T}w(t,x,y)e^{-ikx}dx$.
We denote
\begin{align*}
f^1_k(t,y)=\sum\limits_{l\in\mathbb{Z}}u^1_{l}(t,y)w_{k-l}(t,y),\quad f^2_k(t,y)=\sum\limits_{l\in\mathbb{Z}}u^2_l(t,y)w_{k-l}(t,y).
\end{align*}
Thanks to $\text{div} u=0$, we have $\overline{u}^2(t,y)=0$. Due to $P_0(u^1\partial_x u^1)=0$, we find that
\begin{align*}
(\partial_t-\nu\partial_y^2)\overline{u}^1(t,y)&=-\sum\limits_{l\in\mathbb{Z}\setminus\{0\}}u^2_l(t,y)\pa_yu^1_{-l}(t,y)=
-\sum\limits_{l\in\mathbb{Z}\setminus\{0\}}u^2_l(t,y)w_{-l}(t,y)=-f^2_0(t,y).
\end{align*}
And $w_k(t,y)(k\neq 0)$ satisfies
\begin{align*}
(\partial_t-\nu(\partial_y^2-k^2)+iky)w_k(t,y)=-ikf^1_k(t,y)-\partial_yf^2_k(t,y).
\end{align*}

Next we estimate $E_0$. By integration by parts, we get
\begin{align*}
&\big\langle(\partial_t-\nu\partial_y^2)\overline{u}^1,-\pa_y^2\overline{u}^1\big\rangle=\f12\partial_t\|\pa_y\overline{u}^1(t)\|_{L^2}^2+\nu\|\pa_y^2\overline{u}^1(t)\|_{L^2}^2
=\big\langle \overline{f}^2,\pa_y^2\overline{u}^1\big\rangle,
\end{align*}
which gives
\begin{align*}
\partial_t\|\pa_y\overline{u}^1(t)\|_{L^2}^2+\nu\|\pa_y^2\overline{u}^1(t)\|_{L^2}^2\lesssim\nu^{-1}\|f^2_0(t,y)\|_{L^2}^2,
\end{align*}
from which and $\pa_y\overline{u}^1(t,y)=\overline{w}(t,y)$, we infer that
\begin{align}\label{eq:E0}
E_0^2=\|\overline{w}\|_{L^{\infty}L^2}^2\lesssim\nu^{-1}\|f^2_0\|_{L^2L^2}^2+\|\overline{w}_0\|_{L^2}^2.
\end{align}

Now we estimate $E_k$. It follows from Proposition \ref{prop:sapce-times}  that
\begin{align}\label{eq:Ek}
E_k^2\lesssim &\nu^{-\f12}|k|\|f^1_k\|_{L^2L^2}^2+\nu^{-1}\|f^2_k\|_{L^2L^2}^2+\|w_{0,k}\|_{L^2}^2+
|k|^{-2}\|\omega_{0,k}\|_{L^2}^2.
\end{align}
For $k\neq0$, we have
\begin{align*}
\Big\|\f{u^2_k(t,y)}{(1-|y|)^{\f12}}\Big\|_{L^2L^{\infty}}^2=&\Big\|\sup\limits_{y\in[-1,1]}\f{|u^2_k(t,y)|^2}{1-|y|}\Big\|_{{L^1}}\\
=&\Big\|\max\big\{\sup\limits_{y\in[0,1]}\f{|\int_{1}^y\partial_zu^2_k(t,z)dz|^2}{1-|y|},
\sup\limits_{y\in[-1,0]}\f{|\int_{-1}^y\partial_zu^2_k(t,z)dz|^2}{1-|y|}\big\}\Big\|_{{L^1}}\\
\leq& 4\|\partial_yu^2_k\|_{L^2L^2}^2=4|k|^2\|u^1_k\|_{L^2L^2}^2\leq 4E_k^2.
\end{align*}
from which, we infer that  for $k\in\mathbb{Z}$,
\begin{align}\label{eq:f2}
\|f^2_k\|_{L^2L^2}\leq\sum\limits_{l\in\mathbb{Z}}\Big\|\f{u^2_l(t,y)}{(1-|y|)^{\f12}}\Big\|_{L^2L^{\infty}}
\|(1-|y|)^{\f12}w_{k-l}\|_{L^{\infty}L^2}\leq 2\sum\limits_{l\in\mathbb{Z}}E_lE_{k-l},
\end{align}
and
\begin{align*}
\|f^1_k\|_{L^2L^2}\leq&\|\overline{u}^1\|_{L^{\infty}L^{\infty}}\|w_k\|_{L^2L^2}+\|u^1_k\|_{L^2L^{\infty}}\|\overline{w}\|_{L^{\infty}L^2}\\
&+\sum\limits_{l\in\mathbb{Z}\setminus\{0,k\}}\|u^1_l\|_{L^{\infty}L^{\infty}}\|w_{k-l}\|_{L^2L^2}.
\end{align*}
Thanks to $|l||k-l|\gtrsim |k| (l\neq0,k)$, we have
\begin{align*}
\sum\limits_{l\in\mathbb{Z}\setminus\{0,k\}}\|u^1_l\|_{L^{\infty}L^{\infty}}\|w_{k-l}\|_{L^2L^2}\lesssim&\sum\limits_{l\in\mathbb{Z}\setminus\{0,k\}}|l|^{-\f12}E_l \nu^{-\f14}|k-l|^{-\f12}E_{k-l}\\
\lesssim&|k|^{-\f12}\nu^{-\f14}\sum\limits_{l\in\mathbb{Z}\setminus\{0,k\}}E_lE_{k-l},
\end{align*}
and
\begin{align*}
\|\overline{u}^1\|_{L^{\infty}L^{\infty}}\|w_k\|_{L^2L^2}+\|u^1_k\|_{L^2L^{\infty}}\|\overline{w}\|_{L^{\infty}L^2}
\lesssim \|\overline{w}\|_{L^{\infty}L^2}\|w_k\|_{L^2L^2}\lesssim (\nu k^2)^{-\f14}E_kE_0.
\end{align*}
This shows that
\begin{align}\label{eq:f1}
\|f^1_k\|_{L^2L^2}\lesssim (\nu k^2)^{-\f14}\sum\limits_{l\in\mathbb{Z}}E_lE_{k-l}.
\end{align}

It follows from \eqref{eq:E0}, \eqref{eq:Ek}, \eqref{eq:f1} and \eqref{eq:f2} that
\begin{align*}
&E_k\lesssim\nu^{-\f12}\sum\limits_{l\in\mathbb{Z}}E_lE_{k-l}+\|w_{0,k}\|_{L^2}+|k|^{-1}\|\partial_yw_{0,k}\|_{L^2},\\
&E_0\lesssim\nu^{-\f12}\sum\limits_{l\in\mathbb{Z}\setminus\{0\}}E_lE_{-l}+\|\overline{w}_0\|_{L^2},
\end{align*}
which lead to
\begin{align}
&\sum\limits_{k\in\mathbb{Z}}E_k\lesssim\nu^{-\f12}\sum\limits_{k\in\mathbb{Z}}\sum\limits_{l\in\mathbb{Z}}E_lE_{k-l}+
\sum\limits_{k\in\mathbb{Z}}\|w_{0,k}\|_{L^2}+\sum\limits_{k\in\mathbb{Z}\setminus\{0\}}
|k|^{-1}\|\partial_yw_{0,k}\|_{L^2}.\label{eq:energy}
\end{align}
Due to $\|u_0\|_{H^2}\le c\nu^\f12$, it is easy to verify that
\beno
\sum\limits_{k\in\mathbb{Z}}\|w_{0,k}\|_{L^2}+\sum\limits_{k\in\mathbb{Z}\setminus\{0\}}
|k|^{-1}\|\partial_yw_{0,k}\|_{L^2}\le Cc\nu^\f12.
\eeno
If $c$ is suitably small, then we can deduce from \eqref{eq:energy}  and a continuous argument that 
\begin{align*}
\sum\limits_{k\in\mathbb{Z}}E_k\le Cc\nu^\f12.
\end{align*}

This completes the proof of Theorem \ref{thm:stability}.

\section{Some key estimates related to the Airy function}

\subsection{Basic properties of the Airy function}
Let $Ai(z)$ be the classical Airy function, which satisfies
\beno
\pa_z^2Ai(z)-zAi(z)=0.
\eeno
We have the following asymptotic formula for $|\text{arg} z|\le \pi-\varepsilon, \varepsilon>0$(see \cite{Was, Rom}):
\beno
Ai(z)=\f 1 {2\sqrt{\pi}}z^{-\f14}e^{-\f23 z^{\f 32}}\big(1+R(z)\big),\quad R(z)=O(z^{-\f32}).
\eeno
Thus, we may define
\begin{align*}
&A_0(z)=\int_{e^{i\pi/6}z}^{+\infty}Ai(t)dt=e^{i\pi/6}\int_z^{+\infty}Ai(e^{i\pi/6} t)dt.
\end{align*}

For $A_0(z)$, we have the following important properties from \cite{Rom}.
\begin{lemma}\label{lem:Ao}
It holds that
\begin{itemize}
\item[1.] There exists $\delta_0>0$ so that $A_0(z)$ has no zeros in the half plane ${\rm Im }z\le \delta_0$.

\item[2.] Let $a(\delta)=\sup\Big\{{\rm Re}\Big(\f{A'_0(z)}{A_0(z)}\Big): {\rm Im}z\leq \delta\Big\}$. There exists $\delta_0>0$ so that $a(\delta)\in C([0,\delta_0])$ and
\begin{align*}
a(0)=-0.4843...<-1/3.
\end{align*}

\item[3.] For $|\arg (ze^{\f{i\pi}{6}})|\leq\pi-\e,\ \e>0$, we have the asymptotic formula
\begin{align*}
\f{A_0'(z)}{A_0(z)}=-e^{i\pi/6}\big(ze^{i\pi/6}\big)^{\f12}+O(z^{-1}).
\end{align*}

\end{itemize}
\end{lemma}

Using Lemma \ref{lem:Ao}, we can deduce the following important estimates on $A_0(z)$.

\begin{lemma}\label{lem:Ao-more}

Let $\delta_0$ be as in Lemma \ref{lem:Ao}. There exists $c>0$ so that for ${{\rm Im }z}\le \delta_0$,
\begin{align}
&\Big|\f{A_0'(z)}{A_0(z)}\Big|\lesssim1+|z|^{\f12},\quad {\rm Re}\f{A_0'(z)}{A_0(z)}\leq-c(1+|z|^{\f12}).\nonumber
\end{align}
\end{lemma}

\begin{proof}
For $|z|\ge R_0\gg1$ and ${{\rm Im } z}\le \delta_0$, we use the asymptotic formula in Lemma \ref{lem:Ao}.
For $|z|\le R_0$ and ${{\rm Im } z}\le \delta_0$, we use the facts that
\beno
|A_0(z)|\sim1,\quad \Big|\f{A_0'(z)}{A_0(z)}\Big|\le C,\quad {\rm Re}\f{A_0'(z)}{A_0(z)}\leq-c.
\eeno
\end{proof}

We introduce
\begin{align*}
\omega(z,x)=\frac{A_0(z+x)}{A_{0}(z)}=\exp\Big(\int_{0}^{x}\frac{A'_0(z+t)}{A_0(z+t)}dt\Big).
\end{align*}

\begin{lemma}\label{lem:A0-w}
There exists $\delta_1>0$ so that for ${\rm Im}z\le \delta_1$ and $x\ge 0$,
\begin{align*}
  |\omega(z,x)|\leq e^{-\f{x}{3}}.
\end{align*}
\end{lemma}
\begin{proof}
Thanks to Lemma \ref{lem:Ao}, there $\delta_1>0$ so  that $a(\epsilon)\geq\f{1}{3}$ for $\epsilon\in [0,\delta_1]$. Thus, for ${\rm Im}z\le \delta_1$ and $x\ge 0$, we have
\begin{align*}
 |\omega(z,x)|&=\Big|\exp\Big(\int_{0}^{x}\frac{A'_0(z+t)}{A_0(z+t)}dt\Big)\Big|=\Big|\exp\Big({\rm Re}\int_{0}^{x}\frac{A'_0(z+t)}{A_0(z+t)}dt\Big)\Big|
 \le e^{-\f x 3}.
\end{align*}
\end{proof}

\subsection{Estimates of $W_1, W_2$}
In this subsection, we prove Lemma \ref{lem:W12}.
Let us recall that
\beno
W_1(y)=Ai\big(e^{i\f{\pi}{6}}(L(y-\lambda-i k\nu)+i\epsilon)\big),\quad W_2(y)=Ai\big(e^{i\f{5\pi}{6}}(L(y-\lambda-i k\nu)+i\epsilon)\big),
\eeno
where $L=(\f{k}{\nu})^{\f13}$ and $\epsilon>0$. We need the following technical lemma.

\begin{lemma}\label{lem:Ld}
Let $0<\epsilon\le 1$. It holds that for any $x\ge 0$,
\begin{align}
\int_0^{Lx}\big(1+|t+Ld+i\epsilon|^{\f12}\big)dt\gtrsim Lx|Ld|^{\f12}+|Lx|^{\f32},
\end{align}
where  $d=-1-\lambda-ik\nu$.
\end{lemma}

\begin{proof}
Notice that
\beno
1+|t+Ld+i\epsilon|^{\f12}=1+|t-L(\lambda+1)-i(\nu k^2)^{\f23}+i\epsilon|^{\f12}\gtrsim1+|t-L(\lambda+1)|^{\f12}+(\nu k^2)^{\f23}
\eeno
and $|Ld|\sim|L(\lambda+1)|+(\nu k^2)^{\f23}$. Thus, we have
\begin{align*}
\int_0^{Lx}\big(1+|t+Ld+i\epsilon|^{\f12}\big)dt\sim \int_0^{Lx}\big(1+(\nu k^2)^{\f23}+|t-L(\lambda+1)|^{\f12}\big)dt.
\end{align*}

For the case of $Lx\leq L(\lambda+1)$, we have
\begin{align*}
&\int_0^{Lx}(1+(\nu k^2)^{\f13}+|t-L(\lambda+1)|^{\f12})dt= \int_0^{Lx}(1+(\nu k^2)^{\f13}+(L(\lambda+1)-t)^{\f12})dt\\
&=Lx(1+(\nu k^2)^{\f13})+\f{2}{3}\big[(L(\lambda+1))^{\f32}-(L(\lambda+1)-Lx)^{\f32}\big]\\
&\gtrsim Lx((L(\lambda+1))^{\f12}+(\nu k^2)^{\f13})\sim Lx|Ld|^{\f12}+|Lx|^{\f32}.
\end{align*}
For the case of $Lx\geq L(\lambda+1)\geq0$, we have
\begin{align*}
& \int_0^{Lx}(1+(\nu k^2)^{\f13}+|t-L(\lambda+1)|^{\f12})dt\\
&= Lx(1+(\nu k^2)^{\f13})+\int_0^{L(\lambda+1)}(L(\lambda+1)-t)^{\f12}dt+\int_{L(\lambda+1)}^{Lx}(t-L(\lambda+1))^{\f12}dt\\
 &=Lx(1+(\nu k^2)^{\f13})+\f{2}{3}\big[(L(\lambda+1))^{\f32}+(Lx-L(\lambda+1))^{\f32}\big]\\
 &\gtrsim |Lx|^{\f32}+Lx(\nu k^2)^{\f13}\sim Lx|Ld|^{\f12}+|Lx|^{\f32}.
\end{align*}
For the case of $L(\lambda+1)\leq0$, we have
\begin{align*}
& \int_0^{Lx}(1+(\nu k^2)^{\f13}+|t-L(\lambda+1)|^{\f12})dt\sim(Lx(1+(\nu k^2)^{\f13}+|L(\lambda+1)|^{\f12})+\int_0^{Lx}t^{\f12}dt\\
&\geq Lx((L(\lambda+1))^{\f12}+(\nu k^2)^{\f13})+(2/3)|Lx|^{\f32}\sim Lx|Ld|^{\f12}+|Lx|^{\f32}.
\end{align*}

Summing up, we conclude the lemma.
\end{proof}

Now we are in a position to prove Lemma \ref{lem:W12}.

\begin{proof}
{\bf Step 1.} $L^\infty$ estimate

Thanks to the definition of $W_1$, we find that
\begin{align*}
&\f{L|W_1(y)|}{|A_0(Ld+i\epsilon)|}=\f{L|Ai(e^{i\f{\pi}{6}}(L(y+1+d)+i\epsilon))|}{|A_0(Ld+i\epsilon)|}.
\end{align*}
Thanks to  $|A'_0(z)|=|Ai(e^{i\f{\pi}{6}}z)|$, we get
\begin{align*}
&\f{L|W_1(x-1)|}{|A_0(Ld+i\epsilon)|}=
L\f{|A'_0(Lx+Ld+i\epsilon)|}{|A_0(Lx+Ld+i\epsilon)|}\f{|A_0(Lx+Ld+i\epsilon)|}{|A_0(Ld+i\epsilon)|}.
\end{align*}
By Lemma \ref{lem:Ao-more}, we have
\beno
\Big|\f{A'_0(Lx+Ld+i\epsilon)}{A_0(Lx+Ld+i\epsilon)}\Big|\lesssim 1+|Lx+Ld+i\epsilon|^{\f12}\lesssim1+|Lx+Ld|^{\f12}.
\eeno
from which and Lemma \ref{lem:A0-w}, we infer that for any $x\in[0,2]$,
\begin{align*}
\f{L|W_1(x-1)|}{|A_0(Ld+i\epsilon)|}\lesssim& L (1+|Lx+Ld|^{\f12})e^{-Lx/3}\\
\lesssim&L(1+|Ld|^{\f12})\lesssim L(1+|L(1+\lambda)|^{\f12})=L+|k(1+\lambda)/\nu|^{\f12}.
\end{align*}As $ L=(|k|/\nu)^{\f13}=\nu^{-\f12}(\nu k^2)^{\f16}\leq\nu^{-\f12},$ we have\begin{align*}
&\f{L}{|A_0(Ld+i\epsilon)|}\|W_1\|_{L^{\infty}}\leq C(L+|k(1+\lambda)/\nu|^{\f12})\leq C\nu^{-\f12}(1+ |k(\lambda+1)|)^{\f12}.
\end{align*}
The proof of $\f{L}{|A_0(L\widetilde{d}+i\epsilon)|}\|W_2\|_{L^{\infty}}$ is similar. \smallskip

{\bf Step 2.} $L^1$ estimate

Thanks to the definition of $W_1$, we have
\begin{align*}
\f{L}{|A_0(Ld+i\epsilon)|}\|W_1\|_{L^1}=&\f{L}{|A_0(Ld+i\epsilon)|}\int_{-1}^1|Ai(e^{i\f{\pi}{6}}(L(y-\lambda-i k\nu)+i\epsilon))|dy\\
=&\f{L}{|A_0(Ld+i\epsilon)|}\int_{0}^2|Ai(e^{i\f{\pi}{6}}(Lx+Ld+i\epsilon))|dx\\
=&L\int_{0}^{2}\f{|A'_0(Lx+Ld+i\epsilon)|}{|A_0(Lx+Ld+i\epsilon)|}\f{|A_0(Lx+Ld+i\epsilon)|}{|A_0(Ld+i\epsilon)|}dx.
\end{align*}
For the case of $|Ld|\leq1$, we get by Lemma \ref{lem:Ao-more} and Lemma \ref{lem:A0-w} that
\begin{align*}
\f{L}{|A_0(Ld+i\epsilon)|}\|W_1\|_{L^1}\lesssim& L \int_{0}^{2}(1+|Lx+Ld|^{\f12})e^{-|Lx|/3}dx\\
\lesssim&L\int_{0}^{2}(1+|Lx|^{\f12})e^{-|Lx|/3}dx\lesssim 1.
\end{align*}
For the case of $|Ld|\geq1$, by Lemma \ref{lem:Ao} and the proof of Lemma \ref{lem:A0-w}, and Lemma \ref{lem:Ld}, we infer that
\begin{align*}
\f{L}{|A_0(Ld+i\epsilon)|}\|W_1\|_{L^1}\lesssim& L\int_{0}^{2}(1+|Lx+Ld|^{\f12})e^{-c\int_0^{Lx}(1+|t+Ld|^{\f12})dt}dx\\
\lesssim& L\int_{0}^{2}(1+|Lx+Ld|^{\f12})e^{-cLx|Ld|^{\f12}}dx\lesssim 1,
\end{align*}
here we used
\begin{align*}
&L\int_{0}^{2}e^{-cLx|Ld|^{\f12}}dx=\int_{0}^{2}e^{-cLx|Ld|^{\f12}}d(Lx)\lesssim1,\\
&L\int_{0}^{2}|Lx|^{\f12}e^{-cLx|Ld|^{\f12}}dx=\int_{0}^{2}|Lx|^{\f12}e^{-cLx|Ld|^{\f12}}d(Lx)\lesssim 1,\\
& L\int_{0}^{2}|Ld|^{\f12}e^{-cLx|Ld|^{\f12}}dx=\int_{0}^{2}e^{-cLx|Ld|^{\f12}}d(Lx|Ld|^{\f12})\lesssim 1.
\end{align*}
The proof of $\f{L}{|A_0(L\widetilde{d}+i\epsilon)|}\|W_2\|_{L^1}\lesssim 1$ is similar.\smallskip

{\bf Step 3.} Weighted $L^2$ estimate

We have
\begin{align*}
\f{L}{|A_0(Ld+i\epsilon)|}\|\rho_kW_1\|_{L^2}=&\f{L}{|A_0(Ld+i\epsilon)|}\Big(\int_{-1}^1\rho_k(y)|Ai(e^{i\f{\pi}{6}}(L(y-\lambda-i k\nu)+i\epsilon))|^2dy\Big)^{\f12}\\
=&\f{L}{|A_0(Ld+i\epsilon)|}\Big(\int_{0}^2\rho_{k}(x-1)|Ai(e^{i\f{\pi}{6}}(L(x+d)+i\epsilon))|^2dx\Big)^{\f12}\\
=&L\Big(\int_{0}^{2}\rho_{k}(x-1)\f{|A'_0(Lx+Ld+i\epsilon)|^2}{|A_0(Lx+Ld+i\epsilon)|^2}\f{|A_0(Lx+Ld+i\epsilon)|^2}{|A_0(Ld+i\epsilon)|^2}dx\Big)^{\f12}.
\end{align*}
For the case of $|Ld|\leq1$, we have
\begin{align*}
\f{L}{|A_0(Ld+i\epsilon)|}\|\rho_kW_1\|_{L^2}&\lesssim L \Big(\int_{0}^{2}(1+|Lx+Ld|)e^{-|Lx|/3}dx\Big)^{\f12}\lesssim L^{\f12}.
\end{align*}
For the case of $|Ld|\geq1$, we have
\begin{align*}
\f{L}{|A_0(Ld+i\epsilon)|}\|\rho_kW_1\|_{L^2}\lesssim& L\Big(\int_{0}^{2}\rho_k(x-1)(1+|Lx+Ld|)e^{-2c'Lx|Ld|^{\f12}}dx\Big)^{\f12}\\
\lesssim& L\Big(\int_{0}^{2}(1+Lx+|Ld|\rho_k(x-1))e^{-2c'Lx|Ld|^{\f12}}dx\Big)^{\f12}\\
 \lesssim& L\Big(\int_{0}^{2}(1+Lx+Lx|Ld|)e^{-2c'Lx|Ld|^{\f12}}dx\Big)^{\f12}\lesssim L^{\f12}.
\end{align*}
The proof of $\f{L}{|A_0(L\widetilde{d}+i\epsilon)|}\|\rho_k^{\f12}W_2\|_{L^2}\lesssim L^{\f12}$ is similar.
\end{proof}

\subsection{Estimates of $C_{ij}$}

In this subsection, we prove Lemma \ref{lem:C-ij}. We need the following lemmas.

\begin{lemma}\label{lem:A0-w2}
Let $\delta_0$ be as in Lemma \ref{lem:Ao-more}. Then it holds that for ${\rm Im}z\le \delta_0$ and $x\ge 0$,
\begin{align*}
  |\omega(z,x)|\leq e^{-cx^{3/2}}.
\end{align*}
\end{lemma}
\begin{proof}
By Lemma \ref{lem:Ao-more}, we get

\begin{align*}
 |\omega(z,x)|\leq \Big|\exp\Big({\rm Re}\int_{0}^{x}\frac{A'_0(z+t)}{A_0(z+t)}dt\Big)\Big|
 \leq\exp\Big(-c\int_{0}^{x}(1+|z+t|^{\f12})dt\Big)\leq e^{-cx^{3/2}},
\end{align*}
where we used Lemma \ref{lem:Ld} so that
\begin{align*}
\int_{0}^{x}(1+|z+t|^{\f12})dt=\int_{0}^{Lx'}(1+|t+Ld+i\epsilon|^{\f12})dt\gtrsim |Lx'|^{\f32}=x^{\f32},
\end{align*}
by writing $z=Ld+i\epsilon, x=Lx'$.
\end{proof}

\begin{lemma}
\label{lem:A0-w3}
Let $\delta_1$ be as in Lemma \ref{lem:A0-w}. There exists $k_0>1$ so that if $L\geq 6k$ or $L\geq k\geq k_0$, then we have
 \begin{align*}
&e^{2k}\Big|1-e^{-2k}\omega(z,2L)- \frac{k}{L}\int_{0}^{2L}e^{-\frac{kt}{L}}\omega(z,t) dt\Big|\\
&\quad\geq\sqrt{2}\Big|1 -e^{2k}\omega(z,2L)+ \frac{k}{L}\int_{0}^{2L}e^{\frac{kt}{L}}\omega(z,t) dt\Big|\quad \text{for}\quad {\rm Im} z\le \delta_1.
 \end{align*}
\end{lemma}
\begin{proof}
We first consider the case of  $L\geq 6k,\ k\geq 1.$ It follows from Lemma \ref{lem:A0-w} that for ${\rm Im}z\le \delta_1$,
\begin{align*}
\Big|1-e^{-2k}\omega(z,2L)- \frac{k}{L}\int_{0}^{2L}e^{-\frac{kt}{L}}\omega(z,t)) dt\Big|\geq&1-e^{-2k}|\omega(z,2L)|- \frac{k}{L}\int_{0}^{2L}e^{-\frac{kt}{L}}|\omega(z,t)| dt\\
\geq&1-e^{-2k}e^{-2L/3}- \frac{k}{L}\int_{0}^{2L}e^{-\frac{kt}{L}-\frac{t}{3}} dt\\
\geq&1-e^{-2k}- \frac{k}{L}/(\frac{k}{L}+\frac{1}{3})\geq1-\f16- \frac{1}{6}/(\frac{1}{6}+\frac{1}{3})=\f12,
 \end{align*}
 and
 \begin{align*}
  \Big|1 -e^{2k}\omega(z,2L)+ \frac{k}{L}\int_{0}^{2L}e^{\frac{kt}{L}}\omega(z,t) dt\Big|\leq& 1+e^{2k}|\omega(z,2L)|+ \frac{k}{L}\int_{0}^{2L}e^{\frac{kt}{L}}|\omega(z,t)| dt\\ \leq&1+e^{2k}e^{-2L/3}+ \frac{k}{L}\int_{0}^{2L}e^{\frac{kt}{L}-\frac{t}{3}} dt\\
  \leq&1+e^{-2k}+ \frac{k}{L}/(\frac{1}{3}-\frac{k}{L}){\leq}1+\f16+ \frac{1}{6}/(\frac{1}{3}-\frac{1}{6})\leq \f73.
 \end{align*}
 This  shows that
 \begin{align*}
  \sqrt{2}\Big|1 -e^{2k}\omega(z,2L)+ \frac{k}{L}\int_{0}^{2L}e^{\frac{kt}{L}}\omega(z,t) dt\Big|&\leq \f{7\sqrt{2}}3< 10/3<7/2<e^2/2\leq e^{2k}/2\\ \leq& e^{2k}\Big|1-e^{-2k}\omega(z,2L)- \frac{k}{L}\int_{0}^{2L}e^{-\frac{kt}{L}}\omega(z,t) dt\Big|.
 \end{align*}

 For the case of $L\geq k\geq k_0$, on one hand, we have
 \begin{align*}
  \Big|1-e^{-2k}\omega(z,2L)- \frac{k}{L}\int_{0}^{2L}e^{-\frac{kt}{L}}\omega(z,t)) dt\Big|\geq&1-e^{-2k}e^{-2L/3}- \frac{k}{L}\int_{0}^{2L}e^{-\frac{kt}{L}-\frac{t}{3}} dt\\
  \geq&1-e^{-8k/3}- \frac{k}{L}/(\frac{k}{L}+\frac{1}{3})\geq \f18,
 \end{align*}
on the other hand, by Lemma \ref{lem:A0-w2}, we have
 \begin{align*}
 \Big|1 -e^{2k}\omega(z,2L)+ \frac{k}{L}\int_{0}^{2L}e^{\frac{kt}{L}}\omega(z,t) dt\Big|\leq& 1 +e^{2k}|\omega(z,2L)|+ \int_{0}^{2L}e^{{t}}|\omega(z,t)| dt\\ \leq&1+e^{2k}e^{-c(2L)^{3/2}}+ \int_{0}^{2L}e^{t-ct^{3/2}} dt\\\leq&1+e^{2k-c(2k)^{3/2}}+C \leq C_0,
 \end{align*}
 here $C_0>3$ is an absolute constant. Choose $k_0>1$ so that $ e^{2k_0}>8\sqrt{2}C_0$. Then we have
 \begin{align*}
   &e^{2k}\Big|1-e^{-2k}\omega(z,2L)- \frac{k}{L}\int_{0}^{2L}e^{-\frac{kt}{L}}\omega(z,t)) dt\Big|\geq e^{2k_0}/8\geq \sqrt{2}C_0\\ &\geq\sqrt{2}\Big|1 -e^{2k}\omega(z,2L)+ \frac{k}{L}\int_{0}^{2L}e^{\frac{kt}{L}}\omega(z,t) dt\Big|.
 \end{align*}

This completes the proof of the lemma .
\end{proof}

Now we are in a position to prove Lemma \ref{lem:C-ij}. Let us recall that
\begin{equation}\nonumber
\left( \begin{array}{cc}  C_{11} \\  C_{12}\\
 \end{array}\right)=\f{\left( \begin{array}{cc}  A_2e^k-B_2e^{-k} \\  -B_1e^k+A_1e^{-k}\\
 \end{array}\right)}{A_1A_2-B_1B_2},
 \quad
\left( \begin{array}{cc}  C_{21} \\  C_{22}\\
 \end{array}\right)=\f{\left( \begin{array}{cc}  -A_2e^{-k}+B_2e^{k} \\  B_1e^{-k}-A_1e^{k}\\
 \end{array}\right)}{A_1A_2-B_1B_2},
\end{equation}
where
\begin{align*}
&A_1=\int_{-1}^1e^{ky}W_1(y)dy,\quad A_2=\int_{-1}^1e^{-ky}W_2(y)dy,\\
&B_1=\int_{-1}^1e^{-ky}W_1(y)dy,\quad B_2=\int_{-1}^1e^{ky}W_2(y)dy.
\end{align*}

\begin{proof}

Let $y+1=x=\frac{t}{L}$. Due to $A'_0(z)=-e^{i\pi/6}Ai(e^{i\pi/6}z)$,
we have
\begin{align}
 \nonumber B_1  =&\int_{0}^{2}e^{-k(x-1)}Ai(e^{i\pi/6}(L(x+d)+i\epsilon))dx\\
  \nonumber =&\frac{e^k}{L} \int_{0}^{2L}e^{-\frac{kt}{L}}Ai(e^{i\pi/6}((t+Ld)+i\epsilon))dt=-\frac{e^{k-i\pi/6}}{L} \int_{0}^{2L}e^{-\frac{kt}{L}}A'_0(t+Ld+i\epsilon)dt\\
  \nonumber  =& -\frac{e^{k-i\pi/6}}{L} \Big[e^{-2k}A_0(2L+Ld+i\epsilon)- A_0(Ld+i\epsilon)+ \frac{k}{L}\int_{0}^{2L}e^{-\frac{kt}{L}}A_0(t+Ld+i\epsilon) dt\Big]\\
=& -A_0(Ld+i\epsilon)\frac{e^{k-i\pi/6}}{L} \Big[e^{-2k}\omega(Ld+i\epsilon,2L)- 1+ \frac{k}{L}\int_{0}^{2L}e^{-\frac{kt}{L}}\omega(Ld+i\epsilon,t) dt\Big].\label{eq:B2}
\end{align}
Similarly, we have
 \begin{align}\label{eq:B1}
 A_1 & = -A_0(Ld+i\epsilon)\frac{e^{-k-i\pi/6}}{L}\Big [e^{2k}\omega(Ld+i\epsilon,2L)- 1- \frac{k}{L}\int_{0}^{2L}e^{\frac{kt}{L}}\omega(Ld+i\epsilon,t)dt\Big].
 \end{align}
 Then we infer from Lemma \ref{lem:A0-w3} that
 \begin{align}\label{eq:A1-B1}
  \Big|\frac{A_1}{B_1}\Big| = e^{-2k}\Big|\frac{1 -e^{2k}\omega(Ld+i\epsilon,2L)+ \frac{k}{L}\int_{0}^{2L}e^{\frac{kt}{L}}\omega(Ld+i\epsilon,t) dt}{1-e^{-2k}\omega(Ld+i\epsilon,2L)- \frac{k}{L}\int_{0}^{2L}e^{-\frac{kt}{L}}\omega(Ld+i\epsilon,t) dt}\Big|\le \f {\sqrt{2}} 2.
   \end{align}

 Thanks to ${Ai}(z)=\overline{{A{i}}(\bar{z})}$, we have
 \begin{align}
 \nonumber  \overline{ B_2} &=\int_{-1}^{1}e^{ky}Ai(e^{-i5\pi/6} (L(y-\lambda+ik\nu)-i\epsilon))dy \\
 \nonumber&=\int_{-1}^{1}e^{-ky}Ai(e^{i\pi/6} (L(y+\lambda-ik\nu)+i\epsilon))dy\\
 \label{eq:B2}   &= -A_0(L\widetilde{d}+i\epsilon)\frac{e^{k-i\pi/6}}{L} \Big[e^{-2k}\omega(L\widetilde{d}+i\epsilon,2L)- 1+ \frac{k}{L}\int_{0}^{2L}e^{-\frac{kt}{L}}\omega(L\widetilde{ d}+i\epsilon,t) dt\Big].
 \end{align}
 Similarly, we have
 \begin{align}
 \label{eq:A2}  \overline{A_2}&=-A_0(L\widetilde{d}+i\epsilon)\frac{e^{-k-i\pi/6}}{L}\Big [e^{2k}\omega(L\widetilde{d}+i\epsilon,2L)- 1- \frac{k}{L}\int_{0}^{2L}e^{\frac{kt}{L}}\omega(L\widetilde{d}+i\epsilon,t)dt\Big].
 \end{align}
Thus, by Lemma \ref{lem:A0-w3}, we get
 \begin{align}\label{eq:A2-B2}
        \Big|\frac{A_2}{B_2}\Big|\le \f {\sqrt{2}} 2.
          \end{align}

Now it follows from \eqref{eq:A1-B1} and \eqref{eq:A2-B2} that
\beno
|A_1A_2-B_1B_2|\gtrsim |B_1B_2|.
\eeno
From the proof of Lemma \ref{lem:A0-w3} and \eqref{eq:B1}, we know that
\begin{align*}
|B_1|\geq&\f{e^k}{L}|A_0(Ld+i\epsilon)\Big[1-e^{-2k}|\omega(Ld+i\epsilon,2L)|- \frac{k}{L}\int_{0}^{2L}e^{-\frac{kt}{L}}|\omega(Ld+i\epsilon,t)| dt\Big]\\
\geq&\f18\f{e^k}{L}|A_0(Ld+i\epsilon)|,
\end{align*}
Similarly, $|B_2|\geq\f18\f{e^k}{L}|A_0(L\widetilde{d}+i\epsilon)|$. Thus,
\begin{align*}
|A_1A_2-B_1B_2|\gtrsim|A_0(Ld+i\epsilon)||A_0(L\widetilde{d}+i\epsilon)|\f{e^{2k}}{L^2}.
\end{align*}
Furthermore, we also have
\begin{align*}
&|B_1|\leq 2 \f{e^k}{L} |A_0(Ld+i\epsilon)|,\quad
 |B_2|\leq 2\f{e^k}{L}|A_0(L\widetilde{d}+i\epsilon)|,\\
&|A_1| \le C_0\f{e^{-k}}{L} |A_0(Ld+i\epsilon)|,\quad
  |A_2| \leq C_0\f{e^{-k}}{L}|A_0(L\widetilde{d}+i\epsilon)| .
\end{align*}

Summing up, we can conclude the estimates of $C_{ij}$.
\end{proof}

\section{Appendix}

\begin{lemma}\label{lem:hyperfct-1}
It holds that for any $|k|\geq1$,
\begin{align*}
&\int_{-1}^1\Big|\f{\sinh k(1+y)}{\sinh 2k}\Big|^2dy=\int_{-1}^1\Big|\f{\sinh k(1-y)}{\sinh 2k}\Big|^2dy\leq 2|k|^{-1},\\
&\int_{-1}^1\Big|\f{\cosh k(1+y)}{\sinh 2k}\Big|^2dy=\int_{-1}^1\Big|\f{\cosh k(1-y)}{\sinh 2k}\Big|^2dy\leq 8|k|^{-1}.
\end{align*}
\end{lemma}
\begin{proof}
Use the facts that
\beno
\f{\sinh k(1+y)}{\sinh 2k}\leq2e^{-|k|(1-y)},\quad \big|\f{\cosh k(1+y)}{\sinh 2k}\big|\leq4e^{-|k|(1-y)},
\eeno
we infer that
\begin{align*}
&\int_{-1}^1\Big|\f{\sinh k(1+y)}{\sinh 2k}\Big|^2dy\leq\int_{-1}^{1} 4e^{-2|k|(1-y)}dy\leq2|k|^{-1},\\
&\int_{-1}^1\Big|\f{\cosh k(1+y)}{\sinh 2k}\Big|^2dy\leq\int_{-1}^{1} 16e^{-2|k|(1-y)}dy\leq 8|k|^{-1}.
\end{align*}
\end{proof}

\begin{lemma}\label{lem:hperfct-2}
Let $|k|\geq1$. If $1-\lambda\geq|k|^{-1}$, then we have
\begin{align*}
&\Big\|\f{\sinh k(1+y)}{\sinh 2k}\Big\|_{L^{\infty}(E_1)}\leq 2e^{-|k|(1-\lambda)/2},\\
&\Big\|\f{\cosh k(1+y)}{\sinh 2k}\Big\|_{L^{\infty}(E_1)}\leq 4e^{-|k|(1-\lambda)/2},
\end{align*}
where $E_1=(-1,1)\cap\big(-\infty,(\lambda+1)/2\big)$.
If $\lambda+1\geq|k|^{-1}$, then we have
\begin{align*}
&\Big\|\f{\sinh k(1-y)}{\sinh 2k}\Big\|_{L^{\infty}(E_2)}\leq2 e^{-|k|(1+\lambda)/2},\\
& \Big\|\f{\cosh k(1-y)}{\sinh 2k}\Big\|_{L^{\infty}(E_1)}\leq4e^{-|k|(1+\lambda)/2},
\end{align*}
where $E_2=(-1,1)\cap\big((\lambda-1)/2,+\infty\big)$.
\end{lemma}

\begin{proof}
If $1-\lambda\geq|k|^{-1}$, then we have for $y\in E_1$,
\beno
&&\Big|\f{\sinh k(1+y)}{\sinh 2k}\Big|\leq2e^{-|k|(1-y)}\leq2e^{-|k|(1-\lambda)/2},\\
&&\Big|\f{\cosh k(1+y)}{\sinh 2k}\Big|\leq4e^{-|k|(1-y)}\leq 4e^{-|k|(1-\lambda)/2}.
\eeno
This gives the first inequality. The proof of the second inequality is similar.
\end{proof}

\begin{lemma}\label{lem:phi-w}
If $(\partial_y^2-k^2)\varphi=w$, $\varphi(\pm1)=0$, $|k|\geq 1$, then we have
\begin{align*}
&\|\varphi'\|_{L^2}^2+k^2\|\varphi\|_{L^2}^2=\langle-w,\varphi\rangle\lesssim|k|^{-1}\|w\|_{L^1}^2,\\&
\|\varphi'\|_{L^{\infty}}+|k|\|\varphi\|_{L^{\infty}}\lesssim\|w\|_{L^1},\\&
\|\varphi'\|_{L^{\infty}}+|k|\|\varphi\|_{L^{\infty}}\lesssim |k|^{-\f12}\|w\|_{L^2}.
\end{align*}
\end{lemma}
\begin{proof}
The first inequality follows from the following
\begin{align*}
\|\varphi'\|_{L^2}^2+k^2\|\varphi\|_{L^2}^2&=\langle-w,\varphi\rangle\leq\|w\|_{L^1}\|\varphi\|_{L^{\infty}}\lesssim\|w\|_{L^1}\|\varphi'\|^{\f12}_{L^2}\|\varphi\|^{\f12}_{L^2}\\
&\lesssim\|w\|_{L^1}\big(|k|^{-1}\|\varphi'\|_{L^2}^2+|k|\|\varphi\|_{L^2}^2\big)^{\f12}.
\end{align*}

Using the first inequality, we infer that
\begin{align*}
|k|\|\varphi\|_{L^{\infty}}\lesssim|k|\|\varphi'\|^{\f12}_{L^2}\|\varphi\|^{\f12}_{L^2}
&\lesssim \big(|k|(\|\varphi'\|_{L^2}^2+k^2\|\varphi\|_{L^2}^2)\big)^{\f12}\lesssim\|w\|_{L^1}.
\end{align*}
For $y\in[0,1],$  we choose $y_1\in(y-1/k,y)$ so that $|\varphi'(y_1)|^2\leq |k|\|\varphi'\|_{L^2}^2$. Then we have
\begin{align*}
|\varphi'(y)|\leq&|\varphi'(y_1)|+\int_{y_1}^y|\varphi''(z)|dz\\
\leq&(|k|\|\varphi'\|_{L^2}^2)^{\f12}+\int_{y_1}^y|k^2\varphi(z)+w(z)|dz\\ \leq&C\|w\|_{L^1}+|y-y_1|k^2\|\varphi\|_{L^{\infty}}+\|w\|_{L^1}\\
\leq& C\|w\|_{L^1}+|k|\|\varphi\|_{L^{\infty}}\leq C\|w\|_{L^1}.
\end{align*}
Similarly, $|\varphi'(y)|\leq C\|w\|_{L^1} $ for $y\in[-1,0]$. This proves the second inequality.\smallskip

Thanks to $\|w\|_{L^2}^2=\|(\partial_y^2-k^2)\varphi\|_{L^2}^2=\|\varphi''\|_{L^2}^2+2k^2\|\varphi'\|_{L^2}^2+k^4\|\varphi\|_{L^2}^2$, we have
\begin{align*}
\|\varphi'\|_{L^{\infty}}+|k|\|\varphi\|_{L^{\infty}}\lesssim&\|\varphi''\|^{\f12}_{L^2}\|\varphi'\|^{\f12}_{L^2}+\|\varphi'\|_{L^2}
|k|\|\varphi'\|^{\f12}_{L^2}\|\varphi\|^{\f12}_{L^2}\\ \leq&|k|^{-\f12}\|\varphi''\|_{L^2}+|k|^{\f12}\|\varphi'\|_{L^2}+|k|^{\f32}\|\varphi\|_{L^2}\lesssim |k|^{-\f12}\|w\|_{L^2},
\end{align*}
which gives the third inequality.
\end{proof}

\section*{Acknowledgments}

Z. Zhang is partially supported by NSF of China under Grant 11425103.

\end{CJK*}
\end{document}